\documentclass[11pt]{amsart}
\usepackage[letterpaper,margin=1in]{geometry} 
\usepackage[colorlinks,
             linkcolor=blue,
             citecolor=black!75!red,
             pdfproducer={pdfLaTeX},
             pdfpagemode=None,
             bookmarksopen=true
             bookmarksnumbered=true]{hyperref}
\usepackage[matrix,arrow,ps,color,line,curve,frame,all]{xy}

\usepackage{tikz}
\usetikzlibrary{arrows}

\title{Exotic Elliptic Algebras}

\author{Alex Chirvasitu}
\author{S. Paul Smith}

\address{Department of Mathematics, Box 354350, University of  Washington, Seattle, WA 98195,USA.}

\email{chirva@math.washington.edu, smith@math.washington.edu}



\usepackage{amsmath,amsthm,stmaryrd}
\usepackage{cleveref}

\newtheorem{lemma}{Lemma}[section]
\newtheorem{theorem}[lemma]{Theorem}
\newtheorem{proposition}[lemma]{Proposition}
\newtheorem{corollary}[lemma]{Corollary}

\theoremstyle{definition} 

\newtheorem{definitionnodiamond}[lemma]{Definition}
\newtheorem{examplenodiamond}[lemma]{Example}
\newtheorem{remarknodiamond}[lemma]{Remark}

\newenvironment{definition}{\begin{definitionnodiamond}}{\hfill\ensuremath\blacklozenge\end{definitionnodiamond}}

\newcounter{stepofproof}

\numberwithin{equation}{section}

\crefname{section}{Section}{Sections}
\crefformat{section}{#2Section~#1#3} 
\Crefformat{section}{#2Section~#1#3} 

\crefname{subsection}{}{Subsections}
\crefformat{subsection}{\S#2#1#3} 
\Crefformat{subsection}{\S#2#1#3}

\crefname{definition}{Definition}{Definitions}
\crefformat{definition}{#2Definition~#1#3} 
\Crefformat{definition}{#2Definition~#1#3} 

\crefname{example}{Example}{Examples}
\crefformat{example}{#2Example~#1#3} 
\Crefformat{example}{#2Example~#1#3} 

\crefname{examplenodiamond}{Example}{Examples}
\crefformat{examplenodiamond}{#2Example~#1#3} 
\Crefformat{examplenodiamond}{#2Example~#1#3} 

\crefname{remark}{Remark}{Remarks}
\crefformat{remark}{#2Remark~#1#3} 
\Crefformat{remark}{#2Remark~#1#3} 

\crefname{remarknodiamond}{Remark}{Remarks}
\crefformat{remarknodiamond}{#2Remark~#1#3} 
\Crefformat{remarknodiamond}{#2Remark~#1#3} 

\crefname{convention}{Convention}{Conventions}
\crefformat{convention}{#2Convention~#1#3} 
\Crefformat{convention}{#2Convention~#1#3} 

\crefname{lemma}{Lemma}{Lemmas}
\crefformat{lemma}{#2Lemma~#1#3} 
\Crefformat{lemma}{#2Lemma~#1#3} 

\crefname{proposition}{Proposition}{Propositions}
\crefformat{proposition}{#2Proposition~#1#3} 
\Crefformat{proposition}{#2Proposition~#1#3} 

\crefname{corollary}{Corollary}{Corollaries}
\crefformat{corollary}{#2Corollary~#1#3} 
\Crefformat{corollary}{#2Corollary~#1#3} 

\crefname{theorem}{Theorem}{Theorems}
\crefformat{theorem}{#2Theorem~#1#3} 
\Crefformat{theorem}{#2Theorem~#1#3} 

\crefname{assumption}{Assumption}{Assumptions}
\crefformat{assumption}{#2Assumption~#1#3} 
\Crefformat{assumption}{#2Assumption~#1#3} 

\crefname{equation}{}{}
\crefformat{equation}{(#2#1#3)} 
\Crefformat{equation}{(#2#1#3)}

\crefname{proofstep}{Step}{Steps}
\crefformat{proofstep}{#2Step~#1#3} 
\Crefformat{proofstep}{#2Step~#1#3}

\newcommand\cat[1]{\textsc{#1}}
\newcommand\define[1]{\emph{#1}}

\newcommand\arXiv[1]{\href{http://arxiv.org/abs/#1}{\nolinkurl{arXiv:#1}}}
\newcommand\MRnumber[1]{\href{http://www.ams.org/mathscinet-getitem?mr=#1}{\nolinkurl{MR#1}}}
\newcommand\DOI[1]{\href{http://dx.doi.org/#1}{\nolinkurl{DOI:#1}}}
\newcommand\MAILTO[1]{\href{mailto:#1}{\nolinkurl{#1}}}

\usepackage{amsfonts,amssymb}

\newcommand\bN{\mathbb N}

\newcommand\bP{\mathbb P}

\newcommand\bZ{\mathbb Z}

\newcommand\cB{\mathcal B}
\newcommand\cC{\mathcal C}
\newcommand\cD{\mathcal D}
\newcommand\cE{\mathcal E}
\newcommand\cF{\mathcal F}
\newcommand\cG{\mathcal G}

\newcommand\cL{\mathcal L}
\newcommand\cM{\mathcal M}

\newcommand\cO{\mathcal O}
\newcommand\cP{\mathcal P}

\newcommand\cV{\mathcal V}

\newcommand\Vect{\cat{Vect}}

\DeclareMathOperator\Qcoh{\cat{Qcoh}}
\DeclareMathOperator\QCoh{\cat{Qcoh}}
\newcommand\op{\mathrm{op}}
\newcommand\id{\mathrm{id}}

\renewcommand\lim{\varprojlim}

\def\CC{{\mathbb C}}

\def\FF{{\mathbb F}} 
\def\GG{{\mathbb G}}

\def\NN{{\mathbb N}}

\def\PP{{\mathbb P}}

\def\ZZ{{\mathbb Z}}

\def\bfu{{\bf u}}
\def\bfv{{\bf v}}
 
\def\Ext{\operatorname {Ext}} 
\def\GL{\operatorname {GL}}
 
\def\pr{{\operatorname {pr}}}
 
\def\th{\operatorname {th}}    

\def\Aut{\operatorname{Aut}}

\def\Div{\operatorname{Div}} 
 
\def\Fdim{{\sf Fdim}} 
\newcommand{\GKdim}{\mathrm{GKdim}}
\def\gldim{\operatorname{gldim}} 
\def\Gr{{\sf Gr}} 
\def\Hom{\operatorname{Hom}}

\def\Mod{{\sf Mod}}

\def\Projnc{\operatorname{Proj}_{nc}}
 
\def\QGr{\operatorname{\sf QGr}}
 
\def\rank{\operatorname{rank}}

\def\uExt{\operatorname{\underline{Ext}}}
 
\def\uHom{\operatorname{\underline{Hom}}}
 
\def\a{\alpha}
\def\b{\beta}
\def\c{\gamma}
\def\d{\delta}

\def\l{\lambda}

\def\ve{\varepsilon}

\def\G{\Gamma}
\def\L{\Lambda}

\def\Ups{\Upsilon}

\def\fP{{\mathfrak P}}
 
\def\sC{{\sf C}}

\def\sM{{\sf M}}

\def\sT{{\sf T}}

\def\sfy{{\sf y}}
\def\wtA{{\widetilde{ A}}}
\def\wtB{{\widetilde{ B}}}
\def\wtM{{\widetilde{ M}}}
\def\wtN{{\widetilde{ N}}}
\def\wtP{{\widetilde{ P}}} 
\def\wtQ{{\widetilde{ Q}}}
 
\def\QCoh{{\sf Qcoh}}
\def\qcoh{{\sf Qcoh}}
 
\def\hdot{{\:\raisebox{2pt}{\text{\circle*{1.5}}}}}

\begin{document}

\maketitle

\begin{abstract}
The 4-dimensional Sklyanin algebras, over $\CC$, 
  $A(E,\tau)$, are constructed from an elliptic curve $E$ and a translation automorphism $\tau$ of $E$. 
The Klein vierergruppe $\G$ acts as graded algebra automorphisms of $A(E,\tau)$. There is also an action of $\G$ 
as automorphisms of the matrix algebra $M_2(\CC)$ making it isomorphic to the regular representation. The main object of study
in this paper 
is the algebra $\wtA:=\big(A(E,\tau) \otimes M_2(\CC)\big)^\G$. Like $A(E,\tau)$, 
$\wtA$ is noetherian, generated by 4 elements modulo six quadratic relations, Koszul, Artin-Schelter regular of global dimension 4,
has the same Hilbert series as the polynomial ring on 4 variables, satisfies the $\chi$ condition, and so on.
These results are special cases of general results proved for a triple $(A,T,H)$ consisting of a Hopf algebra $H$,
a (often graded) $H$-comodule algebra $A$, and an $H$-torsor $T$. Those general results involve transferring properties between
$A$, $A \otimes T$, and $(A \otimes T)^{\rm{co} H}$.  We then investigate $\wtA$ from the point of view of non-commutative
projective geometry. We examine its point modules, line modules, and 
a certain quotient $\wtB:=\wtA/(\Theta,\Theta')$ where $\Theta$ and $\Theta'$
are homogeneous central elements of degree two. In doing this we show that $\wtA$ differs from $A$ in interesting ways.
For example, the point modules for $A$ are parametrized by $E$ and 4 more points whereas
 $\wtA$ has exactly 20 point modules. Although $\wtB$ is not a twisted homogeneous coordinate ring in the sense of 
 Artin and Van den Bergh a certain quotient of the category of graded $\wtB$-modules is equivalent to the category of quasi-coherent 
 sheaves on the curve $E/E[2]$ where $E[2]$ is the 2-torsion subgroup. We construct line modules for $\wtA$ that are parametrized by
 the disjoint union $(E/\langle \xi_1\rangle) \sqcup (E/\langle \xi_2\rangle) \sqcup (E/\langle \xi_3\rangle)$ of the quotients of $E$ by
 its three subgroups of order 2.
\end{abstract}


\tableofcontents

\section{Introduction}

\subsection{}
The 3- and 4-dimensional Sklyanin algebras are among the most interesting algebras that have 
appeared in non-commutative algebraic geometry. Such an algebra determines and is determined by an elliptic curve, $E$,
a translation automorphism, $\tau$, of $E$, and an invertible $\cO_E$-module $\cL$ of degree 3, and 4, respectively. 
The representation theory of the Sklyanin algebra $A(E,\tau,\cL)$ and, what is almost the same thing, the geometric features of 
the non-commutative projective space $\Projnc\big(A(E,\tau,\cL)\big)$, is governed by the geometry of 
$E$ and $\tau$ when $E$ is embedded as a cubic or quartic curve in $\PP\big(H^0(E,\cL)^*\big)$. 
We refer the reader to \cite{A90} and  \cite{S94} for overviews of the 3- and 4-dimensional Sklyanin algebras.  
The $n$  in ``$n$-dimensional''  refers to the Gelfand-Kirillov dimension of 
$A(E,\tau)$, or its global dimension, or the dimension of $A(E,\tau,\cL)_1$ which is equal to $H^0(E,\cL)$.

Odesskii and Feigin have defined generalizations of the 4-dimensional Sklyanin algebras in \cite{OF89}, \cite{OF93}, and \cite{OF97}. 
The algebras they construct depend on a pair $(E,\tau)$, as before, but now a higher degree line bundle is used to construct $A(E,\tau,\cL)$.
In particular, when $\deg(\cL)=n^2$, $n \ge 2$, Odesskii and Feigin construct an algebra that they denote by $Q_{n^2}(E,\tau)$. 

Following an idea of Odesskii in \cite{O}, described in \Cref{sse.Q-tilde} below, we construct for every such pair $(E,\tau)$ and integer $n\ge 2$ a connected graded algebra $\widetilde{Q}=\widetilde{Q}_{n^2}(E,\tau)$ by a kind of Galois descent procedure applied to $Q_{n^2}(E,\tau)$. We show that the algebras obtained in this manner inherit many of the good properties enjoyed by $Q_{n^2}(E,\tau)$. For example, they are Artin-Schelter regular. 

\subsection{}

This paper examines the case $n=2$ and shows that the algebras $\wtQ$ exhibit a range of novel features. 
They are still governed very strongly by 
the geometry of $E$ and $\tau$. For this reason we call them ``elliptic algebras'', the name Odesskii and Feigin adopted for their algebras,
and we append the adjective ``exotic'' to indicate that they are somewhat novel when compared to the familiar 4-dimensional Sklyanin algebras and other  4-dimensional Artin-Schelter regular algebras.

\subsection{}
\label{ss.Hopf.torsor}
The procedure we use to construct the algebras $\wtQ$ is quite general. Let $H$ be a finite dimensional
Hopf algebra over a field $k$ and $A$ an $H$-comodule algebra. One might also require $A$ to be a graded algebra and that every
homogeneous component be a subcomodule. Let $T$ be an $H$-torsor (see \S\ref{se.torsor}) 
and define the algebra $A':=A \otimes T$. If $A$ 
is graded one places $T$ in degree zero to make $A'$ a graded algebra. Let $\wtA$ denote the subalgebra of $A'$
consisting of the $H$-coinvariant elements. In \S3 and \S4 we show how various properties pass back and forth
between $A$, $A'$, and $\wtA$. For example, we consider the noetherian property, that of being finite as a module over its center, and
numerous homological properties that play an important role in non-commutative algebraic geometry.  
When $H$ is commutative, which is the case in the definition of $\wtQ$,  $A'$ is an $H$-comodule algebra.

In \S4 we assume that $\dim_k(H)<\infty$,  and (usually) $A$ is a connected graded $H$-comodule algebra. 
We show $A$ is Koszul ($m$-Koszul) if and only if $\wtA$ is. We show $A$ is Artin-Schelter regular of dimension $d$ if and only if $\wtA$ is. We show $\wtA$ satisfies the $\chi$ condition, introduced in \cite{AZ94},  if $A$ does.

\subsection{}
The construction $A \rightsquigarrow \wtA$, and our results about properties shared by $A$ and $\wtA$, should be useful in other situations. It would be sensible to examine the effect of this construction on 2- and 3-dimensional Artin-Schelter regular algebras now that J.J. Zhang and his 
co-authors have determined (many/all?) the finite dimensional Hopf algebras for which such algebras can be comodule algebras. 
Even the case when $A$ is a polynomial ring, or an enveloping algebra, deserves investigation. 

\subsection{}
\label{sse.Q-tilde}
Let $Q=A(E,\tau,\cL)$ be a 4-dimensional Sklyanin algebra. It was shown in \cite{SSJ93} that $\G=(\ZZ/2) \times (\ZZ/2)$ acts as
 graded algebra automorphisms of $Q$  when $k=\CC$. The action there is induced by the translation
action of the 2-torsion subgroup, $E[2]$, on $E$. Here, working over an arbitrary algebraically closed field $k$ of characteristic $\ne 2$,
we define an action of $\G$ as graded $k$-algebra automorphisms of $Q$ and show that this ``corresponds'' to the translation action
of $E[2]$ on $E$. 

In the language of \S\ref{ss.Hopf.torsor}, we take $H$ to be the Hopf algebra of $k$-valued functions on $\G$ and $T$ to be $M_2(k)$, 
the ring of $2 \times 2$ matrices, with an appropriate $H$-comodule algebra structure. We then have $\wtQ=(Q \otimes T)^{\rm{co} H}
= (Q \otimes T)^\G$. The results in \S3 and \S4 show that $\wtQ$ has ``all'' the good properties $Q$ has.
It is a noetherian domain, has global dimension 4, has the same Hilbert series as the polynomial ring on 4 indeterminates, is Artin-Schelter regular, satisfies the $\chi$ condition, etc.

\subsection{}
Among the most important results about Sklyanin algebras are classifications of their point and line modules. 
The point modules of a 3-dimensional Sklyanin algebra are naturally parametrized by $E$ or, more informatively, 
by a natural copy of $E$ embedded as a smooth cubic curve in $\PP^2=\PP(Q_1^*)$. The point modules for a 4-dimensional Sklyanin 
are parametrized by a natural copy of $E$ as a smooth quartic curve in $\PP^3=\PP(Q_1^*)$ and 4 additional points, those being the 
vertices of the 4 singular quadrics that contain the copy of $E$. The line modules are, in both cases, parametrized by the secant lines
to $E$, the lines in $\PP(Q_1^*)$ that meet $E$ with multiplicity $\ge 2$.

The results for $\wtQ$ are very different. For example, $\wtQ$ has only 20 point modules.
In a note circulated in 1988 \cite{VdB88}, Van den Bergh showed that a generic 4-dimensional AS-regular algebra (with some other
properties) has exactly 20 point modules. Since then, there have been a number of examples showing that particular algebras,
rather than the ephemeral ``generic algebras'', have exactly 20 point modules. We believe that ours are the first such examples that
turn up ``in vivo'', so to speak. 

\subsection{}
 Van den Bergh and Tate \cite{TvdB96} showed that the Odesskii-Feigin algebras $Q_{n^2}$ are 
 noetherian, Koszul, Artin-Schelter  regular algebras of dimension $n^2$ with Hilbert series $(1-t)^{-n^2}$. It follows from the relations for 
 $Q_{n^2}$ that $\G=(\ZZ/n) \times (\ZZ/n)$, realized as the $n$-torsion subgroup $E[n] \subset E$, acts as graded algebra
 automorphisms of $Q_{n^2}$. It is an easy matter to see that the ring of $n \times n$ matrices $M_n(\CC)$ is an $H$-torsor where $H$ is the 
 Hopf algebra of $k$-valued functions on $\G$. 
In \S\ref{se.exotic} we show that for all $n \ge 2$, $\widetilde{Q_{n^2}}=\big(Q_{n^2} \otimes M_n(k)\big)^\G$ has ``the same'' properties as $Q_{n^2}$.

\subsection{}
In \S\ref{se.Q-tilde} we begin a detailed examination of the algebra $\wtQ$ in \S\ref{sse.Q-tilde}. 
We give explicit generators and relations for $\wtQ$. It has 4 generators and 6 quadratic relations (\Cref{prop.Q-tilde}). Since $\G=(\ZZ/2)\times (\ZZ/2)$ acts on $Q_1$ it acts as automorphisms of $\PP(Q_1)^*=\PP^3$. This $\PP^3$ contains a natural copy of $E$ embedded as a quartic curve and $\G$ restricts to an action as automorphisms of $E$. 

In \S\ref{se.translation} we show that this action is the same as the translation action of the 2-torsion subgroup $E[2]$. 
Each $\gamma \in \G$ acts as an auto-equivalence $M \rightsquigarrow \gamma^*M$ of the graded-module category $\Gr(Q)$. 
Because $\G$ acts as $E[2]$ does, if $M_p$, $p \in E$, is the point module corresponding to $p \in E$, then $\gamma^*M_p \cong M_{p+\omega}$ for a suitable $\omega \in E[2]$. There is a similar result for line modules: $\c^*M_{p,q} \cong M_{p+\omega,q+\omega}$. 

\subsection{}
By \cite{SS92}, there is a regular sequence in $Q$ consisting of two homogeneous central elements of degree 2, 
$\Omega$ and $\Omega'$ say, such that $Q/(\Omega,\Omega')$ is a twisted homogeneous coordinate ring, $B(E,\tau,\cL)$, 
in the sense of Artin and Van den Bergh \cite{AV90}. The main result in \cite{AV90} tells us that the quotient category $\QGr(B(E,\tau,\cL))$
is equivalent to $\qcoh(E)$, the category of quasi-coherent sheaves on $E$. 

The algebra $\wtQ$ also has a  regular sequence consisting of two homogeneous central elements of degree 2, $\Theta$ and $\Theta'$ say.
Although $\wtB:=\wtQ/( \Theta,\Theta')$ is not a twisted homogeneous coordinate ring, Theorem \ref{th.wtB_Azumaya} proves that $\QGr(\wtB)$ is equivalent to $\QCoh(E/E[2])$.\footnote{Although $E/E[2]$  is isomorphic to $E$ it is
``better'' to think of  $\QGr(\wtB)$ as equivalent to $\QCoh(E/E[2])$.}
Nevertheless, $\wtB$ has no point modules. The points on $E/E[2]$ correspond to fat point modules of multiplicity 2 over $\wtB$. Another new feature is that $\wtB$ is not a domain although $B$ is. Nevertheless, $\wtB$ is a prime ring.

\subsection{}
In \S\ref{sect.pt.modules} we prove that $\wtQ$ has exactly 20 point modules. These modules correspond to 20
 points in $\PP^3=\PP(\wtQ_1^*)$  that we determine explicitly. The ``meaning'' of these 20 points eludes us. Let $\fP$ denote that 
 set of 20 points. The degree shift functor $M \rightsquigarrow M(1)$ induces a permutation $\theta:\fP\to \fP$ of order 2. 
 Shelton and Vancliff \cite{ShV02} have shown that the data $(\fP,\theta)$
 determines $\wtQ$ in the sense that the subspace $R \subseteq Q_1 \otimes Q_1$ of bihomogeneous forms vanishing on  the graph of $\theta$
 has the property that $\wtQ$ is isomorphic to $T(Q_1)/(R)$, the tensor algebra on $Q_1$ modulo the ideal generated by $R$.
 
 In \S\ref{se.line_modules}, we exhibit three families of line modules for $\wtQ$ parametrized by $\big( E/\langle \xi \rangle\big) \sqcup \big( E/\langle \xi' \rangle\big) \sqcup \big( E/\langle \xi'' \rangle\big)$ where $\{\xi,\xi',\xi''\}$ is the set of 2-torsion points on $E$. These are {\it not} all the line modules for $\wtQ$. 
 
 \subsection{}
In \S\S \ref{sec.wtB} and \ref{se.line_modules},  we examine $\G$-equivariant objects in $\Gr(Q)$ and other categories of interest. 
So as not to interrupt the flow of the paper we collect some basic facts about group actions on categories and equivariant objects in an Appendix. The material there is known in one form or another, and in various degrees of generality but we have not found a suitable reference. The reader might find the appendix useful in filling in the details of some of the proofs in \S10.

\subsection{}
In late January 2015, after proving most of the results in this paper, we found an announcement on the web of a seminar talk
by Andrew Davies at the University of Manchester in January 2014 that appeared to contain some of the results we prove here.
On 1/20/2015, we found a copy of his Ph.D. thesis (\cite{D}, \cite{Davies-arXiv})
which has substantial overlap with this paper. Davies also proves several things we don't.  
For example, he describes $\wtB$ (when $\tau$ has infinite order) in the manner of Artin and Stafford  \cite{AS00}.  
Nevertheless, most of what we do is more general, and most of our arguments differ from his. 
For example, when we deal with the 4-dimensional Sklyanin algebras we make no assumption on the order of $\tau$, 
we do not restrict our base field to the complex numbers, and we describe some of the line modules for $\wtQ$. 
Also, the results in \S3 and \S4 for arbitrary $H$ and $T$ are proved by Davies only in the case $H$ is
the ring of $k$-valued functions on a finite abelian group.

\subsection*{Acknowledgement}
We are very grateful to Kenneth Chan for numerous useful conversations while working on this paper and in particular for providing some of the insight on Azumaya algebras and related topics necessary in \Cref{sec.wtB}. We thank Pablo Zadunaisky and Michaela Vancliff for
pointing out errors in an earlier version of this paper.

 \section{Preliminaries}
 \label{se.prelims}
 
In \Cref{se.prelims,se.torsors,se.homological}, we work over an arbitrary field $k$. 
Once we begin discussing the 4-dimensional Sklyanin algebras $k$
will be an algebraically closed field of characteristic $\ne 2$.

\subsection{}

We will use what is now standard terminology and notation for graded rings and non-commutative projective  algebraic geometry.
There are several sources for unexplained terminology:
the Artin-Tate-Van den Bergh papers (\cite{ATV1}, \cite{ATV2}) that started the subject of non-commutative projective  algebraic geometry; Stafford and Van den Bergh's survey \cite{StVdB01}; papers by Stafford and Smith \cite{SS92} and 
Levasseur and  Smith \cite{LS93} on 4-dimensional Sklyanin algebras; the survey \cite{S94} on 4-dimensional Sklyanin algebras;
Artin and Van den Bergh's paper on twisted homogeneous coordinate rings \cite{AV90}; Artin and Zhang's on 
non-commutative projective schemes \cite{AZ94}. 

Suppose $A$ is an $\NN$-graded $k$-algebra such that 
$\dim_k(A_i)<\infty$ for all $i$. The category of $\ZZ$-graded left $A$-modules with degree-preserving $A$-module
homomorphisms is denoted by $\Gr(A)$. The full subcategory of $\Gr(A)$ consisting of modules that are the sum of their
finite dimensional submodules is denoted by $\Fdim(A)$. This is a Serre subcategory so we can form the quotient category 
$$
\QGr(A) \; :=\; \frac{\Gr(A)}{\Fdim(A)}.
$$
In fact, $\Fdim(A)$ is a localizing subcategory so the quotient functor $\pi^*:\Gr(A) \to \QGr(A)$ has a right adjoint $\pi_*$. 
The functor $\pi^*$ is exact. By definition, $\QGr(A)$ has the same objects as $\Gr(A)$. Since $\pi_*\pi^*$ is 
isomorphic to the identity functor we may view objects in $\QGr(A)$ as objects in $\Gr(A)$. 

\subsection{}
We write $\Vect$ for the category of vector spaces over $k$. 
 
\subsection{}\label{subsec.hopf}

Throughout this paper, $H$ is a Hopf algebra over $k$ with bijective antipode. 
We write  ${}^H\cM$  for the category of left $H$-comodules and  $\cM^H$ for the category of right $H$-comodules. 
Furthermore $A$ denotes a right $H$-comodule-algebra, i.e., an algebra object in $\cM^H$.

Let $\Ups$ be an abelian group. We call $A$ an {\sf $\Ups$-graded $H$-comodule algebra} or an {\sf $\Ups$-graded algebra} in $\cM^H$ 
if it is an $H$-comodule algebra such that each homogeneous component, $A_i$, is an $H$-subcomodule. 
For example, if $V$ is a right $H$-comodule and $R\subseteq V\otimes V$ an $H$-subcomodule, then the tensor algebra, $TV$, and its quotient $TV/(R)$, are $\ZZ$-graded algebras in $\cM^H$. 
 
We write $\Mod(R)$ for the the category of left modules over a ring $R$.
We write ${_A}\cM^H$ for the category of $A$-modules internal to the category of $H$-comodules, i.e., vector spaces $V$ equipped with an $A$-module structure and an $H$-comodule structure such that $A\otimes V\to V$ is an $H$-comodule map.    
If $A$ is an $\Ups$-graded algebra in $\cM^H$ we write ${}_{\Gr(A)}\cM^H$ for the category of $\Ups$-graded 
$A$-modules internal to  $\cM^H$, i.e. each homogeneous component $M_i$ is an $H$-comodule. Similar conventions apply to right $A$-modules, with the algebra subscripts appearing on the right in that case.

\section{Torsors, twisting, and descent}
\label{se.torsors}
 
In this section we prove some general results on the inheritance of various properties for certain rings of (co)invariants, relating various good properties of $A$ to those of the algebra $\wtA$ defined in \Cref{defn.A.tilde} below. In \S\S3.1-3.3, the only assumption on $H$ is that it is a Hopf algebra with bijective antipode. In \S3.4
we add the hypothesis that $H$ is commutative.

\subsection{Torsors} 
\label{se.torsor}

A {\sf left $H$-torsor} (or just torsor for short) is a left $H$-comodule-algebra $T$ such that 
\begin{enumerate}
  \item 
  $T \cong H$ in ${}^H\cM$, 
  \item 
   the ring of coinvariants,  ${}^{\rm{co} H}T$, is $k$, and
  \item 
 the linear map 
\begin{equation}\label{eq.torsor}
\begin{tikzpicture}[auto,baseline=1pt]
  \path[anchor=base] (0,0) node (1) {$T\otimes T$} +(4,0) node (2) {$H\otimes T\otimes T$} +(8,0) node (3) {$H\otimes T$};
         \draw[->] (1) to node[] {$\rho\otimes \id$} (2);
         \draw[->] (2) to node[] {$\id\otimes m$} (3);
\end{tikzpicture}
\end{equation}
is bijective where $\rho:T\to H\otimes T$ is the comodule structure and $m:T\otimes T\to T$ is multiplication. 
\end{enumerate}

Throughout \Cref{se.torsors}, $T$ denotes a left $H$-torsor.

\subsubsection{}
\label{rem.torsor}
A comodule algebra for which the composition in \Cref{eq.torsor} is an isomorphism is sometimes called a left $H$-Galois object (see e.g. \cite[Defn. 1.1]{Bic10}).   Loc. cit. and the references therein are good sources for background on torsors. 
Left $H$-torsors  classify exact monoidal functors $\cM^H\to\Vect$, the functor corresponding to $T$ being 
\begin{equation}
\label{mon.func}
 M \; \mapsto \; M\square_HT \; :=\;  \{x\in M\otimes T\ |\ (\rho_M\otimes\id)(x) = (\id\otimes\rho)(x)\},
\end{equation}
where $\rho_M:M\to M\otimes H$ and $\rho:T\to H\otimes T$ are the comodule structure maps. The vector space $M\square_HT$ is 
called the \define{cotensor product} of $M$ and $T$.

\subsubsection{Left versus right torsors}
\label{ssect.left.v.right}
Since the antipode, $s:H \to H$, is an algebra anti-isomorphism, the categories $^H\cM$ and $\cM^H$ are equivalent: 
if $\rho:X \to H \otimes X$ is a left $H$-comodule, then $X$ becomes a right $H$-comodule with respect to the structure map
\begin{equation}
\label{eq.switch.sides}
\xymatrix{
X \ar[rr]^\rho && H \otimes X \ar[rr]^{s \otimes \id} && H \otimes X \ar[rr]^\tau&& X\otimes H
}
\end{equation}
where the right-most  map is $\tau(h \otimes x)= x \otimes h$.  

\subsubsection{Left versus right comodule algebras}
The operation \Cref{eq.switch.sides} does {\it not} turn a left $H$-comodule algebra into a right $H$-comodule algebra.  
However, if $X$ is a  left $H$-comodule algebra and $X^{\op}$ denotes $X$ with the opposite multiplication, then $X^{\op}$ becomes a 
right $H$-comodule algebra with respect to the structure map \Cref{eq.switch.sides}. To see this, first denote the composition 
in \Cref{eq.switch.sides} by $\rho^\circ$ and, when $x \in X$, write $x^\circ$ for $x$ viewed as an element in $X^\op$. 
Thus, if $x,y  \in X$, then $x^\circ y^\circ = (yx)^\circ$. Therefore if $x,y \in X$ 
and $\rho(x)=x_{-1} \otimes x_0$, then $\rho^\circ(x^\circ)=x_0^\circ \otimes s(x_{-1})$ so 
$$
\rho^\circ(x^\circ y^\circ)=\rho^\circ\big((yx)^\circ\big) =\tau(s \otimes \id)\rho(yx) = \tau(s \otimes \id)(y_{-1}x_{-1} \otimes y_0x_0)=y_0x_0 \otimes s(x_{-1})s(y_{-1})
$$
which is equal to $\big(x_0^\circ  \otimes s(x_{-1})\big)\big(y_0^\circ  \otimes s(y_{-1})\big)=\rho^\circ(x^\circ)\rho^\circ(y^\circ)$.
 
Since $T$ is a left $H$-torsor, $T^\op$ with the structure map $\rho^\circ:T^\op  \to T^\op \otimes H$ is a right $H$-torsor.

\subsubsection{The monoidal functor $\widetilde{\bullet}:M \mapsto \wtM$}
By \cite[Lemma 1.4]{U87}, the functor $M\mapsto M\square_HT$ in \Cref{rem.torsor} is a  monoidal functor.  We denote it by $\widetilde{\bullet}:M \mapsto \wtM$. It is naturally equivalent to $M\mapsto (M\otimes T)^{\rm{co} H}$.

In the expression $M\square_HT$ we treat $T$ as a left $H$-comodule. In the expression $(M\otimes T)^{\rm{co} H}$ we treat
$T$ as a right $H$-comodule using the new structure map in \Cref{eq.switch.sides}. The algebra structure on $T$ in not used in 
constructing either $M\square_HT$ or $(M\otimes T)^{\rm{co} H}$.

\subsubsection{}
\label{ssect.A.tilde}
Since $\widetilde{\bullet}$ is a monoidal functor, it sends algebras in $\cM^H$ to algebras in $\Vect$, and hence for $A\in \cM^H$ as in \Cref{subsec.hopf}, 
\begin{equation}
\label{defn.A.tilde}
\wtA \; := \; (A \otimes T)^{{\rm co} H}  
\end{equation}
has a natural algebra structure. We treat $T$ as a right $H$-comodule in the expression $(A \otimes T)^{{\rm co} H} $.

Although $T$ has two algebra structures, its original one and the opposite one, neither makes $A\otimes T$ into an $H$-comodule algebra
unless additional hypotheses are made (see \Cref{ssect.H.comm}). 
Nevertheless, $\wtA$ is a subalgebra of $A\otimes T$ ($T$ having its initial algebra structure, not the opposite one).
 In \Cref{ssect.H.comm} below we specialize to commutative $H$, in which case $A\otimes T$ is a comodule algebra.

$\widetilde{\bullet}$ lifts to a functor ${}_A\cM^H\to \Mod(\wtA)$, and similarly when everything in sight is $\Ups$-graded for some abelian group $\Ups$. We denote all of these functors by the same symbol, relying on context to differentiate between them.

 \subsubsection{}
In the definition of a torsor, the condition that $T\cong H$ in ${}^H\cM$ makes the 
Galois object \define{cleft}; this condition follows automatically from 
\Cref{eq.torsor}  when $H$ is finite-dimensional, which is the case we are really interested in here. This is (part of) \cite[Thm.
1.9]{Bic10}, which cites \cite{KrCo76} for a proof.  

Cleft objects have an alternative characterization by means of Hopf \define{cocycles}. Recall (e.g. \cite[Example 1.3]{Bic10}) that the latter are linear maps $\sigma:H\otimes H\to k$ satisfying certain conditions that we will not spell out here and which are reminiscent of those from group cohomology.

By \cite[Theorem 1.8]{Bic10}, every left torsor in the sense of \Cref{se.torsor} can be obtained from such a gadget $\sigma$ by twisting $H$: $T$ can be identified with $H$ as a vector space, but has a new multiplication defined by
\begin{equation*}
s\circ t = s_1t_1\sigma(s_2\otimes t_2) \quad \hbox{for all } \, s,t\in H.
\end{equation*}
Here, $s\mapsto s_1\otimes s_2$ is the comultiplication in $H$ and juxtaposition on the right hand side means multiplication in $H$. Similarly, the algebra $\wtA$ can be identified with the vector space $A$ 
endowed with the modified multiplication 
\begin{equation*}
a\circ b = a_0b_0\sigma(a_1\otimes b_1) \quad \hbox{for all } \, a,b\in A,
\end{equation*}
where $a\mapsto a_0\otimes a_1$ is the $H$-comodule structure. 

When $H$ is the function algebra of an abelian group $\G$ whose order is not divisible by the characteristic of $k$ this construction specializes in the following way. 

$H$ can be identified with the group algebra $k\widehat{\G}$ of the character group of $\G$, i.e. $A$ is $\widehat{\G}$-graded. A Hopf cocycle $H\otimes H\to k$ then turns out to be the same thing as (the linear extension of) a normalized group 2-cocycle $\mu:\widehat{\G}\times\widehat{\G}\to k^\times$ in the usual sense. Now, denoting by $A_\a$ the $\a$-homogeneous component of $A$ with respect to the $\widehat{\G}$-grading, the twisted algebra $\wtA$ can be identified with the vector space $A$ together with the new multiplication
\begin{equation*}
  a\circ b = \mu(\a,\b)ab  \qquad \hbox{for all } \, \a,\b\in \widehat{\G},\ a\in A_\a,\ b\in A_\b.
\end{equation*}

\subsection{Generalities}

We prove some auxiliary general results of use below.

\begin{lemma}\label{le.descent}
The categories $\cM^H_{T^\op}$ and  $\Vect$ are equivalent via the mutually quasi-inverse functors
\begin{equation}\label{eq.desc}
  \begin{tikzpicture}[auto,baseline=(current  bounding  box.center)]
    \path[anchor=base] (0,0) node (1) {$\Vect$} +(4,0) node (2) {$\cM^H_{T^\op}$};
         \draw[->] (1) to[bend left=20] node[auto,pos=.5] {$\bullet\otimes T^\op$} (2);
         \draw[->] (2) to[bend left=20] node[auto,pos=.5] {$\bullet^{\rm{co} H}$} (1);
  \end{tikzpicture}
\end{equation}  
\end{lemma}
\begin{proof}
By \cite[Thm. I]{Sch90} applied to the comodule algebra $T^{\op}\in \cM^H$ the assertion follows from the torsor condition \Cref{eq.torsor} if
$T^\op$ is injective as an $H$-comodule. It is because $T\cong H$ as a left comodule and
every coalgebra is self-injective in the same way that every algebra is self-projective. 
\end{proof}

\begin{proposition}\label{pr.descent}
There is an isomorphism
\begin{equation}\label{eq:1.pr.descent}
  \Hom^H(M,N\otimes T)\cong \Hom(\wtM,\wtN),
\end{equation}
functorial in $M,N\in \cM^H$. Moreover, it restricts to a functorial isomorphism
\begin{equation}\label{eq:2.pr.descent}
  \Hom^H_A(M,N\otimes T)\cong \Hom_{\wtA}(\wtM,\wtN)
\end{equation}
for $M,N\in {}_A\cM^H$. 
\end{proposition}

\begin{proof}
By the adjunction between scalar extension $\bullet\otimes T^\op:\cM^H\to \cM^H_{T^\op}$ and scalar restriction (i.e. simply forgetting the $T^{\op}$-action) the left hand side of \Cref{eq:1.pr.descent} is naturally isomorphic to the space $\Hom^H_{T^\op}(M\otimes T^{\op},N\otimes T^\op)$, where $T^\op$ acts on just the $T^\op$ tensorands. In turn, this is naturally isomorphic to the right hand side of \Cref{eq:1.pr.descent} by \Cref{le.descent}. 

To verify the second assertion note that the left hand side of \Cref{eq:2.pr.descent} can be realized as an equalizer
\begin{equation}\label{eq:3.pr.descent}
  \begin{tikzpicture}[auto,baseline=(current  bounding  box.center)]
    \path[anchor=base] (0,0) node (1) {$\Hom^H_A(M,N\otimes T)$} +(4,0) node (2) {$\Hom^H(M,N\otimes T)$} +(10,0) node (3) {$\Hom^H(A\otimes M, N\otimes T)$} ;
         \draw[->] (1) -- (2);
         \draw[->] (2) to[bend left=10] node[auto,pos=.5] {$\scriptstyle f\mapsto f\circ\triangleright$} (3);
         \draw[->] (2) to[bend right=10] node[auto,pos=.5,swap] {$\scriptstyle f\mapsto \triangleright\circ (\id_A\otimes f)$} (3);
  \end{tikzpicture}
\end{equation}
where the upper and lower $\triangleright$ symbols denote the action $A\otimes M\to M$ and $A\otimes N\to N$ respectively. 

Applying the natural isomorphism from the first part of the proposition to the two parallel arrows in \Cref{eq:3.pr.descent}, 
and keeping in mind the fact that $\widetilde{\bullet}$ is a monoidal functor, we get the arrows
\begin{equation*}\label{eq:4.pr.descent}
  \begin{tikzpicture}[auto,baseline=(current  bounding  box.center)]
    \path[anchor=base] (0,0) node (2) {$\Hom(\wtM,\wtN)$} +(6,0) node (3) {$\Hom(\wtA\otimes \wtM, \wtN).$} ;
         \draw[->] (2) to[bend left=10] node[auto,pos=.5] {$\scriptstyle f\mapsto f\circ\triangleright$} (3);
         \draw[->] (2) to[bend right=10] node[auto,pos=.5,swap] {$\scriptstyle f\mapsto \triangleright\circ (\id_A\otimes f)$} (3);
  \end{tikzpicture}
\end{equation*}
Their equalizer is precisely the right hand side of \Cref{eq:2.pr.descent}. 
\end{proof}

\subsubsection{}
 \label{rem.descent}
There is a graded version of \Cref{pr.descent} with virtually the same proof ($M$ and $N$ are graded comodules, etc.).

The following simple observation turns out to be rather important.

\begin{lemma}\label{le.wt_proj}
Suppose $H$ is finite-dimensional. The functors $\widetilde{\bullet}:{}_A\cM^H\to\Mod(\wtA)$ 
and
\newline
  $\widetilde{\bullet}:{}_{\Gr(A)}\cM^H\to \Gr(\wtA)$  send projective objects to projective objects. 
\end{lemma}
 
\begin{proof}
Let $A\sharp H^*$ denote the smash product. 
The category ${}_A\cM^H$ can be identified with $\Mod(A\sharp H^*)$. Under this identification, projectives are direct 
summands of direct sums of copies of $A\sharp H^*$. It therefore suffices to show that the image of $A\sharp H^*$ under $\widetilde{\bullet}$ is projective over $\wtA$. 

As an $A$-module $A\sharp H^*$ is simply $A\otimes H^*$ with the $A$-action on the left tensorand. As an $H$-comodule $A\sharp H^*$ 
is the tensor product $A\otimes H^*$, with $H$ coacting on $H^*$ regularly. Since $\widetilde{\bullet}$ is a monoidal functor, it sends $A\sharp H^*\in {}_A\cM^H$ to $\wtA\otimes \widetilde{H^*}$ with the obvious action of $\wtA$. This is a direct sum of copies of $\wtA$ in $\Mod({\wtA})$ and hence projective. 
\end{proof}

\subsection{The noetherian property and GK-dimension}

\begin{proposition}\label{pr.Hilbert}
Let $\Ups$ be an abelian group and  $A$ an $\Ups$-graded $H$-comodule algebra. Then $\dim_k(A_i)=\dim_k(\wtA_i)$ for all $i \in \Ups$.
\end{proposition}
\begin{proof}  
We are assuming $T\cong H$ in $\cM^H$ so $W\otimes T\cong W\otimes H$ in $\cM^H$ for all $W\in \cM^H$. As in the proof of \Cref{pr.desc}, the map $W\otimes H \to W \otimes H$, $w\otimes h \mapsto w_0\otimes w_1h$, 
is an isomorphism from $W\otimes H$ with the diagonal $H$-coaction to $W\otimes H$ with the regular $H$-coaction on the right-hand tensorand. As a consequence, there is a vector space isomorphism $W\cong (W\otimes T)^{\rm{co} H}$. Now apply this fact with $W$ equal to 
each homogeneous component of $A$. 
\end{proof}
 
 \begin{lemma}
\cite[Lem. 6.1]{KraLen00}
\label{lem.GKdim}
Let $A$ be an $\NN$-graded $k$-algebra such that $\dim_k(A_i)<\infty$ for all $i$, and $M$ a finitely generated graded $A$-module.
Then 
$$
\GKdim(M)=1+ \limsup \log_n(\dim_k(M_n)).
$$
\end{lemma}

\begin{proposition}\label{cor.GK}
If $A$ is a $\ZZ$-graded comodule algebra such that $\dim_k(A_i)<\infty$ for all $i$, then $A$ and $\wtA$ have the same Gelfand-Kirillov dimension. 
\end{proposition}

\begin{lemma}\label{le.equiv_proj_inj}
The functor $\cat{forget}:{}_{\Gr(A)}\cM^H\to {\Gr(A)}$ preserves projectivity, as does the analogous functor for ungraded modules.
\end{lemma}
\begin{proof}
This follows from the fact that $\cat{forget}$ is left adjoint to an exact functor, namely 
$\bullet\otimes H:\Gr(A)\to {}_{\Gr(A)}\cM^H$. The same proof works in the ungraded case. 
\end{proof}

\begin{proposition}\label{pr.noeth}
Suppose $H$ is finite-dimensional. If $A$ is left or right noetherian then so is $\wtA$.
\end{proposition}
\begin{proof}
Suppose $A$ is left noetherian. (The  right noetherian case has a similar proof using the right-handed version of \Cref{pr.descent}.)

Let $S$ be an arbitrary set. The goal is to show that for any $\wtA$-module map $f:\wtA^{\otimes S}\to \wtA$ the images of the restrictions $f_{S'}:\wtA^{\oplus S'}\to \wtA$ stabilize as $S'\subseteq S$ ranges over ever larger finite subsets. 

By \Cref{pr.descent}, $f$ can be identified with some $A$-module $H$-comodule map $\varphi:A^{\oplus S}\to A\otimes T$. By naturality, this identification is compatible with taking restrictions $\varphi_{S'}$ to $A^{\oplus S'}$ for finite subsets $S'\subseteq S$ (in the sense that $f_{S'}$ gets identified with $\varphi_{S'}$). 

From the proof of \Cref{pr.descent} we see that the image of $f_{S'}$ consists of the $H$-coinvariants of the $T^\op$-submodule of $A\otimes T$ generated by the image of $\varphi_{S'}$. Hence, it suffices to show that the images of $\varphi_{S'}$ stabilize as $S'$ increases. This, however, is a consequence of the noetherianness of $A$ and the fact that $T$ is finite-dimensional (so that the $A$-module $A\otimes T$ is finitely generated). 
\end{proof}

\subsection{The case when $H$ is commutative, and the algebra $A'$}
\label{ssect.H.comm}

In this section we assume that $H$ is commutative, i.e., the ring of regular functions on
an affine group scheme (not necessarily reductive or reduced). 

Because $H$ is commutative, if $V$ and $W$ are right $H$-comodules,  the map 
$V \otimes W \to W \otimes V$, $v \otimes w \mapsto w \otimes v$, is an isomorphism of right $H$-comodules. 
It follows from this that if $T$ is made into a right $H$-comodule via the procedure in \Cref{ssect.left.v.right}, then 
\begin{equation}
\label{defn.A'}
A'\;  :=\; A\otimes T
\end{equation}
becomes a right $H$-comodule algebra with its usual tensor product algebra structure. We emphasize that the $T$ factor in $A \otimes T$ has
its original multiplication and is made into a right $H$-comodule algebra by the procedure in \Cref{ssect.left.v.right} and {\it not} by giving $T$
the opposite multiplication.

As mentioned in \Cref{ssect.A.tilde},  $\wtA$ is a subalgebra of $A'$. The following result therefore makes sense.

\begin{proposition}\label{pr.desc}
The categories ${_{A'}}\cM^H$ and  $\Mod(\wtA)$ are equivalent via the mutually quasi-inverse functors
\begin{equation}\label{eq.desc}
  \begin{tikzpicture}[auto,baseline=(current  bounding  box.center)]
    \path[anchor=base] (0,0) node (1) {$\Mod(\wtA)$} +(4,0) node (2) {${_{A'}}\cM^H$};
         \draw[->] (1) to[bend left=20] node[auto,pos=.5] {$\scriptstyle A'\otimes_{\wtA}\bullet$} (2);
         \draw[->] (2) to[bend left=20] node[auto,pos=.5] {$\bullet^{\rm{co} H}$} (1);
  \end{tikzpicture}
\end{equation}
Furthermore, the extension $\wtA\to A'$ is faithfully flat on the right and on the left. 
\end{proposition}
\begin{proof}
It will be convenient to phrase the proof in terms of left comodules. Note that since $H$ is commutative its antipode is an automorphism and 
therefore the equivalence between $\cM^H$ and ${}^H\cM$ described in \Cref{ssect.left.v.right} is a \define{monoidal} equivalence. In this manner, we think of $A$ and $A'$ as left comodule algebras for the duration of the proof, and show that the two functors above implement an equivalence between $\Mod(\wtA)$ and ${}_{A'}^H\cM$. We will also freely interchange the order of tensorands, as permitted by the commutativity of $H$.  

By \cite[Thm. I]{Sch90}, both assertions follow if $A'$ is injective as an $H$-comodule and the map 
\begin{equation}\label{eq.torsor_bis}
\begin{tikzpicture}[auto,baseline=1pt]
  \path[anchor=base] (0,0) node (1) {$A'\otimes A'$} +(4,0) node (2) {$H\otimes A'\otimes A'$} +(8,0) node (3) {$H\otimes A'$};
         \draw[->] (1) to node[] {$\rho\otimes \id$} (2);
         \draw[->] (2) to node[] {$\id\otimes m$} (3);
\end{tikzpicture}
\end{equation}
analogous to \Cref{eq.torsor} is onto, where $\rho:A'\to H\otimes A'$ is the left comodule structure mentioned at the beginning of the proof and $m$ is multiplication.

The $H$-comodule $T\cong H$ is injective in ${}^H\cM$ (every coalgebra is self-injective, in the same way that every algebra is self-projective). Now, for any left $H$-comodule $M$, the map 
\begin{equation*}
  H\otimes M \to H\otimes M, \qquad  h\otimes m\mapsto hm_{-1}\otimes m_0
\end{equation*}
is an isomorphism from $M\otimes H\cong H\otimes M$ with the tensor product comodule structure to $M\otimes H$ with the comodule structure coming from the right hand tensorand alone. In other words $M\otimes H$ is isomorphic in ${}^H\cM$ to a direct sum of $\dim_k(M)$ copies of $H$ and in particular is injective. Applying this to $M=A$, it follows that $A'=A\otimes T\cong A\otimes H$ is injective in ${}^H\cM$. 

To check the surjectivity of \Cref{eq.torsor_bis} note that since \Cref{eq.torsor} is an isomorphism so is the composition 
\begin{equation*}
T\otimes A'=T\otimes T\otimes A \to H\otimes T\otimes A'= H\otimes T\otimes T\otimes A \to H\otimes T\otimes A = H\otimes A', 
\end{equation*}
i.e. the restriction of \Cref{eq.torsor_bis} to $T\otimes A'\subseteq A'\otimes A'$ already surjects onto $H\otimes A'$.    
\end{proof}

\begin{lemma}\label{le.descent_of_finiteness}
Keeping the notation above, if $N\in{}_{A'}\cM^H$ is finitely generated over $A'$, then $N^{\rm{co} H}$ is finitely generated over $\wtA$.   
\end{lemma} 

\begin{proof}
Finite generation can be characterized in category-theoretic terms as follows. 
Let $I$ be a \define{filtered} small category in the sense of \cite[Section IX.1]{Mac98}: Every two objects $i,i'$ fit inside a diagram
\begin{equation*}
  \begin{tikzpicture}[auto,baseline=(current  bounding  box.center)]
    \path[anchor=base] (0,0) node (1) {$i$} +(0,-1) node (2) {$i'$} +(2,-.5) node (3) {$k$};
         \draw[dashed,->] (1) to[bend left=10] (3);
         \draw[dashed,->] (2) to[bend right=10] (3);
  \end{tikzpicture}
\end{equation*}
and every solid left hand wedge as in the picture below can be completed to a commutative diagram by a dotted right hand wedge
\begin{equation*}
  \begin{tikzpicture}[auto,baseline=(current  bounding  box.center)]
    \path[anchor=base] (0,0) node (1) {$j$} +(2,.5) node (2) {$i$} +(2,-.5) node (3) {$i'$} +(4,0) node (4) {$k$};
         \draw[->] (1) to[bend left=10] (2);
         \draw[->] (1) to[bend right=10] (3);
         \draw[dashed,->] (2) to[bend left=10] (4);
         \draw[dashed,->] (3) to[bend right=10] (4);
  \end{tikzpicture}
\end{equation*}
For any functor $F:I\to \Mod(A')$ we have a canonical map 
\begin{equation}\label{eq.filtered}
 \varinjlim_{i\in I}\Hom_{A'}(N,F(i))\to \Hom_{A'}(N,\varinjlim_i F(i)). 
\end{equation}
We leave it to the reader to check that $N$ is finitely generated if and only if for every filtered $I$ and every functor $F$ such that every arrow $F(i\to i')$ is an embedding the map \Cref{eq.filtered} is an isomorphism. 
Also, the hom spaces on the two sides of the arrow are $H$-comodules, and the isomorphism respects these comodule structures. 

Let $F:I\to {}_{A'}\cM^H$ be a functor from a filtered small category such that all $F(i\to i')$ are monomorphisms. 
Since by \Cref{pr.desc} the equivalence
${}_{A'}\cM^H \equiv \Mod(\wtA)$ is effected by the functor
$(\bullet)^{\rm{co} H}$ which preserves filtered colimits, the
analogue of \Cref{eq.filtered} over $A^{\rm{co} H}$ is obtained by
applying this functor to \Cref{eq.filtered}. Since \Cref{eq.filtered} is an
isomorphism, so is its image under $(\bullet)^{\rm{co} H}$.
\end{proof}

There are analogous graded versions of
\Cref{le.descent_of_finiteness} and \Cref{pr.desc}.

\section{Homological properties under twisting}
\label{se.homological}

We keep the notation and conventions from the previous section, under the assumption that $H$ is finite dimensional.
We do not assume $H$ is commutative until \Cref{thm.chi_bis}.

\subsection{} 
Let $A$ be a (usually connected) graded $k$-algebra. For $M,N \in \Gr(A)$ we define the graded vector space
\[
 \uHom(M,N)\; := \; \bigoplus_{d\in\bZ} \Hom(M,N(d)),
\]
where $N(d)$ is the degree shift of $N$ by $d$ and $\Hom$ here is understood from context to be the space of degree-preserving $A$-module maps. Just like ordinary $\Hom$, $\uHom$ has derived functors $\uExt^i$ taking values in the category of graded vector spaces. We denote the degree-$j$ component of $\uExt^i(M,N)$ by $\uExt^i(M,N)_j$, as usual.

If $A$ is noetherian and $M$ is finitely generated then $\uExt(M,-)$ and $\Ext(M,-)$ agree or, more precisely, 
$\Ext(M,-)$ is the vector space obtained by forgetting the grading on $\uExt(M,-)$. This is not the case in general though.

\subsection{}
Let $A$ be a connected graded $k$-algebra in $\cM^H$. If we make the smash product $A \sharp H^*$ into a $\ZZ$-graded $k$-algebra by placing $H^*$ in degree 0, then ${}_{\Gr(A)}\cM^H$ is equivalent to $\Gr(A \sharp H^*)$. Therefore every $M \in {}_{\Gr(A)}\cM^H$
has a resolution by projective objects in ${}_{\Gr(A)}\cM^H$. Let $(P_*,d)$ be such a projective resolution; it is also a projective resolution in $\Gr(A)$ by \Cref{le.equiv_proj_inj}. If $N \in {}_{\Gr(A)}\cM^H$, then the homology of $\uHom_A(P_*,N)$ is in $\cM^H$. Thus, if $M,N \in  {}_{\Gr(A)}\cM^H$, then every $\uExt^i_A(M,N)_j$ is in $\cM^H$:

\begin{lemma}\label{le.ext_comod}
Let $A$ be a connected graded $H$-comodule algebra and $M,N\in {}_{\Gr(A)}\cM^H$. Then the components $\uExt^i_A(M,N)_j$ acquire $H$-comodule structures natural in $M,N\in {}_{\Gr(A)}\cM^H$.
\end{lemma}

Similarly, if $M,N \in {}_A\cM^H$, then  $\Ext^i_A(M,N)\in \cM^H$, naturally in $M$ and $N$.

The following result will be used repeatedly.

\begin{theorem}
\label{thm.Ext.A.A-tilde}
Let $A$ be a connected graded $H$-comodule algebra and $M,N\in {}_{\Gr(A)}\cM^H$. There is a natural isomorphism of bigraded vector spaces
\begin{equation}\label{eq:0.Ext.A.A-tilde}
\uExt_A^*(M,N)_\hdot \; \cong \; \uExt_{\wtA}^*\big(\widetilde{M},\widetilde{N}\big)_\hdot.
\end{equation}
\end{theorem}
\begin{proof}
Let $(P_*,d)$ be a projective resolution of ${}_AM$ in ${}_{\Gr(A)}\cM^H$ (and hence also in ${\Gr(A)}$ by \Cref{le.equiv_proj_inj}). Then $(\wtP_*,\widetilde{d})$
is a projective resolution of $\wtM$ in ${}_{\Gr(\wtA)}\cM$ by \Cref{le.wt_proj}, and $\uExt_{\wtA}^*(\wtM,\wtN)_\hdot$ is the cohomology of the complex $\uHom_{\wtA}(\wtP_*,\wtN)$. By \Cref{pr.descent} (or rather its graded version; see \Cref{rem.descent}), this is the same as the cohomology of the complex 
\begin{equation}\label{eq:1.Ext.A.A-tilde}
\uHom_A^H(P_*,N\otimes T)\cong \uHom_A(P_*,N\otimes T)^{\rm{co} H}\cong (\uHom_A(P_*,N)\otimes T)^{\rm{co} H},  
\end{equation}
where the second isomorphism uses the finite-dimensionality of $T$. 

The right-most complex is the image of $\uHom_A(P_*,N)$ (regarded as a complex of $\bZ$-graded $H$-comodules) under the functor $\widetilde{\bullet}$ to graded vector spaces. Since this functor is exact, it turns the cohomology of $\uHom_A(P_*,N)$, i.e., $\uExt_A^*(M,N)_\hdot$, into that of \Cref{eq:1.Ext.A.A-tilde}. In other words, $\widetilde{\bullet}$ turns the left-hand side of \Cref{eq:0.Ext.A.A-tilde} into its right-hand side.  

Finally, $\widetilde{\bullet}$ is isomorphic to the
forgetful functor $\cM^H\to \Vect$ as a linear functor (though not as a
monoidal functor) because $T\cong H$ as a comodule; the conclusion follows. 
\end{proof}

There is a version of \Cref{thm.Ext.A.A-tilde} for ungraded modules $M,N\in {}_A\cM^H$; the same proof, with the obvious modifications,
works.

 \begin{corollary}\label{cor.m-homog}
 Let $A$ be a connected graded $H$-comodule algebra. If $A \cong TV/(R)$, then $A \cong TV/(\widetilde{R})$ where $\widetilde{R}$ and $R$
 are isomorphic as graded vector spaces.
 \end{corollary}
 \begin{proof}
 This follows by applying \Cref{thm.Ext.A.A-tilde} to $M=N=k$ from the fact that there are isomorphisms $\Ext_A^1(k,k) \cong V^*$ and $\Ext_A^2(k,k) \cong R^*$ of bigraded vector spaces.
 \end{proof}

\subsection{The Koszul property}

\begin{definition}\label{def.Koszul}
Let $m$ be an integer $\ge 2$. A connected graded algebra $A$ is {\sf $m$-Koszul} if $A \cong TV/(R)$ with $\deg(V)=1$, $R \subseteq V^{\otimes m}$, and  $\Ext^i_A(k,k)$ is concentrated in just one degree for all $i$.  
\end{definition}

 \begin{corollary}\label{cor.Koszul}
Let $m$ be an integer $\ge 2$. A connected graded $H$-comodule algebra $A$ is m-Koszul if and only if $\wtA$ is.
\end{corollary}
\begin{proof}
This follows immediately from  \Cref{cor.m-homog}  and \Cref{thm.Ext.A.A-tilde} applied to $M=N=k$.
\end{proof}

\subsection{Artin-Schelter regularity}

We begin by recalling the relevant notions.

\begin{definition}\label{def.AS} 
A connected graded $k$-algebra $A$ is  {\sf Artin-Schelter Gorenstein} (AS-Gorenstein for short) of dimension $d$ if the left and right injective dimensions of 
$A$ as a graded $A$-module equal $d$ and  
\begin{equation}\label{eq.AS}
\uExt_A^i(k,A) = \uExt_{A^\circ}^i(k,A) \cong \d_{id} \, k(\ell). 
\end{equation}
for some integer $\ell$.
  
If $A$ is AS-Gorenstein we say it is {\sf Artin-Schelter regular} (AS-regular for short) of dimension $d$ if in addition 
$\gldim(A)=d<\infty$.
\end{definition}

Artin and Schelter's original definition of regularity included a restriction on the growth of $\dim_k(A_i)$ but in some situations it is sensible to avoid that restriction. We will show that if $A$ is AS-regular of dimension $d$ then so is $\wtA$. Since $\dim_k(A_i)=\dim_k(\wtA_i)$ for all $i$ (\Cref{cor.GK}), if $A$ is AS-regular with the growth restriction so is $\wtA$.

\begin{proposition}\label{pr.gldim}
For all noetherian connected graded algebras $A\in\cM^H$, $\gldim(\wtA) = \gldim(A)$. 
\end{proposition}
\begin{proof}
This follows immediately from \Cref{pr.noeth}, \Cref{thm.Ext.A.A-tilde} and the fact that for noetherian connected graded algebras the homological dimension can be computed as the supremum of those $i$ for which $\Ext^i(k,k)$ is non-zero. 
\end{proof}

\begin{theorem}\label{thm.AS-reg}
If a noetherian connected graded algebra $A\in \cM^H$ is AS-regular of dimension $d$ so is $\wtA$. 
\end{theorem}
\begin{proof}
By \Cref{pr.gldim}, $\gldim(\wtA) = d$. \Cref{thm.Ext.A.A-tilde} and its right handed version applied to $M=k$ and $N=A$ show that \Cref{eq.AS} holds (or does not hold) simultaneously for $A$ and $\wtA$. 
\end{proof}

\begin{corollary}\label{cor.skew_CY}
If $A$ is a noetherian twisted Calabi-Yau algebra, so is $\wtA$.
\end{corollary}
\begin{proof}
By \cite[Lem. 1.2]{RRZ14}, an algebra is twisted Calabi-Yau if and only if it is AS-regular.  
\end{proof}

  We can drop the noetherian hypothesis from \Cref{thm.AS-reg} and \Cref{cor.skew_CY} if we assume that $H$ is cosemisimple, i.e. its category of comodules is semisimple.

\subsection{Condition $\chi$}

In this subsection we prove that the finiteness condition $\chi$ introduced in \cite{AZ94} is preserved under twisting. Throughout, $A$ will be an $\bN$-graded algebra.

\begin{definition}\cite[Defn. 3.7]{AZ94}\label{def.chi}
We say that $A$ has {\sf property $\chi$} if for all non-negative integers $i,d$ and all finitely-generated graded $A$-modules $N$ there is an integer $n_0$ such that $\uExt^i_A(A/A_{\ge n},N)_{\ge d}$ is finitely generated over $A$ for all $n\ge n_0$. (The left $A$-module structure on $\uExt$ comes from the right $A$-action on $A/A_{\ge n}$.)
\end{definition}

The $\chi$ condition is crucial in proving Serre-type results on  finiteness
of cohomology for non-commutative projective schemes (see e.g. \cite[Thm. 7.4]{AZ94}).

\begin{theorem}\label{thm.chi}
If the noetherian connected graded algebra $A\in\cM^H$ of finite global dimension has property $\chi$ then so does $\wtA$. 
\end{theorem}
\begin{proof}
If the finite generation condition from \Cref{def.chi} holds for all $N$ for a fixed choice of $i$ and $d$ we say that condition $\chi^i_d$ holds. 

By \Cref{pr.Hilbert,pr.noeth,pr.gldim}, $\wtA$ is also noetherian connected graded and of finite global dimension. This latter condition means that all sufficiently high $\uExt^i$ vanish, so that we can prove that all $\chi^i_d$ hold by descending induction on $i$. We now do this.

Fix $i$ and suppose we have proved that $\chi^j_d$ holds for all $d$ and all $j>i$. Fix $N\in\Gr(\wtA)$ and $d$ as in \Cref{def.chi}. 
Because $\wtA$ is noetherian,  $N$ is the cokernel in a short exact sequence
\begin{equation*}
  0\to K\to \wtA^{\oplus S}\to N\to 0
\end{equation*}
of finitely generated graded modules. Applying the resulting long exact $\uExt$ sequence and the induction hypothesis we conclude that it suffices to prove that the graded $\wtA$-module $\uExt^i_{\wtA}(\wtA/\wtA_{\ge n},\wtA^{\oplus S})_{\ge d}$ is finitely generated for sufficiently large $n$. 

Just as in the proof of \Cref{thm.Ext.A.A-tilde}, $\uExt^i_{\wtA}(\wtA/\wtA_{\ge n},\wtA^{\oplus S})_{\ge d}$ is the image of 
\begin{equation*}
U_{n} = \uExt^i_{A}(A/A_{\ge n},A^{\oplus S})_{\ge d}\; \in\;  {}_{\Gr(A)}\cM^H  
\end{equation*}
under the functor $\widetilde{\bullet}$. By hypothesis, $U_n$ is finitely generated over $A$ for sufficiently large $n$. Since $U_n$ is also an $H$-comodule, it is finitely generated over $A\sharp H^*$ and hence is a quotient of some finite direct sum of copies of $A\sharp H^*$ in ${}_{\Gr(A)}\cM^H$. Applying $\widetilde{\bullet}$ we obtain  
\begin{equation*}
  \widetilde{U_n} = \uExt^i_{\wtA}(\wtA/\wtA_{\ge n},\wtA^{\oplus S})_{\ge d}
\end{equation*}
as a quotient of a finite direct sum of copies of $\widetilde{A\sharp H^*}\cong \wtA\otimes \widetilde{H^*}\in\Gr(\wtA)$. 
\end{proof}

When $H$ is commutative the noetherian and global dimension hypotheses are not needed.

\begin{theorem}\label{thm.chi_bis}
 If $H$ is commutative and the graded algebra $A\in \cM^H$ satisfies condition $\chi$ then so does $\wtA$.  
\end{theorem}
\begin{proof}
Let $N$ be a finitely generated graded $\wtA$-module and $i,d$ fixed integers. Because $A$ has property $\chi$, there is some $n_0$ for which the finiteness condition in \Cref{def.chi} holds for the graded $A$-module $N'=A'\otimes_{\wtA}N$ (the $A$-module structure is obtained by restricting scalars from $A'=A\otimes T^\op$ to $A$). We will show that $n_0$ satisfies the requirements of \Cref{def.chi} for $N$. 

Apply the graded analogue of \Cref{pr.desc} to identify ${\Gr(\wtA)}$ with ${}_{\Gr(A')}\cM^H$. Arguing as in the proof of \Cref{thm.Ext.A.A-tilde} we see that the $\wtA$-module $\uExt^i_{\wtA}(\wtA/\wtA_{\ge n},N)_{\ge d}$ that we are interested in is precisely the space of $H$-coinvariants in
\begin{equation}\label{eq.chi_bis}
 \uExt^i_{A'}(A'/A'_{\ge n},N')_{\ge d}\cong \uExt^i_{A}(A/A_{\ge n},N')_{\ge d}. 
\end{equation} 
To conclude, apply \Cref{le.descent_of_finiteness} (substituting \Cref{eq.chi_bis} for $N$ in that result).
\end{proof}

\section{``Exotic'' elliptic algebras}\label{se.exotic}

We now apply the above results to Sklyanin algebras. 

\subsection{}\label{subse.n^2}
Fix an integer $n\ge 3$. Let $k=\CC$. Fix a primitive $(n^2)^{\th}$ root of unity $\ve \in k$. 

Let $Q=Q_{n^2,1}(E,\tau)$ be the Sklyanin algebra  defined in \cite{OF89}. 

By \cite[\S1, Remark 2]{OF89}, the finite Heisenberg group of order $n^6$, $H_{n^2}$, acts as automorphisms of $Q$. 
There is a basis $x_i$, $1 \le i \le n^2$, for the degree-1 component of $Q$ on which the generators
of the Heisenberg group act as $x_i \mapsto x_{i+1}$ and $x_i \mapsto \ve^ix_i$ where the indices are labelled modulo $n^2$. 
The $n^{\rm{th}}$ powers of the two generators generate a subgroup $\G \subseteq H_{n^2}$ that is isomorphic to $(\ZZ/n)^2$. 
The generators of $\G$ act by $x_i\mapsto x_{i+n}$  and $x_i\mapsto \zeta^ix_i$ where $\zeta=\ve^n$. 

Let $H=k(\G)$ denote the algebra of $k$-valued functions on $\G$ and let $M_n(k)$ denote the $n\times n$ matrix algebra. 
We make $\G$ act on $M_n(k)$ by having its generators act as conjugation by 
$$
\begin{pmatrix}0&1&0&\cdots&0\\0&0&1&\cdots &0\\\vdots&\vdots&\ddots&\ddots&\vdots\\0&0&0&\ddots&1\\1&0&0&\cdots&0\end{pmatrix}
\qquad \hbox{and} \qquad 
\begin{pmatrix}1&0&\cdots&0\\0&\zeta&\cdots &0\\\vdots&\vdots&\ddots&\vdots\\0&0&\cdots&\zeta^{n-1}\end{pmatrix}.
$$ 

By duality, the action of $\G$ as automorphisms of $M_n(k)$ gives $M_n(k)$ the structure of an $H$-comodule algebra. 

\begin{lemma}
\label{lem.Q.torsor}
The above action makes $M_n(k)$ into a left $H$-torsor in the sense of \Cref{se.torsor}. 
\end{lemma}
\begin{proof}
Every character of $\G$ appears with multiplicity one in $M_n(k)$. In particular, $M_n(k)^{\rm{co} H}=M_n(k)^\G=k$. 

A $k$-algebra on which $\G$ acts as automorphisms is the same thing as a $k$-algebra with a grading by the character group of $\G$. 
Every homogeneous component of $T=M_n(k)$ is the $k$-span of an invertible matrix. Hence, if $\chi$ and $\chi'$ are characters of $\G$,
then $T_\chi T_{\chi'}=T_{\chi\chi'}$. In other words, $T$ is a strongly graded algebra.  A result of Ulbrich
shows that for every group $\Ups$ the $\Ups$-graded algebras that are Galois as comodules over 
 the group algebra $k\Ups$ are exactly the strongly graded ones  \cite[Thm. 8.1.7]{Mon93}. Let $\Ups$ be the character group of $\G$. Using the natural isomorphism, Pontryagin duality, $k\Ups \cong k(\G)=H$, so $T$ is a left $H$-torsor. 
\end{proof}

Let $\widetilde{Q}=(Q\otimes M_n(k))^{\rm{co} H}$.

\begin{proposition}\label{pr.tildeQ_regular}
The algebra $\widetilde{Q}$ is AS-regular of dimension $n^2$, Koszul,  and noetherian,  and has Hilbert series $ {(1-t)^{-n^2}}$.
\end{proposition}  
\begin{proof}
By \cite[Thm. 1.1, Cor. 1.3]{TvdB96}, all the hypotheses of \Cref{pr.Hilbert,cor.Koszul,pr.noeth,thm.AS-reg} are satisfied.
\end{proof}

 \Cref{lem.Q.torsor} and \Cref{pr.tildeQ_regular} hold when $n=2$ and $k$ is any algebraically closed field of characteristic $\ne 2$. 
 See \Cref{se.Q-tilde}.


\section{Generators and relations for  $\widetilde{Q_4}$}
\label{se.Q-tilde}

Let $k$ be an algebraically closed field whose characteristic is  not 2. 

We now specialize the discussion from \Cref{se.exotic} to $n=2$, considering the action of the group $\G=\bZ/2\times\bZ/2$ on $Q=Q_{n^2}=Q_4$.

\subsection{}
Let $\a_1,\a_2,\a_3 \in k$ be such that $\a_1+\a_2+\a_3+\a_1\a_2\a_3=0$ and $ \{\a_1,\a_2,\a_3\} \cap\{0,\pm 1\}=\varnothing$. 
Often we write $\a=\a_1$, $\b=\a_2$, and $\c = \a_3$. 

We fix $a,b,c, i \in k$ such that $a^2=\a$, $b^2=\b$, $c^2=\c$, and $i^2=-1$. 

When $k=\CC$ and $E=\CC/\L$, $\a$, $\b$, and $\c$, are the values at $\tau$ of certain elliptic functions with period lattice $\L$ \cite[\S2]{Skl82},
\cite[\S2.10]{SS92}. 
Thus, when $k=\CC$ we can take
$$
a\,=\,\frac{\theta_{11}(\tau)\theta_{00}(\tau)}{\theta_{01}(\tau)\theta_{10}(\tau)}, \qquad 
b\, =\,  i \frac{\theta_{11}(\tau)\theta_{01}(\tau)}{\theta_{10}(\tau)\theta_{11}(\tau)}, \qquad 
c \, =\,  i \frac{\theta_{11}(\tau)\theta_{10}(\tau)}{\theta_{11}(\tau)\theta_{01}(\tau)},  
$$
where $\theta_{11}, \theta_{00}, \theta_{01}, \theta_{10}$ are Jacobi's four theta functions as defined at \cite[p.71]{Weber-vol3}.

\subsection{}
 Let $Q=k[x_0,x_1,x_2,x_3]$ be the quotient of the free algebra
$k \langle x_0,x_1,x_2,x_3\rangle$ by the six relations
\begin{equation}
\label{Q-relns}
x_0x_i-x_ix_0 \; = \; \a_i(x_jx_k+x_kx_j), \qquad \quad x_0x_i+x_ix_0 \; = \;  x_jx_k-x_kx_j, 
\end{equation}
where $(i,j,k)$ runs over the cyclic permutations of $(1,2,3)$.

\subsection{} 
The earlier results will be applied to the Hopf algebra $H$ of $k$-valued functions on  
$$
\G  \; =\; \{1,\c_1,\c_2,\c_3=\c_1\c_2\}\;  \cong \; \ZZ_2 \times \ZZ_2
$$
and its action as $k$-algebra automorphisms of $Q$ given by
\begin{table}[htdp]
\begin{center}
\begin{tabular}{|c|c|c|c|c|c|c|c|}
\hline
 & $x_0$ &  $x_1$ &  $x_2$ &  $x_3$
\\
\hline
$\c_1 $  & $x_0$ &  $x_1$ &  $-x_2$ &  $-x_3$
\\
\hline
$\c_2 $  & $x_0$ &  $-x_1$ &  $x_2$ &  $-x_3$
\\
\hline
$\c_3$  & $x_0$ &  $-x_1$ &  $-x_2$ &  $x_3$
\\
\hline
\end{tabular}
\end{center}
\vskip .12in
\caption{The action of $\G$ as automorphisms of $Q$}
 \label{Gamma.action}
\end{table}
\newline
\noindent
The irreducible characters of $\G$ are labelled $\chi_0, \chi_1,\chi_2,\chi_3$  in such a way that
$\c(x_j)=\chi_j(\c)x_j$ for all $\c \in \G$ and $j=0,1,2,3$.

\subsection{A quaternionic basis for $M_2(k)$ and the conjugation action of $\G$ on $M_2(k)$}
\label{ssect.quat.basis}
Define
\begin{equation}
\label{defn.a_i}
q_0=\begin{pmatrix} 1 & 0 \\ 0 & 1 \end{pmatrix}, \qquad  q_1=\begin{pmatrix} i & 0 \\ 0 & -i \end{pmatrix}, \qquad q_2= \begin{pmatrix} 0 & i \\ i & 0 \end{pmatrix}, \qquad q_3=  \begin{pmatrix} 0 & -1 \\ 1 & 0 \end{pmatrix}. 
\end{equation}
Then $q_1^2=q_2^2=q_3^2 = -1$ and, if $(i,j,k)$ is a cyclic permutation of $(1,2,3)$, $q_iq_j=q_k$ and $q_iq_j+q_jq_i=0$.

Define an action of $\G$ as automorphisms of $M_2(k)$ by 
$\c_j(a):=q_jaq_j^{-1}$, i.e.,  $g(q_j)=\chi_j(g)q_j$. 

As before, $\wtQ=(Q \otimes M_2(k))^\G$. If $\c \in \G$, then 
$\c(x_iq_j)=\chi_i(\c)\chi_j(\c) x_iq_j
$
so 
$$
y_0:=x_0, \quad y_1:=x_1q_1, \quad y_2:=x_2q_2, \quad y_3:=x_3q_3,
$$
are $\G$-invariant elements of $Q \otimes M_2(k)$.

\begin{proposition}
\label{prop.Q-tilde}
The algebra $\widetilde{Q}$ is generated by $y_0,y_1,y_2,y_3$ modulo the relations
\begin{equation}
\label{Q-tilde-relns}
y_0y_i-y_iy_0 \;=\; \a_i(y_jy_k-y_ky_j) \qquad \hbox{and} \qquad y_0y_i+y_iy_0  \;=\;  y_jy_k+y_ky_j,
\end{equation}
were $(i,j,k)$ is a cyclic permutation of $(1,2,3)$. 
The function  $y_j \mapsto -y_j$, $j=0,1,2,3$, extends to an algebra anti-automorphism of $\wtQ$. 
\end{proposition}
\begin{proof}
Because $\widetilde{Q}$ is Koszul with Hilbert series $(1-t)^{-4}$, it is generated by 4 degree-one elements subject to 6 degree-two
relations. Since $y_0,y_1,y_2,y_3$ are $\G$-invariant elements of degree one, they generate $\widetilde{Q}$. It follows from the quadratic
relations for $Q_4$ that  $(x_0x_i-x_ix_0)q_i = \a_i(x_jx_k+x_kx_j)q_jq_k $ and $(x_0x_i+x_ix_0)q_i =  (x_jx_k-x_kx_j)q_jq_k$. Rewriting 
these relations in terms of $y_0,y_1,y_2,y_3$ gives the relations in (\ref{Q-tilde-relns}).
\end{proof} 

Since $\wtQ$ is a regular noetherian algebra of global dimension and GK-dimension 4, it is a domain by \cite[Thm.3.9]{ATV2}.

\begin{proposition}
\label{prop.Q-tilde-tilde}
There is an action of $\G$ as graded $k$-algebra automorphisms of $\wtQ$ given by
\begin{table}[htdp]
\begin{center}
\begin{tabular}{|c|c|c|c|c|c|c|c|}
\hline
 & $y_0$ &  $y_1$ &  $y_2$ &  $y_3$
\\
\hline
$\c_1 $  & $y_0$ &  $y_1$ &  $-y_2$ &  $-y_3$
\\
\hline
$\c_2 $  & $y_0$ &  $-y_1$ &  $y_2$ &  $-y_3$
\\
\hline
$\c_3$  & $y_0$ &  $-y_1$ &  $-y_2$ &  $y_3$
\\
\hline
\end{tabular}
\end{center}
\vskip .12in
\caption{The action of $\G$ as automorphisms of $\wtQ$}
\label{autom}
\end{table}
\newline
\noindent
Using the conjugation action of  $\G$ as automorphisms of $M_2(k)$, this gives an action of $\G$ as automorphisms of $\wtQ\otimes M_2(k)$.
The invariant subalgebra $(\wtQ \otimes M_2(k))^\G$ is generated by 
 $$
z_0:=y_0, \quad z_1:=y_1q_1, \quad z_2:=y_2q_2, \quad z_3:=y_3q_3
$$
and is isomorphic to $Q$ via  $z_j\mapsto x_j$.
\end{proposition}
\begin{proof}
A calculation shows that the action of $\G$ respects the relations (\ref{Q-tilde-relns}). 
Because $(\wtQ \otimes M_2(k))^\G$ is Koszul with Hilbert series $(1-t)^{-4}$, 
it is generated by 4 degree-one elements subject to 6 degree-two relations. The elements $z_0,z_1,z_2,z_3$ are 
$\G$-invariant so generate $(\wtQ \otimes M_2(k))^\G$. It follows from the quadratic
relations for $\wtQ$ that  $(y_0y_i-y_iy_0)q_i = \a_i(y_jy_k-y_ky_j)q_jq_k $ and $(y_0y_i+y_iy_0)q_i =  (y_jy_k+y_ky_j)q_jq_k$. Rewriting 
these relations in terms of $z_0,z_1,z_2,z_3$ gives the relations 
$z_0z_i-z_iz_0 = \a_i(z_jz_k+z_kz_j)$ and $z_0z_i+z_iz_0=  z_jz_k-z_kz_j$. 
\end{proof}

\subsection{Central elements in $\wtQ$}
In  \cite[Thm.2]{Skl82}, Sklyanin proved that 
\begin{equation}
\label{central.elts}
\Omega:=-x_0^2+x_1^2+x_2^2+x_3^2
\qquad \hbox{and} \qquad
\Omega':= x_1^2+ \bigg(\frac{1+\a_1}{1-\a_2}\bigg)x_2^2+\bigg(\frac{1-\a_1}{1+\a_3}\bigg)x_3^2
\end{equation}
belong to the center of $Q$ when $k=\CC$. 
By the Principle of Permanence of Algebraic Identities, $\Omega$ and $\Omega'$ are central for all $k$.

The elements $x_0^2,x_1^2,x_2^2,x_3^2$ are fixed by the action of $\G$. Since $y_j^2=-x_j^2$ for $j=1,2,3$, the elements 
$$
\Theta:=y_0^2+y_1^2+y_2^2+y_3^2 
\qquad \hbox{and} \qquad
\Theta':= y_1^2+ \bigg(\frac{1+\a_1}{1-\a_2}\bigg)y_2^2+\bigg(\frac{1-\a_1}{1+\a_3}\bigg)y_3^2
$$
belong to the center of $\wtQ$.  We note that $\Theta=-\Omega$ and $\Theta'=-\Omega'$.

\section{$\G$ acts on $E$ as translation by the 2-torsion subgroup}
\label{se.translation}

\subsection{}
If we use $x_0,x_1,x_2,x_3$ as an ordered set of coordinate functions on $Q_1^*$, then the action of $\G$ on $Q_1^*$ induced by its action on $Q_1$ is given by the formulas 
\begin{equation}
\label{Gamma-action}
\begin{cases}
\c_1(\d_0,\d_1,\d_2,\d_3) \;=\; (\d_0,\d_1,-\d_2,-\d_3) & \\
\c_2(\d_0,\d_1,\d_2,\d_3)  \;=\; (\d_0,-\d_1,\d_2,-\d_3) & \\
\c_3(\d_0,\d_1,\d_2,\d_3)  \;=\; (\d_0,-\d_1,-\d_2,\d_3). &
\end{cases}
\end{equation}

We will write $\PP^3$ for $\PP(Q^*_1)$, the projective space of lines in $Q_1^*$. The action of $\G$ on $Q_1^*$ induces an 
action of $\G$ as automorphisms of $\PP^3$ given by the formulas in \Cref{Gamma-action}.

The relations for $Q$, which are elements of $Q_1 \otimes Q_1$, are bi-homogeneous forms on $\PP^3 \times \PP^3$. 
We write $R=\ker(Q_1 \otimes Q_1 \stackrel{\rm{mult}}{\longrightarrow} Q_2)$
 and define the subscheme
$$
V \; :=\; \{(\bfu,\bfv) \; | \; r(\bfu,\bfv)=0 \; \hbox{for all $r \in R$}\} \; \subseteq \; \PP^3 \times \PP^3.
$$
Let $\pr_i:\PP^3 \times \PP^3\to \PP^3$, $i=1,2$, be the projections of $V$ onto the left and right copies of $\PP^3$.

\begin{proposition}
\cite[Props. 2.4, 2.5]{SS92}
\label{prop.SS}
With the above notation,  
$$
\pr_1(V)\;=\; \pr_2(V) \; = \; E\,  \cup\, \big\{(1,0,0,0),\,(1,0,0,0),\,(1,0,0,0),\,(1,0,0,0)\big\}
$$
where $E$ is the intersection of the quadrics
\begin{align*}
 x_0^2+x_1^2+x_2^2+x_3^2 &\;=\; 0,
\\
 (1-\c)x_1^2+(1+\a\c)x_2^2+(1+\a)x_3^2 &\;=\;  0.
\end{align*}
Furthermore, $E$ is an elliptic curve.
\end{proposition}

The reader will notice that we use the same notation for elements in $Q$ as for elements in the symmetric algebra $S(Q_1)$.
Thus, in \Cref{prop.SS},  $x_0^2+x_1^2+x_2^2+x_3^2$ is an element in $S(Q_1)$, i.e., a degree-two form on $\PP^3$, whereas in \Cref{central.elts}, $-x_0^2+x_1^2+x_2^2+x_3^2$ denotes an element in $Q$.

It is clear that $\G$ fixes the points in 
$\{(1,0,0,0),\,(1,0,0,0),\,(1,0,0,0),\,(1,0,0,0)\}$. It is also clear that $E$ is stable under the action of $\G$ (indeed, that must be so because $R$
is $\G$-stable). 
The map $\G \to \Aut(E)$ is injective so we will identity $\G$ with a subgroup of $\Aut(E)$. Once we have fixed a group law $\oplus$ on $E$
we will identify $E$ with the subgroup of $\Aut(E)$ consisting of the translation automorphisms, i.e., $E \to \Aut(E)$ sends a point $\bfv \in E$
to the automorphism $\bfu \mapsto \bfu\oplus \bfv$.

Once we have defined the group $(E,\oplus)$ we will write $o$ for its identity element and
$$
E[2]\; := \; \{\bfv \in E \; | \; \bfv \oplus \bfv=o\}.
$$
The next main result,   \Cref{thm.gp.law}, 
shows we can define $\oplus$ such that $\G=E[2]$ as subgroups of $\Aut(E)$. We will then identify $\G$ with $E[2]$. 
In anticipation of that result we define an involution $\ominus:E \to E$ and a distinguished point $o \in E$ by
\begin{equation}
\label{defn.minus}
\ominus (w,x,y,z) \; :=  \; (-w,x,y,z)
\end{equation}
and
$$
o\; := \; \big(0, \sqrt{\nu -1}, \, \sqrt{1-\mu},\, \sqrt{\mu-\nu}\,\big) 
$$
where 
$$
\mu:= \frac{1-\c}{1+\a} \qquad \hbox{and} \qquad \nu:=\frac{1+\c}{1-\b}
$$
and $\sqrt{\nu -1}$, $\sqrt{1-\mu}$, and $\sqrt{\mu-\nu}$ are some fixed square roots.\footnote{The choice of square root doesn't matter---as one takes the different square roots one obtains 4 different candidates for $o$. But, as we will see, with the choice of $\oplus$ 
we eventually make, those 4 points are the points in $E[2]$. The situation is analogous to that of a smooth plane cubic: there are nine inflection points and if one chooses the group law so that one of those points is the identity, then the inflection points are the points in $E[3]$, the 3-torsion 
subgroup.} The restrictions on the values of $\a$, $\b$, $\c$, imply that $|\{1,\mu,\nu\}|=3$. We use this fact in the proof of \Cref{lem.4.quadrics}.

\begin{lemma}
$E \cap \{x_0=0\}   \;=\; \Big\{p\in E \; \Big\vert \; p=\ominus p\Big\}  \; = \; \Big\{\big(0,\pm \sqrt{\nu -1}, \pm \sqrt{1-\mu}, \pm \sqrt{\mu-\nu}\,\big) \Big\}.
$
 \end{lemma}
\begin{proof}
It follows from the definition of $\ominus$ that $E \cap \{x_0=0\} = \{p\in E \; | \; p=\ominus p\}$.
Computing $E \cap \{x_0=0\}$ reduces to computing the intersection of the plane conics $x_1^2+x_2^2+x_3^2=0$ and 
$\mu x_1^2+\nu x_2^2 +x_3^2=0$. The conics meet at four points, namely 
$
\big(\pm \sqrt{\nu -1}, \, \pm \sqrt{1-\mu}, \, \pm \sqrt{\mu-\nu}\big) \in \PP^2.
$
The result follows.
\end{proof}

\begin{lemma}
\label{lem.C}
There is a degree-two  morphism  $\pi:E \to \PP^1$ such that $\pi(p)=\pi(\ominus p)$ for all $p \in E$, i.e., the fibers of $\pi$ are the sets 
$\{p,\ominus p\}$, $p \in E$. In particular, the ramification locus of $\pi$ is $\{p \in E \; | \; p=\ominus p\}= \{o,\xi_1,\xi_2,\xi_3\}$ where
\begin{align*}
o  & \; := \; \big(0, \sqrt{\nu -1}, \, \sqrt{1-\mu},\, \sqrt{\mu-\nu}\,\big) 
\\
\xi_1 \; := \; \c_1(o) & \;=\; \big(0,\, \sqrt{\nu -1}, \, -\sqrt{1-\mu}, \, \sqrt{\mu-\nu}\,\big) 
\\
\xi_2  \; := \; \c_2(o)& \;=\; \big(0,\, - \sqrt{\nu -1}, \,\sqrt{1-\mu}, \,\sqrt{\mu-\nu}\,\big) 
\\
\xi_3  \; := \; \c_3(o) & \; = \;  \big(0,\, - \sqrt{\nu -1}, \,- \sqrt{1-\mu}, \, \sqrt{\mu-\nu}\,\big).
\end{align*}
\end{lemma}
\begin{proof}
The conic $C$, given by $\mu x_1^2+\nu x_2^2 +x_3^2=0$, is smooth so isomorphic to $\PP^1$. 
Define $\pi:E \to C$ by $\pi(w,x,y,z)=(x,y,z)$. The result is now obvious.
\end{proof}

\begin{proposition} 
\label{prop.E}
Let $E' \subseteq \PP^2$ be the curve $y^2z=x(x-z)(x-\l z)$ where
$$
\l  \; :=\; \frac{\nu- \mu\nu}{\nu-\mu} \; = \; \frac{1}{\c}\Bigg(\frac{1+\c}{1+\a}\Bigg)\Bigg(\frac{\a+\c}{1-\b}\Bigg), 
$$
and consider the group  $(E' ,\oplus)$ in which $(0,1,0)$ is the identity and three points of $E'$ sum to zero if and only if they are collinear.
\begin{enumerate}
\item{}
There is an isomorphism of varieties $g:E \to E'$ such that 
$$
g(o)=\infty=(0,1,0), \quad g(\xi_1)=(0,0,1), \quad g(\xi_2)=(1,0,1), \quad g(\xi_3)=(\l,0,1).
$$
\item{}
If $(E,\oplus)$ is the unique group law such that $g:(E,\oplus) \to (E',\oplus)$ is an isomorphism of groups, then
$E[2]=\{p \; | \; p=\ominus p\}=\{o,\xi_1,\xi_2,\xi_3\}$, and
\item{}
$p \oplus(\ominus p)=o$ for all $p \in E$, and
\item{}
4 points on $E$ are coplanar if and only if their sum is zero.
\end{enumerate}
\end{proposition}
\begin{proof} 
(1)
Let $\pi:E \to C=\{\mu x_1^2+\nu x_2^2 +x_3^2=0\}$ be the morphism $\pi(x_0,x_1,x_2,x_3) =(x_1,x_2,x_3)$ in \Cref{lem.C} and $f:C \to \PP^1$ the isomorphism  
$$
f(x_1,x_2,x_3) = (\sqrt{-\nu} x_2 + \sqrt{\mu}x_1, \, x_3) =  (x_3,\, \sqrt{-\nu} x_2 - \sqrt{\mu}x_1)
$$ 
 with inverse 
$$
f^{-1}(s,t)= \Big(\hbox{$\frac{1}{\sqrt{\mu}}$}(s^2-t^2), \hbox{$\frac{1}{\sqrt{-\nu}}$}(s^2+t^2),2st\Big).
$$
Let $h=f\circ \pi:E \to \PP^1$. The ramification locus of $\pi$, and hence of $h$, is obviously $\{p \in E \; | \; p=\ominus p\}$. 
Let $E'$ be the plane cubic $y^2z=x(x-z)(x-\l z)$ and $h':E' \to \PP^1$ the morphism $h'(x,y,z) = (x,z)$.

Consider the following diagram:
\begin{equation}
  \begin{tikzpicture}[auto,baseline=(current  bounding  box.center)]
    \path[anchor=base] (0,0) node (E) {$E$} +(2,.5) node (E') {$E'$} +(2,-.5) node (C) {$C$} +(4,0) node (P1) {$\PP^1$};
         \draw[dashed,->] (E) to[bend left=10] node[auto,pos=.5] {$\scriptstyle g$} (E');
         \draw[->] (E) to[bend right=10] node[auto,pos=.5,swap] {$\scriptstyle \pi$} (C);
         \draw[->] (E') to[bend left=10] node[auto,pos=.5] {$\scriptstyle h'$} (P1);
         \draw[->] (C) to[bend right=10] node[auto,pos=.5,swap] {$\scriptstyle f$} (P1);
  \end{tikzpicture}
\end{equation}
The following result is implicit in \cite[Ch.4, \S4]{Hart}: If $E$ and $E'$ are elliptic curves and $h:E \to \PP^1$ and $h':E' \to \PP^1$
are degree 2 morphisms having the same branch points, then there is an isomorphism of varieties $g:E \to E'$ such that $h'g=h$.

The four branch points for $h$ are
 $$
 \Big(\pm\sqrt{\mu\nu-\nu} \pm \sqrt{\mu\nu-\mu}, \, \sqrt{\mu -\nu}\Big)= 
 \Big(\sqrt{\mu -\nu}, \, \pm\sqrt{\mu\nu-\nu} \mp \sqrt{\mu\nu-\mu}\Big).
 $$
 The cross-ratios of these four points are 
 $\big\{\l,\frac{1}{\l},  1-\l, \frac{1}{1-\l},\frac{\l}{\l-1},\frac{\l-1}{\l}\big\}$ where 
$$
\l  \; :=\; \frac{\nu- \mu\nu}{\nu-\mu} \; = \; \frac{1}{\c}\Bigg(\frac{1+\c}{1+\a}\Bigg)\Bigg(\frac{\a+\c}{1-\b}\Bigg).
$$

 The four branch points for $h':E' \to \PP^1$ have the same cross-ratios so $E \cong E'$. In particular, 
 there is an isomorphism of varieties $g:E \to E'$ such that 
$$
g(o)=\infty=(0,1,0), \quad g(\xi_1)=(0,0,1), \quad g(\xi_2)=(1,0,1), \quad g(\xi_3)=(\l,0,1).
$$

(2)
Let $\oplus$ be the unique group law on $E$ such that $g(p\oplus p')=g(p)\oplus g(p')$ for all $p,p' \in E$. Then $g$ is an isomorphism
of algebraic groups. 
Since $E'[2]=\{0,1,0),(0,0,1), (1,0,1), (\l,0,1)\}$, $E[2]=\{o,\xi_1,\xi_2,\xi_3\}=\{p \in E \; | \; p=\ominus p\}$.

(3)
Since $g:E \to E'$ is a group isomorphism it suffices to show that 
$g(p)\oplus g(\ominus p)=o$. The fibers of $h$ consist of points that sum
to zero so it suffices to show that $h(g(p)) = h(g(\ominus p))$. However, $hg=f\pi$ and $\pi(p)=\pi(\ominus p)$
so $hg(p)=hg(\ominus p)$. 

(4)
Let $\Phi:\Div(E) \to E$ be the map $\Phi\big((q_1)+\ldots+(q_m)-(r_1)\ldots -(r_n)\big):=q_1 \oplus \ldots \oplus q_m \ominus r_1 \ominus r_n$.
It is easy to show that if $D$ and $D'$ are divisors of the same degree, then $D\sim D'$ if and only if $\Phi(D)=\Phi(D')$. 
The points $\{o,\xi_1,\xi_2,\xi_3\}$ are coplanar. Four points $q_0,\ldots,q_3 \in E$ are coplanar if and only if 
$(o) +(\xi_1)+(\xi_2)+(\xi_3) \sim (q_0) +(q_1)+(q_2)+(q_3)$. Since $o\oplus \xi_1 \oplus \xi_2 \oplus \xi_3=o$, 
$q_0,\ldots,q_3 \in E$ are coplanar if and only if $q_0\oplus q_1 \oplus q_2 \oplus q_3=o$. 
\end{proof}

\begin{lemma} 
\label{lem.4.quadrics}
There are exactly four singular quadrics that contain $E$, namely
\begin{align*}
Q_0 & \; = \; \{\mu x_1^2+\nu x_2^2+x_3^2=0\},
\\
Q_1 & \; = \; \{\mu x_0^2+(\mu-\nu)x_2^2+(\mu-1)x_3^2=0\},
\\
Q_2 & \; = \; \{\nu x_0^2 +(\nu-\mu)x_1^2+(\nu-1)x_3^2=0\},
\\
Q_3 & \; = \; \{x_0^2+(1-\mu)x_1^2+(1-\nu)x_2^2=0\}.
\end{align*}
Let $p \in E$. For each $i$, the line through $\ominus p$ and $\c_i(p)$ lies on $Q_i$. 
\end{lemma}
\begin{proof}
Since the equation defining each $Q_i$ is a linear combination of the equations in \Cref{prop.SS}, $Q_i$ contains $E$.
Each $Q_i$ has a unique singular point, namely $e_i$ where 
$$
e_0:=(1,0,0,0), \quad e_1:=(0,1,0,0), \quad e_2:=(0,0,1,0), \quad e_3:=(0,0,0,1). 
$$ 
Thus $Q_i$ is a union of lines and every line on $Q_i$ passes through $e_i$. 

Let $f_1,f_2$ be quadratic forms such that $E=\{f_1=f_2=0\}$. A quadric contains $E$ if and only if it is the zero locus of 
$\l_1f_1+\l_2f_2$ for some $(\l_1,\l_2) \in \PP^1$; conversely, for all $(\l_1,\l_2) \in \PP^1$ the zero locus of $\l_1f_1+\l_2f_2$ is a
quadric that contains $E$.  Since $|\{1,\mu,\nu\}|=3$, there are exactly 4 singular quadrics in the pencil of quadrics that contain $E$;
these are the quadrics $Q_i$ (see \cite[Prop. 3.4]{LS93}).

Let $p=(w,x,y,z) \in E$. 
Let $L$ be line through $\ominus p$ and $e_0$. Thus $L=\{(t-sw,sx,sy,sz)\; | \; (s,t) \in \PP^1\}$. 
The line $L$ lies on $Q_0$ and meets $E$ when 
$$
 (t-sw)^2+(sx)^2 +(sy)^2+(sz)^2 \; = \; \mu (sx)^2+\nu (sy)^2 +(sz)^2\;=\;0.
 $$
 The second expression is zero for all $s$. The first expression is zero if and only if $t^2-2stw=0$; one solution to this is $t=0$ and
 it corresponds to  the  point $\ominus p \in L \cap E$. The other solution occurs when $t-2sw=0$ and corresponds to the point
 $(w,x,y,z)=p$.  In other words, if  $w \ne 0$, then the line through $p$ and $\ominus p$ lies on  $Q_0$. 
  
 The line through $\ominus p$ and $e_1$ is $\{(-sw,sx+t,sy,sz)\; | \; (s,t) \in \PP^1\}$. It lies on $Q_1$ and meets $E$ when 
$$
(-sw)^2+(sx+t)^2 +(sy)^2+(sz)^2 \; = \; \mu (-sw)^2+\nu (sy)^2 +(sz)^2\;=\;0.
 $$
 The second expression is zero for all $s$ and the first is zero if and only if $t^2+2stx=0$. The solution $t=0$ to this equation corresponds to the 
 point $\ominus p \in L \cap E$. The other solution occurs when $t+2sx=0$ and gives the point
 $(-w,-x,y,z)= \c_1(p)$.  Another way of saying this is that if  $x \ne 0$, then the line through $(-w,x,y,z)$ and $(w,x,-y,-z)$ lies on  $Q_1$;
 i.e., the line through $\ominus p$ and $\c_1(p)$ lies on $Q_1$.
 
 Similar calculations show that the line through $\ominus p$ and $\c_i(p)$ lies on $Q_i$ for $i=2,3$.
\end{proof}

 The statement of \Cref{lem.4.quadrics} doesn't quite make sense if $\ominus p=\c_i(p)$. It should be changed to say that the 
 line through $e_i$ and $\ominus(p)$ meets $E$ again at $\c_i(p)$, i.e., the line is tangent to $E$.

\begin{theorem}
\label{thm.gp.law}
There is a group law $\oplus$ on $E$ such that each element in $\G$ acts as translation by a point in $E[2]$.
\end{theorem}
\begin{proof}
Let $\c_i$ be the automorphism in \Cref{Gamma.action} and let $\xi_i$ be the point in \Cref{lem.C}. 
We will show that $\c_i$ is translation by $\xi_i$, i.e., $\xi_i=\c_i(o)$.
 
Let $p$ and $q$ be arbitrary points of $E$. The line through $\ominus p$ and $\c_i(p)$ lies on $Q_i$. So does the line
through $\ominus q$ and $\c_i(q)$. Because these lines are on $Q_i$ they meet at $e_i$. The lines therefore span a plane, i.e., 
 $\ominus p$, $\c_i(p)$, $\ominus q$, and $\c_i(q)$, are coplanar. Therefore $(\ominus p) \oplus \c_i(p) \oplus( \ominus q) \oplus  \c_i(q)=o$.
 Taking $q=o$ and rearranging the equation gives  $p = \c_i(p) \oplus  \c_i(o)$ or, $\c_i(p)=p \ominus \c_i(o)=p \oplus \c_i(o)$.  
\end{proof}

\subsection{Twisting a $Q$-module by $\c_i$}
\label{sse.twist}

Let $\c\in \G$ and $M$ a graded left $Q$-module. We define $\c^*M$ to be the graded $Q$-module which is equal to $M$ as a graded vector space and has the new $Q$-action
$$
r\hdot_\c m \; := \; \c^{-1}(r)m
$$
for $r \in Q$ and $m \in \c^*M=M$. We make $\c^*$ into an auto-equivalence of $\Gr(Q)$ in the obvious way and we note that these auto-equivalences have the property that $\c^*\d^*=(\c\d)^*$. 

\begin{proposition}
\label{prop.twist} 
Let $p,q \in E$ and let $M_p$ and $M_{p,q}$ be the associated point and line modules. Then $\c_i^*M_p \cong M_{p+\xi_i}$ and 
$\c_i^*M_{p,q} \cong M_{p+\xi_i,q+\xi_i}$.
\end{proposition}
\begin{proof}
Let $r \in Q_1$ and $p \in \PP^3=\PP(Q_1^*)$. 
The action of $\c_i$ on $Q_1$ and $Q_1^*$ is such that $\c_i(r)(p)=r(\c_i^{-1}(p))=r(\c_i(p))$. Thus, $r(p)=0$ if
and only if $\c_i(r)$ vanishes at $\c_i(p)$. Since $M_p =Q/Qp^\perp$ where 
$p^\perp$ is the subspace of $Q_1$ vanishing at $p$, $\c_i^*M_p = Q/Q(p+\xi_i)^\perp$. A similar argument works for line modules.
\end{proof}


\section{Properties of $\wtB$}\label{sec.wtB}

By \cite[\S3.9]{SS92}, $Q/(\Omega,\Omega')$ is isomorphic to the twisted homogeneous coordinate ring $B(E,\tau,\cL)$. 
Since $\Omega$ and $\Omega'$  are fixed by $\G$, there is an induced action of $\G$ on $Q/(\Omega,\Omega')$. 

The quotient $\wtQ/(\Omega,\Omega')$ is isomorphic to $\wtB:=(B(E,\tau,\cL) \otimes M_2(k))^\G$.

\subsection{The category $\QGr(\widetilde{B})$}

Let $B=B ( E, \tau , \cL )$, $B'=B \otimes M_2(k)$, $\widetilde{B} = (B')^{\G}$, and $\cB=M_2(\cO_E)$. The main result in this subsection is

\begin{theorem}\label{th.wtB_Azumaya}
There is an equivalence of categories $\QGr(\wtB)\equiv \qcoh(E/E[2])$.   
\end{theorem}

\begin{corollary}\label{pr.QGr_wtB_simples}
The set of isomorphism classes of simple $\QGr(\wtB)$-objects is in natural bijection with $E/E[2]$. 
\end{corollary}

The plan is to work our way through the chain of equivalences
\begin{equation}\label{eq.wtB_Azumaya}
\QGr(\wtB) \equiv \QGr(B')^\G\equiv \qcoh(\cB)^\G\equiv \qcoh(\cB^\G)\equiv \qcoh(E/E[2]).  
\end{equation}
The notation needs some unpacking. 

First, $\G$ acts on the categories $\QGr(B')$ and $\qcoh(\cB)$. Such an action comprises an auto-equivalence $\gamma^*$ of the respective category for each $\gamma\in\G$ together with natural isomorphisms $t_{\gamma,\delta}: \gamma^*  \circ \delta^* \cong(\gamma\delta)^*$ for $\gamma,\delta\in\G$ such that 
\begin{equation}
\label{Gamma.action.compatibility}
\xymatrix{
\gamma^* \circ \d^* \circ \ve^* \ar[r]  \ar[d]   & (\c\d)^* \circ \ve^* \ar[d]
\\
\c^* \circ(\d\ve)^* \ar[r] & (\gamma\d\ve)^*
}
\end{equation}
commutes for all $\gamma,\delta,\ve \in \Gamma$.

The action of $\G$  as automorphisms of $B'$ induces an action of $\G$ on $\Gr(B')$ as described in \S\ref{sse.twist}. 
Since the subcategory $\Fdim(B')$ is stable under each $\gamma^*$, 
the $\G$-action passes to the quotient category $\QGr(B')$. The action on $\qcoh(\cB)$ comes from translation on $E$ by $E[2]$ together with twisting via the $\G$-action on the $M_2(k)$ tensorand in $\cB=\cO_E\otimes M_2(k)$. 

If $\G$ acts on a category $\cC$ we can then form the category of $\G$-equivariant objects $\cC^\G$. The objects of $\cC^\G$ are objects $c\in\cC$ equipped with isomorphisms $\varphi_\gamma:c\to \gamma^*c$ for $\gamma\in \G$ such that 
\begin{equation}\label{eq.equiv_diagram}
  \begin{tikzpicture}[auto,baseline=(current  bounding  box.center)]
    \path[anchor=base] (0,0) node (c) {$c$} +(2,.6) node (gc) {$\gamma^*c$} +(5,.7) node (gdc1) {$\gamma^*(\delta^*c)$} +(5,-.7) node (gdc2) {$(\gamma\delta)^*c$};
    \draw[->] (c) to[bend left=10] node[pos=.5] {$\scriptstyle \varphi_\gamma$} (gc);
    \draw[->] (gc) to[bend left=10] node[pos=.6] {$\scriptstyle \gamma^*(\varphi_\delta)$} (gdc1);
    \draw[->] (c) to[bend right=10] node[pos=.5,swap] {$\scriptstyle \varphi_{\gamma\delta}$} (gdc2);
    \draw[->] (gdc1) to node[pos=.5] {$\scriptstyle t_{\gamma,\delta}$} (gdc2);
  \end{tikzpicture}
\end{equation}
commutes and the morphisms are those in $\cC$ that preserve all the structure. Explicitly, if $(\varphi_\gamma)_{\gamma \in \G}$
and $(\varphi'_\gamma)_{\gamma \in \G}$ are equivariant structures on objects $c$ and $c'$, respectively, a morphism $f:(\varphi_\gamma)_{\gamma \in \G} \to (\varphi'_\gamma)_{\gamma \in \G}$ is a morphism $f:c \to c'$ in $\cC$ such that $\a^*(f)\varphi_\gamma = \varphi'_\gamma f$ for all $\gamma \in \Gamma$. 
This elucidates the notation $\cC^\G$ in \Cref{eq.wtB_Azumaya} for $\cC=\QGr(B')$ or $\qcoh(\cB)$.

Finally, $\cB^\G$ denotes the sheaf of algebras on $E/E[2]$ obtained by descent from $\cB$. To make sense of this, let $\rho:E \to E/E[2]$ be the \'etale quotient morphism. Now recall

\begin{proposition} \cite[Prop. 2, p.70]{Mum70}
\label{pr.Mum} 
The functors
\[
\cG  \rightsquigarrow \rho^*\cG \qquad \hbox{and} \qquad 
\cF \rightsquigarrow (\rho_*\cF)^\G 
\] 
are mutually inverse equivalences between $\qcoh(E/E[2])$ and $\qcoh(E)^\G$.
\end{proposition}

The equivalences in \Cref{pr.Mum} are  monoidal, because $\rho^*$ is, so they identify $\G$-equivariant sheaves of algebras on $E$ with sheaves of algebras on $E/E[2]$. Keeping this in mind, $\cB^\G$ is simply shorthand for the sheaf of algebras on $E/E[2]$ corresponding to $\cB\in\qcoh(E)^\G$, i.e. $(\rho_* \cB)^\G$.

\begin{proof}[Proof of \Cref{th.wtB_Azumaya}]
We go through the equivalences in \Cref{eq.wtB_Azumaya} one by one, moving rightward.
 
{\bf First equivalence.} The graded version of \Cref{pr.desc} (applied to $B'$ coacted upon by the function algebra of $\G$) provides the equivalence $\Gr(\wtB)$ and $\Gr(B)^\G$. The equivalence restricts to the subcategories $\Fdim(\wtB)$ and $\Fdim(B')^\G$ so descends to the quotient categories $\QGr$.  

{\bf Second equivalence.} By \cite[Thm. 3.12]{AV90},  $\QGr(B)\equiv \qcoh(\cO_E)$. Since $\cB=\cO_E\otimes M_2(k)$, Morita equivalence lifts this to 
\begin{equation}\label{eq.AV90}
\QGr(B')\equiv \qcoh(\cB).   
\end{equation}
Now note that $\G$ acts on the geometric data $(E,\tau,\cL)$ that gives rise to $B=B(E,\tau,\cL)$ in the sense that it acts on $E$, commutes with $\tau$, and there is an $\G$-equivariant structure on $\cL$. Moreover, it acts in the same way on the $M_2(k)$ tensorands in $B'=B\otimes M_2(k)$ and $\cB=\cO_E\otimes M_2(k)$. 
This observation together with the precise description of the equivalence $\QGr(B)\equiv \qcoh(E)$ from \cite[Thm. 3.12]{AV90} shows that \Cref{eq.AV90} intertwines the $\G$-actions on the two categories. This implies the desired result that it lifts to an equivalence between the respective categories of $\G$-equivariant objects.   

{\bf Third equivalence.} This  also follows from \Cref{pr.Mum}. As observed before that equivalence is monoidal, and it identifies $\cB\in \qcoh(E)^\G$ with $\cB^\G\in \qcoh(E/E[2])$. The monoidality then ensures that it implements an equivalence between the categories of modules over $\cB$ and $\cB^\G$ internal to $\qcoh(E)^\G$ and $\qcoh(E/E[2])$ respectively.  

{\bf Fourth equivalence.} 
Because $\rho:E \to E/E[2]$ is \'etale and $\rho^*(\cB^\G) \cong M_2(\cO_E)$, $\cB^\G$ is a sheaf of Azumaya algebras on $E/E[2]$. 
The fourth equivalence now follows from Morita equivalence and the fact that $\cB^\G$ is Azumaya and hence (because we are working over an algebraically closed field) of the form $\cE nd(\cV)$ for some vector bundle $\cV$ on $E/E[2]$. 
\end{proof}

We can actually find an explicit vector bundle $\cV$ on $E/E [ 2 ]$ such that
$\mathcal{B}^{\G} \cong \mathcal{E} n d ( \cV )$. 

\begin{proposition}
Let $\mathcal{V}$ be the unique non-split extension
$0 \longrightarrow \mathcal{O}_{E/E [ 2 ]} \longrightarrow \mathcal{V}
   \longrightarrow \mathcal{O}_{E/E [ 2 ]} \longrightarrow 0$.
  There is an isomorphism of $\mathcal{O}_{E/E [ 2 ]}$-algebras
  $\mathcal{B}^{\G} \cong \mathcal{E} n d ( \mathcal{V} )$.
\end{proposition}

\begin{proof}
  We already know that $\mathcal{B}^{\G}$ is trivial Azumaya, hence
  $\mathcal{B}^{\G} \cong \mathcal{E} n d ( \cV )$ for some
  rank $2$ vector bundle $\cV$. By Atiyah's classification of vector bundles on
  elliptic curves, either $\cV$ is decomposable, or isomorphic to $\mathcal{V}
  \otimes \cL$ for some $\cL \in \mathrm{Pic} (E/E [ 2 ])$. If $\cV$ is decomposable, the $\mathcal{O}_{E/E [ 2 ]}$-module
  $\mathcal{B}^{\G}$ contains two copies of $\mathcal{O}_{E/E [ 2 ]}$ as
  direct summands, whence $\dim  H^{0} ( \mathcal{B}^{\G} ) \geqslant 2$.
  Since $\dim  H^{0} ( \mathcal{B}^{\G} ) = \dim  H^{0} ( \mathcal{B}
  )^{\G} =1$, we must have $\mathcal{B}^{\G} \cong \mathcal{E} n d (
  \mathcal{V} \otimes \cL ) \cong \mathcal{E} n d ( \mathcal{V} )$. 
\end{proof}

\subsection{$E/E[2]$ is a closed subvariety of $\Projnc(\wtQ)$}
The title of this subsection is made precise in the following way.  In \cite[\S3.4]{VdB-blowup}, a subcategory $\cB$ of an
abelian category $\cD$ is said to be {\sf closed} if  it is closed under  subquotients and the inclusion functor $i_*:\cB \to \cD$ is fully faithful and has a left adjoint $i^*$ and a right adjoint $i^!$. In \cite[Thm. 1.2]{SPS15}, which corrects an error in \cite{SPS04},
it is shown that if $J$ is a two-sided ideal in an $\NN$-graded $k$-algebra $A$, then the inclusion functor $\Gr(A/J) 
\to \Gr(A)$ induces a fully faithful functor $i_*:\QGr(A/J) \to \QGr(A)$ whose essential image is closed in the sense of
\cite[\S3.4]{VdB-blowup}. In particular, since $\wtB$ is a quotient of $\wtQ$, this result in conjunction with Theorem
\ref{th.wtB_Azumaya} shows that the essential image of the composition $\qcoh(E/E[2]) \to \QGr(\wtB) \to \QGr(\wtQ)$ is closed in the sense of
\cite[\S3.4]{VdB-blowup}.

\subsection{Fat point modules for $\wtB$}
Let $p \in E$. Let $p^\perp \subset Q_1$ be the subspace of $Q_1$ vanishing at $p$. 
We call $M_p: =Q/Qp^\perp$ the {\sf point module}   associated to $p$. 
We view $k^2$ as a left $M_2(k)$-module in the natural way. Then $M_p \otimes k^2$ is a 
left $Q \otimes M_2(k)$-module, and hence a left $\wtQ$-module.

Since $(\Omega,\Omega')$ annihilates $M_p$, $M_p \otimes k^2$ is a $\wtB$-module.

\begin{lemma}
\label{lem.pts.of.E}
If $p \in E$, then at most one of $\{x_0,x_1,x_2,x_3\}$ vanishes at $p$. 
\end{lemma}
\begin{proof}
Suppose $x_r(p)=x_s(p)=0$ and $r \ne s$. Let $t \in \{0,1,2,3\}-\{r,s\}$. 
There are non-zero scalars $\l,\mu,\nu$ such that $\l x_r^2+ \mu x_s^2+\nu x_t^2$  vanishes on $E$ so $x_t(p)=0$ also. 
But $x_0^2+x_1^2+x_2^2+x_3^2$ vanishes on $E$ so it would follow that $x_j(p)=0$ for all $j$. That is absurd. 
\end{proof}

\begin{proposition}
Let $p \in E$. If $m \otimes v$ is a non-zero element in $(M_p \otimes k^2)_n$, then $\wtQ(m \otimes v) \supseteq (M_p\otimes k^2)_{\ge n+1}$.
In particular, every quotient of $M_p \otimes k^2$ by a non-zero graded $\wtQ$-submodule has finite dimension; i.e., $M_p \otimes k^2$
is 1-critical.
\end{proposition}
\begin{proof}
Let $N$ be a non-zero graded $\wtQ$-submodule of $M_p \otimes k^2$. Let $e_n \otimes v$ be a non-zero element in $N$
where $\{e_n\}$ is a basis for the degree-$n$ component of $M_p$ and $v \in k^2-\{0\}$. 

Every non-zero matrix in $(kq_0+kq_2)\cup(kq_0+kq_3)\cup(kq_1+kq_2) \cup (kq_1+kq_3)$ has rank 2 so
$$
(kq_0+kq_2)v \; = \; (kq_0+kq_3)v \; = \;   (kq_1+kq_2)v \; = \;  ( kq_1+kq_3)v \; = \; k^2.
$$

If $p+n\tau = (\lambda_0,\lambda_1,\lambda_2,\lambda_3)$ with respect to the coordinates $x_0,\ldots,x_3$, 
then there is a basis $\{e_{n+1}\}$ for the degree-$(n+1)$ component of $M_p$ such that $x_j e_n=\lambda_j e_{n+1}$ for $j=0,\ldots,3$.

By  Lemma \ref{lem.pts.of.E}, at least one element in $\{x_0,x_1\}$ and at least one element in $\{x_2,x_3\}$ does not vanish at $p+n\tau$.
Suppose, for the sake of argument, that $x_1(p+n\tau) \ne 0$ and $x_2(p+n\tau) \ne 0$. Then $x_1e_n$ and $x_2e_n$ are non-zero.
It follows that $(kx_1\otimes q_1+kx_2\otimes q_2) \cdot(e_n \otimes v)=e_{n+1} \otimes k^2$. 
Thus, $\wtQ_1(e_n \otimes v)=e_{n+1} \otimes k^2$. The same sort of argument can be used in the other cases
(for example, if $x_0(p+n\tau)$ and $x_2(p+n\tau)$ are non-zero) to show that $\wtQ_1(e_n \otimes v)$ is always
equal to $e_{n+1} \otimes k^2$. 

It now follows by induction on $n$ that $\wtQ (e_n \otimes v)\supseteq (M_p)_{\ge n+1}  \otimes k^2$. The result follows.
\end{proof}

\begin{corollary}
Every simple object in $\QGr(\wtB)$ is isomorphic to $\pi^*(M_p \otimes k^2)$ for some $p \in E$.  
\end{corollary}
 
The previous result is the reason that $M_p \otimes k^2$ is called a {\it fat point module} for $\wtQ$: ``point'' because in algebraic geometry simple objects in $\qcoh(X)$ correspond to closed points, ``fat'' because $\Hom_{\QGr(\wtQ)}(\wtQ, \pi^*(M_p \otimes k^2))=2$, not 1.

\begin{proposition}
If $\omega \in E[2]$  and $p \in E$, then there is an isomorphism of $\wtQ$-modules 
$$
M_p \otimes k^2 \; \cong \; M_{p+\omega}\otimes k^2.
$$
\end{proposition}
\begin{proof}
Write $E[2]=\{o,\xi_1,\xi_2,\xi_3\}$. If $\omega=o$ the identity map is an isomorphism. Fix $i \in \{1,2,3\}$. 

Let $\{e_n \; | \; n \ge 0\}$ be a homogeneous basis for $M_p$ with $\deg(e_n)=n$. 
For each $n$, let $\xi_{nj} \in k$, $j=0,1,2,3$, be the unique scalars such that
$$
x_je_n=\xi_{nj}e_{n+1}.
$$
Thus, $(\xi_{n0},\xi_{n1},\xi_{n2},\xi_{n3})=p+n\tau$. Let $\xi'_{n0}=\xi_{n0}$, $\xi'_{ni}=\xi_{ni}$, and $\xi'_{nj}=-\xi_{nj}$ when 
$j \in \{1,2,3\}-\{i\}$.
Therefore $p+n\tau+\xi_i=(\xi_{n0}',\xi'_{n1},\xi'_{n2},\xi'_{n3})$. 
Let   $\{f_n \; | \; n \ge 0\}$ be the unique homogeneous basis  for $M_{p+\xi_i}$ 
with $\deg(f_n)=n$ such that $x_jf_n=\xi_{nj}'f_{n+1}$ for $j=0,1,2,3$.

Define $\varphi_i:M_p \otimes k^2 \; \longrightarrow  M_{p+\xi_i} \otimes k^2$ by $\varphi_i(e_n \otimes v) := f_n \otimes q_iv$.
It follows that 
$$
\varphi_i\big(y_j \cdot(e_n \otimes v)\big) = \varphi_i(x_je_n \otimes q_jv)= \varphi_i(\xi_je_{n+1} \otimes q_jv)= \xi_jf_{n+1} \otimes q_iq_j v
$$
and
$$
y_j \cdot \varphi_i(e_n \otimes v)= y_j \cdot (f_n \otimes q_iv)= x_jf_{n+1} \otimes q_jq_iv)= \xi_j'f_{n+1} \otimes q_jq_i v.
$$
For all $j$, $\xi_jf_{n+1} \otimes q_iq_j v=\xi_j'f_{n+1} \otimes q_jq_i v$  because
\begin{itemize}
  \item 
  if $j\in \{0,i\}$, then $\xi_j=\xi_j'$ and $q_iq_j=q_jq_i$;
  \item 
 if $j \in \{1,2,3\}-\{i\}$, then $\xi_j=-\xi_j'$ and $q_iq_j=-q_jq_i$. 
\end{itemize}
Therefore $\varphi_i\big(y_j \cdot(e_n \otimes v)\big) = y_j \cdot \varphi_i(e_n \otimes v)$ for $j=0,1,2,3$. This proves that $\varphi_i$ is 
a homomorphism of graded $\wtQ$-modules. It is obviously bijective so the proof is complete.
\end{proof}

\subsection{$\wtB$ is a prime ring}

 Davies \cite[Cor. 5.3.21]{Davies-arXiv} proved that $\wtB$ is a prime ring when $\tau$ has infinite order 
\cite[Hypothesis 5.0.2]{Davies-arXiv}. We use a different method to prove the result without any restriction on $\tau$. 

\begin{proposition}
Let $I_1$ and $I_2$ be graded ideals in an $\NN$-graded left and right noetherian $k$-algebra $A$. 
Suppose there is a projective scheme $X$ and an equivalence of categories $\Phi:\QGr(A) \to \Qcoh(X)$.  
By \cite{SPS15}, there are functors $\a_{1*}$ and $\a_{2*}$, and closed subschemes $Z_1,Z_2 \subseteq X$ 
such that the essential image of $\Phi \a_{i*}$ is equal to $\Qcoh(Z_i)$, and  there is a commutative diagram 
\begin{equation}
\label{closed.subspace}
\xymatrix{
\Gr (A/I_1) \ar[r]^{f_{1*}}  \ar[d]_{\pi_1}  & \Gr (A)  \ar[d]^{\pi} & \ar[l]_{f_{2*}} \Gr(A/I_2) \ar[d]^{\pi_2}
\\
\QGr(A/I_1)  \ar[r]_{\a_{1*}}  &  \QGr(A) \ar[d]^\Phi & \ar[l]^{a_{2*}} \QGr(A/I_2)
\\
&  \Qcoh(X).
}
\end{equation}
in which $f_{i*}:\Gr(A/I_i) \to \Gr(A)$, $i=1,2$, are the natural inclusion functors, and $\pi_1$, $\pi_2$,
and $\pi$ denote the quotient functors. 
If $I_1 \cap I_2=0$ and $X$ is reduced and irreducible, then $Z_1 \cup Z_2=X$.
\end{proposition}
\begin{proof}
Let $\cO_x$ be the skyscraper sheaf at a closed point  $x \in X$ and $M$ an $A$-module such that 
$\Phi \pi M \cong \cO_x$ and $\pi^*(M/N)=0$ for all non-zero $N \subseteq M$. 
If $I_2M=0$, then $\cO_x \cong \Phi i_{2*}\pi_2M$ so $x \in Z_2$. 
On the other hand, suppose $I_2M \ne 0$.
Then  $\pi(M/I_2M)=0$ so $\pi(I_2M) \cong \cO_x$. Since  $I_1I_2M=0$, 
$\pi(I_2M)=\pi f_{1*}(I_2M)=i_{1*}\pi_1M$ which implies that $i_{1*}\pi_1M \cong \cO_x$.
Hence $x \in Z_1$. 

Thus, every closed point of $X$ belongs to $Z_1 \cup Z_2$. The proposition now follows 
from the fact that $X$ is reduced and irreducible.
\end{proof}

\begin{theorem}
\label{thm.B-tilde_prime}
Let $A$ be a connected, $\NN$-graded, left and right noetherian $k$-algebra 
Suppose there is a projective scheme $X$ and an equivalence of categories $\Phi:\QGr(A) \to \Qcoh(X)$.  
If $A$ is semiprime and  $X$ is reduced and irreducible, then $A$ is a prime ring.
\end{theorem}
\begin{proof}
Suppose the result is false. Then there are non-zero elements $x$ and $y$ such that $xAy=0$.
If $x_m$ and $y_n$ are the top-degree components of $x$ and $y$, then $x_mAy_n=0$. Let $I_1=Ax_mA$ and
$I_2=Ay_nA$. Then $I_1$ and $I_2$ are graded ideals such that $I_1I_2=0$. Since $(I_1 \cap I_2)^2 \subseteq I_1I_2$,
the fact that $A$ is semiprime implies $I_1 \cap I_2 =0$. Hence $Z_1 \cup Z_2=X$. But $X$ is irreducible so 
either $Z_1=X$ or $Z_2=X$.

Without loss of generality suppose that $Z_1=X$. Then the functor $i_{1*}:\QGr(A/I_1) \to \QGr(A)$ is an equivalence.
In particular, there is a module $M \in \Gr(A/I_1)$ such that  $\pi A \cong i_{1*}\pi_1M=\pi f_{1*}M$.
Hence, if $\omega$ is the right adjoint to $\pi$ constructed by Gabriel, 
$\omega\pi A \cong \omega \pi f_{1*}M$. By Step 2 in the proof of \cite[Thm. 1.2]{SPS15}, 
$ \omega \pi  f_{1*}M \cong f_{1*}\omega'\pi' M$ where $\omega'$ is right adjoint to $\pi'$. It follows that $I_1$ 
annihilates $\omega\pi A$.

There is an exact sequence $0 \to T \to A \to \omega\pi A$ where $T$ is the largest finite dimensional submodule of $A$. 
Since $A_0=k$, $T\subseteq A_{\ge 1}$. It follows that $T^n=0$ for $n \gg 0$. But $A$ is semiprime so $T=0$. 
Therefore $I_1$ annihilates $A$. Hence $I_1=0$.  
\end{proof}

\begin{corollary}
$\wtB$ is a prime ring.
\end{corollary}
\begin{proof}
As observed in \cite[Cor. 5.1.8]{D}, because $B$ is a domain $B \otimes M_2(k)$ is a prime ring, so \cite[Cor. 1.5(1)]{Mon80}
shows that  $(B \otimes M_2(k))^\G$, which is $\wtB$, is a semiprime ring. Therefore 
 \Cref{th.wtB_Azumaya,thm.B-tilde_prime} imply that $\wtB$ is a prime ring.
\end{proof}

\subsubsection{Remark}
The hypothesis in \Cref{thm.B-tilde_prime} that the algebra $A$ is connected was needed to show that 
$A$ does not contain a non-zero left ideal of finite dimension. For $\wtB$, one can prove that without
appealing to the fact that $\wtB$ is connected.
Since $\wtB = \wtQ/(\Theta,\Theta')$ where $\Theta,\Theta'$ is a regular sequence on $\wtQ$ of length 2,
the projective dimension of $\wtB$ as a left $\wtQ$-module is 2. 
Hence, by \cite[Prop. 2.1(e)]{LS93}, $\wtB$ does not contain a non-zero left ideal of finite dimension.

\subsubsection{}
The twisted homogeneous coordinate ring of a reduced and irreducible variety, in particular $B(E,\tau,\cL)$, is a domain.
 
 \begin{proposition}
\label{prop.B-tilde.not.domain}
$\wtB$ is not a domain. In particular, in $\wtB$, $0=y_0^2+y_1^2+y_2^2+y_3^2  = $
$$ 
(y_0-y_1-y_2-y_3)^2
 \; = \; (y_0-y_1+y_2+y_3)^2
  \; = \; (y_0+y_1-y_2+y_3)^2
\; = \; (y_0+y_1+y_2-y_3)^2.
$$
\end{proposition}
\begin{proof}
This is a straightforward calculation:  $(y_0-y_1-y_2-y_3)^2$ equals
$$
 y_0^2+y_1^2+y_2^2+y_3^2 -\sum_{i=1}^3(y_0y_i+y_iy_0 - y_jy_k-y_ky_j)
 $$
 where $(i,j,k)$ is a cyclic permutation of $1,2,3$. But $y_0y_i+y_iy_0 = y_jy_k+y_ky_j$ when $(i,j,k)$ is a cyclic permutation of $1,2,3$
 and $ y_0^2+y_1^2+y_2^2+y_3^2 =-\Omega$ which is zero in $\wtB$.
Similar calculations show that the squares of the other 3 elements are zero in $\wtB$; alternatively, one can use the fact that $\G$ acts as automorphisms of  $\wtB$ and these four elements in $\wtB_1$ form a $\G$-orbit. 
\end{proof}

\section{Point modules for $\wtQ$}
\label{sect.pt.modules}

A {\sf point module} for  a connected graded algebra $A$ is a graded left $A$-module $M$ such that $M=AM_0$ and $\dim_k(M_i)=1$ for all $i \ge 0$. The importance of point modules is that they are simple objects in $\QGr(A)$. 

\subsection{}
Suppose $M$ is a point module for $\wtQ$. Its degree-zero component, $M_0$, is annihilated by a 3-dimensional subspace of $\wtQ_1$. 
That 3-dimensional subspace determines and is determined by a point in $\PP^3$, its vanishing locus. 
We will show that the only points in $\PP^3$ that arise in this way are those in Table \ref{table.20.pts}  where 
the coordinates are written with respect to the coordinate system $(y_0,y_1,y_2,y_3)$. We write $\fP$ for this set of points. 

Recall that $a,b,c,i$ are fixed square roots of $\a,\b,\c,-1$.

\begin{table}[htdp]
\begin{center}
\begin{tabular}{|l||l|l|l|l|l|}
\hline
$\quad \fP_\infty$ & $\qquad \fP_0$ & $\qquad \fP_1$ & $\qquad \fP_2$ & $\qquad \fP_3$ & $\G$
\\
\hline
\hline
$(1,0,0,0)$ & $(1,1,1,1)$ & $(bc,-i,-ib,-c)$ & $(ac,-a,-i,-ic)$&   $(ab, -ia,-b,-i)$  & 
\\
\hline
$(0,1,0,0)$ &$(1,1,-1,-1)$   & $(bc,-i,ib,c)$  &  $(ac,-a,i,ic)$&  $(ab, -ia,b,i)$  & $\c_1$
\\
\hline
 $(0,0,1,0)$ &$(1,-1,1,-1)$   & $(bc,i,-ib,c)$  &  $(ac,a,-i,ic)$&  $(ab, ia,-b,i)$  & $\c_2$
  \\
\hline
 $(0,0,0,1)$ & $(1,-1,-1,1)$ &$(bc,i,ib,-c)$  & $(ac,a,i,-ic)$  & $(ab, ia,b,-i)$  & $\c_3$
  \\
\hline
\end{tabular}
\end{center}
\vskip .12in
\caption{The points in $\fP$.}
\label{table.20.pts}
\end{table}
The points in $\fP_\infty$ are fixed by $\G$ and every other $\fP_i$ is a $\G$-orbit. 
If $\bfu$ is the topmost point in one of the columns $\fP_i$, $i=0,1,2,3$,
 the other points in that column are $\c_1(\bfu)$, $\c_2(\bfu)$, and $\c_3(\bfu)$, in that order.

We define a permutation $\theta$ of $\fP$ with the property  $\theta^2=\id_{\fP}$ by
\begin{equation}
\label{eq.theta}
\theta(\bfu):= 
\begin{cases} 
\bfu & \text{if $\bfu \in \fP_\infty \cup \fP_0$}
\\
\c_i(\bfu) & \text{if $\bfu \in \fP_i$, $i=1,2,3$.}
\end{cases}
\end{equation}

\subsection{The point scheme, $\cP$}
\label{ssect.pt.mods.Gamma}
\label{ssect.Gamma}
Let $V$ denote the linear span of $y_0,y_1,y_2,y_3$. The defining relations for $\widetilde{Q}$ belong to $V^{\otimes 2}$. 
Non-zero elements in $V^{\otimes 2}$ are forms of bi-degree $(1,1)$ on $\PP(V^*) \times \PP(V^*)=\PP^3 \times \PP^3$.
Let
$$
\mathcal{P}: \; =\; \hbox{the subscheme of $\PP^3 \times \PP^3$ where the quadratic relations for $\wtQ$ vanish}.
$$
We will show that $\mathcal{P}$ is a reduced scheme consisting of 20 points.

\begin{lemma}
\label{lem.Gamma.symm}
If $(\bfu,\bfv) \in \mathcal{P}$, then $(\bfv,\bfu) \in \mathcal{P}$.
\end{lemma}
\begin{proof}
As remarked in Proposition \ref{prop.Q-tilde}, there is an anti-automorphism of $\wtQ$ given by $y_i \mapsto -y_i$ for $i=0,1,2,3$.
Thus, if $r=\sum \mu_{ij}y_i\otimes y_j$ is a quadratic relation for $\wtQ$ so is $r'=\sum \mu_{ij}y_j\otimes y_i$. Obviously, $r$ vanishes at  
$(\bfu,\bfv)\in \PP^3 \times \PP^3$ if and only if $r'$ vanishes at $(\bfv,\bfu)$. 
The lemma now follows from the fact that $\mathcal{P}$ is the zero locus of the set of quadratic relations for $\wtQ$. 
\end{proof}

\subsubsection{From point modules to points in $\cP$}
\label{sssect.pt-mods.pts}
Suppose $M$ is a point module for $\wtQ$. Let $e_0,e_1,\ldots$ be a basis for $M$ with $\deg(e_n)=n$. 
Define $\l_{nj} \in k$ by the requirement that $y_je_n=\l_{nj} e_{n+1}$. Because $M$ is a point module, for each $n$, some $\l_{nj}$ is non-zero. 
The point $p_n:=(\l_{n0},\l_{n1},\l_{n2},\l_{n3})   \in \PP^3$ does not depend on the basis $\{e_n\}_{n \ge 0}$. 
Since $y_j(p_n)=\l_{nj}$, the $p_n$'s belong to $\PP(V^*)$. 

Because $M$ is a $\wtQ$-module, each quadratic relation $r \in V^{\otimes 2}$ has the property that $r\cdot e_n=0$ for all $n$.
Thus, $r$ viewed  as a (1,1) form on $\PP^3 \times \PP^3$ vanishes at $(p_{n+1},p_n)$. 
Hence $(p_{n+1},p_n) \in \mathcal{P}$.  

\subsection{The point modules $M_\bfu$, $\bfu \in \fP$} 

\begin{proposition}
\label{prop.pt.mods}
Let $\bfu \in \fP$. Let $\theta$ be the function defined at \Cref{eq.theta} and for each $n \ge 0$ write 
$\theta^n(\bfu)=(\l_{n0},\l_{n1},\l_{n2},\l_{n3})$ where the coordinates are written with respect to $(y_0,y_1,y_2,y_3)$. There is a point module,
$M_\bfu$,  with homogeneous basis $e_0,e_1,\ldots$, $\deg(e_n)=n$, and action 
\begin{equation}
\label{eq.pt.mod}
y_j e_n := \l_{nj} e_{n+1}.
\end{equation}
These 20 point modules are pair-wise non-isomorphic. 
\end{proposition}
\begin{proof}
It is clear that $M_\bfu$ is generated by $e_0$ so it suffices to show  that (\ref{eq.pt.mod}) really does define a left $\wtQ$-module. To do this 
we must show that every relation for $\wtQ$ annihilates every $e_n$. In other words, we must show that every quadratic relation for $\wtQ$, when viewed as a form of bi-degree $(1,1)$ on $\PP^3 \times \PP^3$, vanishes at $\big((\l_{n+1,0},\l_{n+1,1},\l_{n+1,2},\l_{n+1,3}),(\l_{n0},\l_{n1},\l_{n2},\l_{n3})\big) \in \PP^3 \times \PP^3$ for all $n \ge 0$.; i.e., it suffices to show that these forms vanish at $(\theta(\bfv),\bfv)$ for all $\bfv \in \fP$. Since $\theta^2=1$, this is equivalent to showing they vanish at $(\bfv,\theta(\bfv))$ for all $\bfv \in \fP$. 

The relations for $\wtQ$ are the entries in the matrix $\sM_1\sfy$ where 
$$
\sM_1=\begin{pmatrix}
-y_1 & y_0 & \a y_3 & - \a y_2 \\
-y_2 & -\b y_3 & y_0 & \b y_1 \\
-y_3 & \c y_2 & -\c y_1 & y_0 \\
y_1 & y_0 & -y_3 & - y_2 \\
y_2 & -y_3 & y_0 & -y_1 \\
y_3 &  -y_2 & -y_1 & y_0 \\
\end{pmatrix}
\quad \hbox{and} \quad  \sfy = \begin{pmatrix} y_0 \\ y_1 \\ y_2 \\ y_3 \end{pmatrix}.
$$
We must therefore show that $\sM_1(\bfv)\theta(\bfv)^\sT=0$ for all $\bfv \in \fP$. This is a routine calculation. We give one example to illustrate the 
process. 

Let $\bfv=(\d_0,\d_1,\d_2,\d_3) \in \fP_1$. Then $\theta(\bfv)=\c_1(\bfv) = (\d_0,\d_1,-\d_2,-\d_3)$
so
$$
\sM_1(\bfv)\theta(\bfv)^\sT = 
\begin{pmatrix}
-\d_1 & \d_0 & \a \d_3 & - \a \d_2 \\
-\d_2 & -\b \d_3 & \d_0 & \b \d_1 \\
-\d_3 & \c \d_2 & -\c \d_1 & \d_0 \\
\d_1 & \d_0 & -\d_3 & - \d_2 \\
\d_2 & -\d_3 & \d_0 & -\d_1 \\
\d_3 &  -\d_2 & -\d_1 & \d_0 \\
\end{pmatrix}
\begin{pmatrix} \d_0 \\ \d_1 \\ -\d_2 \\- \d_3 \end{pmatrix}
=
2\begin{pmatrix} 0\\ -\d_0\d_2-\b \d_3\d_1 \\-\d_0\d_3+\c \d_1\d_2\\ \d_0\d_1+\d_2\d_3 \\ 0 \\ 0 \end{pmatrix}.
$$
It is easy to check that this $6 \times 1$ matrix is 0 for all  $\bfv\in \fP_1$. 

The annihilator of $e_0$ in $\wtQ_1$ is the subspace that vanishes at $\bfu$. Hence if $\bfu$ and $\bfv$ are different points of $\fP$,
$M_\bfu \not\cong M_\bfv$. 
\end{proof}

\begin{theorem}
\label{thm.pt.mods}
The 20 point modules $M_\bfu$, $\bfu \in \fP$, in Proposition \ref{prop.pt.mods} are all the $\wtQ$-point modules.
\end{theorem}
\begin{proof}
 Let $M$ be a point module for $\wtQ$. Let $\{e_n\; | \; n \ge 0\}$ be a homogeneous basis for $M$ with $\deg(e_n)=n$. 
 Let $p_n$, $n \ge 0$, be the points in $\PP^3$ determined by  the procedure described in \S\ref{sssect.pt-mods.pts}.
 Then  $(p_{n+1},p_n) \in \mathcal{P}$ for all $n \ge 0$. By \Cref{lem.Gamma.symm}, $(p_n,p_{n+1}) \in \mathcal{P}$. Thus, to prove the 
 Theorem it suffices to show that $\mathcal{P}= \big\{\big(\bfu,\theta(\bfu)\big) \; \big\vert \; \bfu \in \fP \big\}$. This is what we do in
 \Cref{thm.pt.mods.Q-tilde} below. 
 \end{proof}

\begin{theorem}
\label{thm.pt.mods.Q-tilde} 
Let $\mathcal{P} \subseteq \PP^3 \times \PP^3$ be the subscheme defined in \S\ref{ssect.Gamma}.  Then 
$$
\mathcal{P}\;=\; \{(\bfu,\bfv) \in \PP^3 \times \PP^3 \; | \; \sM_1(\bfu)\bfv=0\}  
\; =\;  \big\{\big(\bfu,\theta(\bfu)\big) \; \big\vert \; \bfu \in \fP \big\}.
$$ 
In particular, $\mathcal{P}$ is the graph of the automorphism $\theta$ of $\fP$. 
\end{theorem}
\begin{proof}
 Let $\pr_1,\pr_2:\mathcal{P} \to \PP^3$ denote the projections onto the first and second factors of $\PP^3 \times \PP^3$. We will show that $\pr_1(\mathcal{P})=\fP$.
 Let $\bfu \in \pr_1(\mathcal{P})$. There is a point $\bfv \in \PP^3$ such that $(\bfu,\bfv) \in \mathcal{P}$, i.e., such that $\sM_1(\bfu)\bfv=0$.
 This implies that $\rank(\sM_1(\bfu)) \le 3$. Thus the $4 \times 4$ minors of $\sM_1$  vanish at $\bfu$. 
We used SAGE \cite{sage} to compute these minors. After removing a common factor of 2, they are
 \begin{align*}
&
 -b\c y_0y_1^3 - \a \c y_0y_1y_2^2 + \b \c y_1^2y_2y_3 + a\c y_2^3y_3 - \a \b y_0y_1y_3^2 + \a \b y_2y_3^3 - y_0^3y_1 + y_0^2y_2y_3
 \\
 &
 \qquad = (y_2y_3-y_0y_1)(y_0^2+\b\c y_1^2+\a\c y_2^2+\a\b y_3^2),
\\
&
 -\b \c y_0y_1^2y_2 - \a \c y_0y_2^3 + \b \c y_1^3y_3 + \a \c y_1y_2^2y_3 - \a \b y_0y_2y_3^2 + \a \b y_1y_3^3 - y_0^3y_2 + y_0^2y_1y_3
 \\
  &
 \qquad = (y_1y_3-y_0y_2)(y_0^2+\b\c y_1^2+\a\c y_2^2+\a\b y_3^2),
\\
&
 \b \c y_1^3y_2 + \a \c y_1y_2^3 - \b \c y_0y_1^2y_3 - \a \c y_0y_2^2y_3 + \a \b y_1y_2y_3^2 - \a \b y_0y_3^3 + y_0^2y_1y_2 - y_0^3y_3
  \\
  &
 \qquad = (y_1y_2-y_0y_3)(y_0^2+\b\c y_1^2+\a\c y_2^2+\a\b y_3^2),
\\
&
 -\a \b y_1^2y_3^2 + \a \b y_2^2y_3^2 - \b y_0^2y_1^2 - \a y_0^2y_2^2 + \b y_1^2y_3^2 + \a y_2^2y_3^2 - y_0^2y_1^2 + y_0^2y_2^2, 
\\
&
-\a \b y_1^2y_2y_3 + \a \b y_2y_3^3 + \b y_0y_1^3 - \a y_0^2y_2y_3 + \a y_2^3y_3 - \b y_0y_1y_3^2 + y_0^3y_1 - y_0y_1y_2^2
\\
&
\qquad=(y_0 y_1-\a y_2y_3)(y_0^2+\b y_1^2-y_2^2-\b y_3^2),
\\
&
-\a \b y_1y_2^2y_3 + \a \b y_1y_3^3 - \a y_0y_2^3 + \b y_0^2y_1y_3 - \b y_1^3y_3 + \a y_0y_2y_3^2 + y_0^3y_2 - y_0y_1^2y_2
\\
&
\qquad=(y_0 y_2+\b y_1y_3)(y_0^2 - y_1^2- \a y_2^2 + \a y_3^2),
\\
&
 \a \c y_1^2y_2y_3 - \a \c y_2^3y_3 + \c y_0y_1^3 - \c y_0y_1y_2^2 - \a y_0^2y_2y_3 + \a y_2y_3^3 - y_0^3y_1 + y_0y_1y_3^2,
\\
&
\qquad=(y_0 y_1+\a y_2y_3)(-y_0^2 +\c y_1^2- \c y_2^2 +  y_3^2),
\\
&
 \a \c y_1^2y_2^2 - \a \c y_2^2y_3^2 - \c y_0^2y_1^2 + \c y_1^2y_2^2 - \a y_0^2y_3^2 + \a y_2^2y_3^2 + y_0^2y_1^2 - y_0^2y_3^2,
\\
&
 \a \c y_1y_2^3 - \a \c y_1y_2y_3^2 - \c y_0^2y_1y_2 + \c y_1^3y_2 - \a y_0y_2^2y_3 + \a y_0y_3^3 + y_0^3y_3 - y_0y_1^2y_3
\\
&
\qquad=(y_0 y_3 -\c y_1y_2)(y_0^2 - y_1^2-\a y_2^2 +\a  y_3^2),
\end{align*}
\begin{align*}
&
 \a y_0y_1y_2^2 + \a y_2^3y_3 - \a y_0y_1y_3^2 - \a y_2y_3^3 - y_0^3y_1 + y_0y_1^3 - y_0^2y_2y_3 + y_1^2y_2y_3
\\
&
\qquad = (y_0y_1+y_2y_3) (-y_0^2 + y_1^2 +\a y_2^2 - \a y_3^2),
\\
&
 -\b \c y_1^3y_3 + \b \c y_1y_2^2y_3 + \c y_0y_1^2y_2 - \c y_0y_2^3 + \b y_0^2y_1y_3 - \b y_1y_3^3 - y_0^3y_2 + y_0y_2y_3^2
\\
&
\qquad = (y_0y_2-\b y_1y_3) (-y_0^2 + \c y_1^2 -\c y_2^2 + y_3^2),
\\
&
 -\b \c y_1^3y_2 + \b \c y_1y_2y_3^2 - \c y_0^2y_1y_2 + \c y_1y_2^3 - \b y_0y_1^2y_3 + \b y_0y_3^3 - y_0^3y_3 + y_0y_2^2y_3
\\
&
\qquad = (y_0y_3 +\c y_1y_2) (-y_0^2 - \b y_1^2 + y_2^2 +\b  y_3^2),
\\
&
 -\b \c y_1^2y_2^2 + \b \c y_1^2y_3^2 - \c y_0^2y_2^2 + \c y_1^2y_2^2 - \b y_0^2y_3^2 + \b y_1^2y_3^2 - y_0^2y_2^2 + y_0^2y_3^2,
\\
&
 -\b y_0y_1^2y_ - \b y_1^3y_3 + \b y_0y_2y_3^2 + \b y_1y_3^3 - y_0^3y_2 + y_0y_2^3 - y_0^2y_1y_3 + y_1y_2^2y_3
\\
&
\qquad = (y_0y_2+y_1y_3) (-y_0^2 -\b y_1^2 + y_2^2 +\b y_3^2),
\\
&
 \c y_1^3y_2 - \c y_1y_2^3 + \c y_0y_1^2y_3 - \c y_0y_2^2y_3 - y_0^2y_1y_2 - y_0^3y_3 + y_1y_2y_3^2 + y_0y_3^3 
\\
&
\qquad = (y_0y_3+y_1y_2) (-x_0^2 +\c y_1^2 -\c y_2^2+y_3^2).
 \end{align*} 
Some reorganization and changes of sign show that the linear span of the above 15 polynomials is the same as the linear span of the 
following 15 polynomials:  
  \begin{align*}
&
 (y_2y_3-y_0y_1)(y_0^2+\b\c y_1^2+\a\c y_2^2+\a\b y_3^2)
\\
&
  (y_1y_3-y_0y_2)(y_0^2+\b\c y_1^2+\a\c y_2^2+\a\b y_3^2)
\\
&
 (y_1y_2-y_0y_3)(y_0^2+\b\c y_1^2+\a\c y_2^2+\a\b y_3^2)
\\
&
(y_0y_1+y_2y_3) (y_0^2 - y_1^2 -\a y_2^2 + \a y_3^2)
\\
&
(y_0 y_2+\b y_1y_3)(y_0^2 - y_1^2- \a y_2^2 + \a y_3^2)
\\
&
(y_0 y_3 -\c y_1y_2)(y_0^2 - y_1^2-\a y_2^2 +\a  y_3^2)
\\
&
(y_0 y_1-\a y_2y_3)(y_0^2+\b y_1^2-y_2^2-\b y_3^2)
\\
&
(y_0y_2+y_1y_3) (y_0^2+\b y_1^2-y_2^2-\b y_3^2)
\\
&
 (y_0y_3 +\c y_1y_2) (y_0^2+\b y_1^2-y_2^2-\b y_3^2)
\\
&
(y_0 y_1+\a y_2y_3)(y_0^2 -\c y_1^2+ \c y_2^2 -  y_3^2)
\\
&
(y_0y_2-\b y_1y_3)(y_0^2 -\c y_1^2+ \c y_2^2 -  y_3^2)
\\
&
 (y_0y_3+y_1y_2) (y_0^2 -\c y_1^2+ \c y_2^2 -  y_3^2)
\\
&
 \a \b y_1^2y_3^2 - \a \b y_2^2y_3^2 + \b y_0^2y_1^2 - \b y_1^2y_3^2 + \a y_0^2y_2^2  - \a y_2^2y_3^2 + y_0^2y_1^2 - y_0^2y_2^2, 
\\
&
 \b \c y_1^2y_2^2 - \b \c y_1^2y_3^2 + \c y_0^2y_2^2 - \c y_1^2y_2^2 + \b y_0^2y_3^2 - \b y_1^2y_3^2 + y_0^2y_2^2 - y_0^2y_3^2,
\\
&
 \a \c y_1^2y_2^2 - \a \c y_2^2y_3^2 + \a y_2^2y_3^2 - \c y_0^2y_1^2 + \c y_1^2y_2^2 - \a y_0^2y_3^2  + y_0^2y_1^2 - y_0^2y_3^2.
 \end{align*}

The proof of Proposition \ref{prop.pt.mods} showed that $\sM_1(\bfu)\theta(\bfu)^\sT=0$ for all $\bfu \in \fP$ so 
these 15 polynomials vanish at the points in $\fP$. One can also check this directly by evaluating these 
quartic polynomials at $\bfu \in \fP$. For example, it is obvious that $y_iy_j$ vanishes on $\fP_\infty$ if $i \ne j$ from which it immediately follows
that all 15 polynomials vanish on $\fP_\infty$. As another example, $y_2y_3-y_0y_1$, $y_1y_3-y_0y_2$, and $y_1y_2-y_0y_3$,
vanish on $\fP_0$, whence the first 3 of the 15 polynomials vanish on $\fP_0$; the other twelve polynomials belong to the ideal 
$(y_0^2-y_1^2,y_0^2-y_2^2,y_0^2-y_3^2)$ so they too vanish on $\fP_0$. As a final example, consider $\fP_2$. The first three quartics
vanish on $\fP_2$ because $y_0^2+\b\c y_1^2+\a\c y_2^2+\a\b y_3^2$ does. The second three quartics vanish on $\fP_2$ because
$y_0^2 - y_1^2 -\a y_2^2 + \a y_3^2$ does. The third three quartics vanish on $\fP_2$ because
the ideal $(y_0y_1-\a y_2y_3, y_0y_2+ y_1y_3, y_0y_3+\c y_1y_2)$ does. The fourth three quartics
vanish on $\fP_2$ because $y_0^2 -\c y_1^2+ \c y_2^2 -  y_3^2$  does. A calculation shows the last three quartics vanish on $\fP_2$.

Suppose these 15 quartics vanish at a point $\bfu\in \PP^3$. To complete the proof we will show that $\bfu \in \fP$. 
 
The determinant 
 $$
\det \begin{pmatrix}
1& \b\c &  \a\c & \a\b 
\\
1& - 1 &  -\a & \a 
\\
1& \b &-1 & -\b
\\
1&   -\c & \c& -1
\end{pmatrix} 
\; = \; -(1+\a\b+\b\c+\c\a)^2
$$
is non-zero: the hypothesis that $\a+\b+\c+\a\b\c=0$ implies $1+\a\b+\b\c+\c\a=(1+\a)(1+\b)(1+\c)$ which is non-zero
 because we are assuming that $\{\a,\b,\c\} \cap \{0,\pm 1\}=\varnothing$.  
Because the determinant is non-zero the polynomials 
\begin{equation}
\label{4.quadrics}
\begin{cases}
& y_0^2+\b\c y_1^2+\a\c y_2^2+\a\b y_3^2, 
\\
&y_0^2 - y_1^2 -\a y_2^2 + \a y_3^2, 
\\
& y_0^2 +\b y_1^2 - y_2^2 -\b y_3^2,  
\\
&y_0^2 -\c y_1^2 +\c y_2^2 - y_3^2,
\end{cases}
\end{equation}
 are linearly independent.  
Their linear span is therefore the same as that of $\{y_0^2,y_1^2,y_2^2,y_3^2\}$. 
Hence the common zero locus of the  polynomials in \Cref{4.quadrics} is empty and at most three of  them vanish at $\bfu$. 
  
We now do some case-by-case analysis to show that $\bfu$ belongs to some $\fP_i$.
 
\underline{$\fP_\infty  \cup \fP_0$.}
Suppose $\bfu$ is not in the zero locus of $y_0^2+\b\c y_1^2+\a\c y_2^2+\a\b y_3^2$.
 Then
\begin{equation}
\label{eq.1.01.23}
y_0y_1 -  y_2y_3=y_0y_2- y_1y_3=y_0y_3-y_1y_2=0
  \end{equation}
 at $\bfu$. If one of the coordinate functions $y_0,y_1,y_2,y_3$ vanishes at $\bfu$, then three of do so
  \begin{equation}
  \label{eq.coord.pts}
 \bfu \; \in \; \{(1,0,0,0),(0,1,0,0),(0,0,1,0),(0,0,0,1)\} \; =\; \fP_\infty.
 \end{equation}
If none of $y_0,y_1,y_2,y_3$ vanishes at $\bfu$, then it follows from (\ref{eq.1.01.23}) that 
 $$
 \bfu \; \in \; \{(1,1,1,1),(1,1,-1,-1),(1,-1,1,-1),(1,-1,-1,1)\} \; = \; \fP_0.
 $$
 
 \underline{$\fP_1$.}
Suppose $\bfu$ is not in the zero locus of $y_0^2 - y_1^2 -\a y_2^2 + \a y_3^2$ and not in $\fP_\infty \cup \fP_0$.
Then
\begin{equation}
\label{eq.2.01.23}
y_0y_1+y_2y_3=y_0 y_2+\b y_1y_3=y_0 y_3 -\c y_1y_2=0
 \end{equation}
 at $\bfu$. If one of $y_0,y_1,y_2,y_3$ vanishes at $\bfu$, then three of them do so $\bfu \in \fP_\infty$. This is not the case so
none of $y_0,y_1,y_2,y_3$ vanishes at $\bfu$.  Without loss of generality we can, and do, assume that   $\bfu=(bc,y_1,y_2,y_3)$.  
It follows from (\ref{eq.2.01.23})  that $y_0^3(y_1y_2y_3)=\b\c(y_1y_2y_3)^2$. Therefore $bc=y_1y_2y_3$.
It also  follows from (\ref{eq.2.01.23})  that $\b \c y_1^2= \c y_2^2 = -\b y_3^2$.  Some case-by-case analysis shows that 
 $$
 \bfu \; \in \; \{(bc,-i,ib,c),(bc,-i,-ib,-c),(bc,i,ib,-c),(bc,i,-ib,c)\} \; = \; \fP_1.
 $$
 
 \underline{$\fP_2$.}
Suppose  $\bfu$ is not in the zero locus of $y_0^2 +\b y_1^2 - y_2^2 -\b y_3^2$ and not in $\fP_\infty \cup \fP_0$. Then
\begin{equation}
\label{P2-quadrics}
y_0y_1-\a y_2y_3=y_0 y_2+ y_1y_3=y_0 y_3 +\c y_1y_2=0
 \end{equation}
at $\bfu$. As in the previous paragraph, $y_0y_1y_2y_3$ does not vanish at $\bfu$. Without loss of generality we can, and do, assume that   $\bfu=(ac,y_1,y_2,y_3)$.  The same sort of analysis as that in the previous paragraph shows that
$$
 \bfu \; \in \; \{ (ac,a,-i,ic),(ac,a,i,-ic),  (ac,-a,-i,-ic),(ac,-a,i,ic) \; = \; \fP_2.
 $$
 
  \underline{$\fP_3$.}
Suppose $\bfu$ is not in the zero locus of $y_0^2 -\c y_1^2 +\c y_2^2 - y_3^2$ and not in $\fP_\infty \cup \fP_0$. Then
\begin{equation}
\label{P2-quadrics}
y_0y_1+\a y_2y_3=y_0 y_2-\b  y_1y_3=y_0 y_3 + y_1y_2=0
 \end{equation}
at $\bfu$. Proceeding as before, we eventually see that  
 $$
 \bfu \; \in \; \{ (ab, ia,b,-i),(ab, ia,-b,i),(ab, -ia,b,i), (ab, -ia,-b,-i) \} \; =\; \fP_3.
 $$
This completes the proof that $\pr_1(\mathcal{P}) \subset \fP$. Thus $\pr_2(\mathcal{P})=\fP$. 
 
By Lemma \ref{lem.Gamma.symm}, $\pr_2(\mathcal{P})=\fP$ also. Since $\pr_2(\mathcal{P})$ does not contain a line, the rank of $\sM_1(\bfu)$ is 3 for all
$\bfu \in \pr_1(\mathcal{P})$. Let $\bfu \in \fP$. Since $\sM_1(\bfu)\theta(\bfu)^\sT=0$, $\theta(\bfu)^\sT$ is the only $\bfv \in \PP^3$ such that  $\sM_1(\bfu)\bfv^\sT=0$. Hence $(\bfu,\theta(\bfu))$ is the only point in $\pr_1^{-1}(\bfu)$. It follows that $\mathcal{P}=\{(\bfu,\theta(\bfu)) \; | \; \bfu \in \fP\}$. 
\end{proof}
 
 \begin{proposition}
 \label{prop.no.B-tilde.points}
 The central element $\Theta=y_0^2+y_1^2+y_2^2+y_3^2$ does not annihilate any point modules for $\widetilde{Q}$.
 Consequently, $\widetilde{B}$ has no point modules.
\end{proposition}
\begin{proof}
Let $\bfu \in \fP$. 

To describe the action of $\Theta$ on  $M_\bfu$ we must  fix a basis for $M_\bfu$.  
We pick a basis for $M_\bfu$ that is compatible with the entries in Table \ref{table.20.pts}. To do this it is helpful, for a moment, to think of
the entries in Table \ref{table.20.pts} as points in $k^4$. Suppose $\bfu =(\d_0,\d_1,\d_2,\d_3)$. 
Let $e_0$ be any non-zero element in $(M_\bfu)_0$. 
Let $e_1$ be the unique element in $(M_\bfu)_1$ such that $y_ie_0=\d_ie_1$ for $i=0,1,2,3$. 
Likewise, if $(\d_0',\d_1',\d_2',\d_3')$ is the entry in  Table \ref{table.20.pts} for $\theta(\bfu)$, 
there is a unique element $e_2 \in (M_\bfu)_2$ such that $y_ie_1=\d_i'e_2$ for $i=0,1,2,3$.

If $\bfu \in \fP_\infty$, then $\Theta e_0=e_2$. If $\bfu \in \fP_0$, then $\Theta e_0=4e_2$.

Let  $\bfu=(bc,-i,-ib,-c) \in \fP_1$. Then $\theta(\bfu)=(bc,-i,ib,c)$. Therefore
\begin{align*}
\Theta e_0 & \;=\; (y_0^2+y_1^2+y_2^2+y_3^2)e_0 
\\
& \;=\; (bcy_0-iy_1-iby_2-cy_3)e_1
 \\
&   \;=\; \big((bc)^2-1+b^2-c^2\big)e_2
 \\
&  \;=\; (\b-1)(\c+1)e_2
  \end{align*}
  Likewise, if $\bfu=  (bc,i,-ib,c) \in \fP_1$, then $\theta(\bfu)=(bc,i,ib,-c)$ and a similar calculation shows that 
$\Theta e_0 = (\b-1)(\c+1)e_2$. Thus, $\Theta e_0 = (\b-1)(\c+1)e_2$ for all $\bfu \in \fP_1$.

Similar calculations show that $\Theta e_0= (\a+1)(\c-1)e_2$ for all $\bfu \in \fP_2$.
Finally, if $\bfu \in \fP_3$, then $\Theta e_0= (\a-1)(\b+1)e_2$. 
\end{proof}

\subsection{}
Not only do the relations for $\wtQ$ determine $\mathcal{P}$, but $\mathcal{P}$ determines the defining relations for $\wtQ$:
the quadratic relations for $\wtQ$ are precisely the elements of $V^{\otimes 2}$ that vanish at $\mathcal{P}$. 
This is a consequence of the following remarkable result.

\begin{theorem}
[Shelton-Vancliff]
\cite{ShV02}
Let $V$ be a 4-dimensional vector space and $R \subseteq V^{\otimes 2}$ a 6-dimensional subspace. Let $TV$ denote the tensor algebra on $V$
and let $\mathcal{P} \subset \PP(V^*) \times \PP(V^*)$ be the scheme-theoretic zero locus of $R$. If $\dim(\mathcal{P})=0$, then
$$
R=\{ f \in V^{\otimes 2} \; | \; f|_\mathcal{P}=0\}.
$$
\end{theorem}

\subsection{}
 There has been some interest in Artin-Schelter regular algebras with Hilbert series $(1-t)^{-4}$ that have only
finitely many point modules \cite{VVW98}, \cite{SV99}, \cite{SV06}, \cite{SV07}. The interest arises because this phenomenon does not occur
for Artin-Schelter regular algebras with Hilbert series $(1-t)^{-3}$;  the point modules for the latter algebras are parametrized 
either by a cubic divisor in $\PP^2$ or by $\PP^2$. In 1988, M. Van den Bergh circulated a short note showing that a generic 
4-dimensional Artin-Schelter regular algebra with Hilbert series $(1-t)^{-4}$ has exactly 20 point modules
\cite{VdB88}. Van den Bergh's example is a generic Clifford algebra. In particular, it is a finite module over its center. 

Davies \cite[\S5.1]{D}  shows, when the translation automorphism has infinite order, that $\widetilde{Q}$ is not isomorphic to any of the 
previously found examples of 4-dimensional regular algebras having 20 point modules.

\begin{proposition}
\label{prop.pt.mod.rings}
The point modules $M_\bfu$ for $\bfu \in \fP_\infty \cup \fP_0$ are quotient rings of $\wtQ$. If $\bfu=(\l_0,\l_1,\l_2,\l_3) \in  \fP_\infty \cup \fP_0$, then
$$
M_\bfu \; \cong \; \frac{\wtQ}{(\l_jy_i-\l_iy_j \; | \; 0 \le i,j \le 3)} \; \cong \; k[t].
$$
\end{proposition}

\begin{proposition}
The scheme-theoretic zero locus in $\PP^3 \times \PP^3$ of the relations for $\widetilde{Q}$ is a reduced scheme with 20 points.
\end{proposition}
\begin{proof}
(Van den Bergh \cite{VdB88}.)
We have already seen that the relations for $\wtQ$ vanish at 20 points in $\PP^3 \times \PP^3$.
Let $X$ denote the image of the Segre embedding $\PP^3 \times \PP^3 \to \PP^{15}$. If we view $\PP^{15}$ as the space of $4 \times 4$
matrices, then $X$ is the space of rank-one matrices. By \cite[\S18.15]{H92}, for example, the degree of $X$ is ${6 \choose 3} =20$.
The 6 defining relations for $\wtQ$ are linear combinations of terms $x_ix_j$ which, under the Segre embedding, become linear
combinations of the coordinate functions $x_{ij}$. Hence the vanishing locus of the relations in $\PP^{15}$ is the vanishing locus of 6 linear forms,
hence a linear subspace, $L$ say, of dimension 9. Hence, by B\'ezout's Theorem, if the scheme-theoretic intersection $L \cap X$ is finite it has degree 20. But, $L \cap X$ consists of 20 different points so it is reduced. 
\end{proof}

\section{Secant lines to $E$ and line modules for $Q$}
\label{se.line_modules_for_Q}

The relevance of this section will become apparent in \S\ref{se.line_modules} when we construct some line modules for 
$\wtQ$ that are parametrized by certain lines in $\PP(Q_1^*)$. To make the word ``parametrized'' precise we 
will show that the parametrizing space is a closed subvariety of the Grassmannian of lines in $\PP(Q_1^*)$.

\subsection{Secant lines}

The second symmetric power of $E$ is the quotient variety $S^2E:=(E \times E)/\ZZ_2$ where $\ZZ_2$ acts by 
$(p,q) \mapsto (q,p)$. We think of the points in $S^2E$ as effective divisors of degree 2 on $E$ and write 
$(p)+(q)$ for the image of  $(p,q) \in E \times E$. 

Because the quartic curve $E \subset \PP(Q_1^*)=\PP^3$ has no trisecants, there is a well-defined 
morphism $E \times E \to \GG(1,3)$ that sends $(p,q) \in E \times E$ to $\overline{pq}$, the line in $ \PP(Q_1^*)=\PP^3$ whose scheme-theoretic intersection with $E$ is $(p)+(q)$. By the universal property of the quotient  $(E \times E)/\ZZ_2$
this morphism factors through a morphism $\c:S^2E \to \GG(1,3)$. The image of $\c$ is a closed subscheme
of $\GG(1,3)$ called the variety of secant lines to $E$. See \cite[Ex. 8.3]{H92}, for example.

\begin{proposition}\label{pr.S2E-G}
The map $\c:S^2E \to \GG(1,3)$  defined by $\c\big((p)+(q)\big):=\overline{pq}$ is a closed immersion.
\end{proposition}
\begin{proof}
 The morphism $\c$ is injective because $E$ has no trisecants, so it
 suffices to argue that the image of the morphism is smooth. This
 follows from the standard description of the singular points of a
 secant variety: a line in the image of $\c$ is singular if and only
 if it is a trisecant (see e.g. the discussion on page 312 of \cite{H92} regarding Exercise 16.11 in that book). 
 \end{proof}

\subsection{The line modules $M_{p,q}$}
A {\sf line module} for $Q$, or $\wtQ$, is a cyclic graded module whose Hilbert series is $(1-t)^{-2}$. 

\begin{theorem}
\cite[Thm. 4.5]{LS93}
The function that sends $(p)+(q) \in S^2E$ to $Q/Qx+Qx'$ where $\overline{pq}=\{x=x'=0\}$ is a bijection from $S^2E$
to the set of isomorphism classes of line modules for $Q$.
\end{theorem}

If $(p)+(q) \in S^2E$ and $\overline{pq}=\{x=x'=0\}$  we write $M_{p,q}:=Q/Qx+Qx'$.

\subsection{}
In \S\ref{se.line_modules} we will show that if $y=y'=0$ is a line in $\PP(\wtQ_1^*)=\PP(\wtQ_1^*)$ that meets 
$E$ at $(p)+(p+\xi)$ for some $p \in E$ and $\xi \in E[2]-\{o\}$, then $\wtQ/\wtQ y + \wtQ y'$ is a line modules for $\wtQ$.
Such lines will be parametrized by the subscheme of $\GG(1,3)$ that is the image of the composition
$E/\langle \xi \rangle \to S^2E \to \GG(1,3)$.

\begin{lemma}\label{le.E-S2E}
The morphism $\b:E/\langle \xi \rangle \to S^2E$ defined by $\b(p+\langle \xi \rangle)=(p)+(p+\xi)$ is a closed immersion.
\end{lemma}
\begin{proof}
It is clear that $\b$ is injective as a set map on the closed points
of $E/\langle \xi \rangle$, so it suffices to prove that its
derivative is one-to-one on each tangent space, or equivalently that
the composition of $\b$ with the \'etale map $\pi:E\to
E/\langle\xi\rangle$ has this same property. 

The composition $\beta\pi$ is
\begin{equation*}
  E\to  E\times E\to S^2E,
\end{equation*}
where the left hand arrow sends $p$ to $(p,p+\xi)$ and the right hand arrow
is the quotient morphism. Since the latter is \'etale off the diagonal $\Delta\subset E\times E$ and the former is a closed immersion into $E\times E\setminus \Delta$ the conclusion follows. 
\end{proof}

\section{Line Modules for $\wtQ$}
\label{se.line_modules}

\subsection{}
\label{ssect.lines1}

In this section we exhibit three families of line modules for $\wtQ$ parametrized by the disjoint union of the 
three elliptic curves $E/\langle \xi \rangle$ as $\xi$ ranges over the three 2-torsion points of $E$. The isomorphism classes
of the line modules parametrized by $E/\langle \xi \rangle$ are in natural bijection with the lines 
$\overline{p,p+\xi}$, $p \in E$; the union of these lines is an elliptic scroll in $\PP(S_1^*)$.

These are {\it not} all the line modules for $\wtQ$.

\subsection{}
\label{ssect.twisting.2}
By \Cref{prop.twist}, the elements in $\G=\{\c_0,\c_1,\c_2,\c_3\}$ and $E[2]=\{\xi_0,\xi_1,\xi_2,\xi_3\}$ 
may be labelled in such a way that $\c_i^*M_{p,q} \cong M_{p+\xi_i,q+\xi_i}$. Thus, $\c^*\big(M_{p,q} \oplus M_{r,s}\big)
\cong M_{p,q} \oplus M_{r,s}$ for all $\gamma \in \G$ if and only if  $\{p,q,r,s\}$ is an $E[2]$-coset.

\subsection{}
Recall that $Q'=Q\otimes M_2(k)$. The next result follows from \Cref{pr.desc}.

\begin{proposition}\label{pr.descent_lines_wtQ}
The function $M \mapsto M^\G$ is a bijection from isomorphism classes of $\G$-equivariant $Q'$-modules with Hilbert series $4(1-t)^{-2}$ to isomorphism classes of $\wtQ$-line modules.  
\end{proposition}

By Morita equivalence, a $\G$-equivariant $Q'$-module $M$ with Hilbert series $4(1-t)^{-2}$ is isomorphic to 
$N\otimes k^2$ for some $Q$-module $N$ with Hilbert series $2(1-t)^{-2}$ (a ``fat line'' of multiplicity two over $Q$). Moreover, by the remark in \Cref{ssect.twisting.2}, 
the equivariance ensures/requires that the isomorphism class of $M$ is invariant under translation by the 2-torsion subgroup. 

The main ingredient in constructing $\wtQ$-lines will be $Q$-modules 
with Hilbert series $2(1-t)^{-2}$. The obvious  such modules are those of the form 
$M_{p,q}\oplus M_{r,s}$ where the invariance condition requires $\{p,q,r,s\}$ to be an $E[2]$-coset. 
\Cref{pr.lines_wtQ} will provide the examples announced in \Cref{ssect.lines1}.

\begin{lemma}
\label{lem.M.xi.1}
Let  $x,y \in E/E[2]$ and let $\xi$ and $\omega$ be 2-torsion points. Define
\begin{equation}
\label{defn.M.x.xi}
M_{x,\xi} \; := \; \big(M_{p,p+\xi} \oplus M_{p+\xi',p+\xi''}\big) \otimes k^2
\end{equation}
where $p$ is any point in $E$ such that $x=p+E[2]$. 
\begin{enumerate}
  \item 
  The $Q'$-module $M_{x,\xi}$ does not depend on the choice of $p$. 
  \item 
$M_{x,\xi} \cong M_{y,\omega}$ if and only if $(x,\xi)=(y,\omega)$.  
  \item
The map $\Phi:k^\times \times k^\times \to \Aut_{Q'}\big(M_{x,\xi}\big)$, $\Phi(\l,\l')(m,m'):=(\l m,\l'm')$, is an isomorphism.  
\end{enumerate}
\end{lemma}
\begin{proof}
Let $E[2]=\{o,\xi,\xi',\xi''\}$. 

(1) 
Suppose $x$ is also the image of $q \in E$. Since $\xi'+\xi''=\xi$,  
$$
\big\{\{q,q+\xi\}, \{q+\xi',q+\xi''\}\big\} \; = \; \big\{\{p,p+\xi\}, \{p+\xi',p+\xi''\}\big\}. 
$$
Therefore $M_{p,p+\xi} \oplus M_{p+\xi',p+\xi''}=M_{q,q+\xi} \oplus M_{q+\xi',q+\xi''}$. Hence $M_{x,\xi}$
does not depend on the choice of $p$. In particular, if $(x,\xi)=(y,\omega)$, then $M_{x,\xi} = M_{y,\omega}$. 

(2) 
Suppose that the $Q'$-modules $M_{x,\xi}$ and $M_{y,\omega}$ are isomorphic. Let $q\in E$ be
be such that $y=q+E[2]$. By Morita equivalence, there is an isomorphism of $Q$-modules
$$
 M_{p,p+\xi} \oplus M_{p+\xi',p+\xi''}  \; \cong \; M_{q,q+\omega} \oplus M_{q+\omega',q+\omega''} 
 $$
 where $E[2]=\{o,\omega,\omega',\omega''\}$. Since isomorphism classes of line modules for $Q$ are in natural bijection with effective divisors of degree 2 on $E$,
 $$
 \big\{\{q,q+\omega\}, \{q+\omega',q+\omega''\}\big\} \; = \; \big\{\{p,p+\xi\}, \{p+\xi',p+\xi''\}\big\}. 
 $$
It follows immediately from this equality that $q+E[2]=p+E[2]$, i.e., $x=y$. Since $\omega$ can be recovered from $ \big\{\{q,q+\omega\}, \{q+\omega',q+\omega''\}\big\} $
as the difference between the elements in $\{q,q+\omega\}$ and also as the difference between the elements in $\{q+\omega',q+\omega''\}$, it follows that $\omega=\xi$.

(3)
Every line module for $Q$ is cyclic so its graded automorphism group is isomorphic to $k^\times$, each $\l \in k^\times$ acting on
the line module by scalar multiplication.

By Morita equivalence,  $\Aut_{Q'}(M_{x,\xi}) = \Aut_{Q}(M_{p,p+\xi}\oplus M_{p+\xi',p+\xi''}) \cong k^\times \times k^\times$
where the isomorphism is because $M_{p,p+\xi} \not\cong M_{p+\xi',p+\xi''}$. An automorphism $(\l,\l') \in (k^\times)^2$ acts on 
$M_{x,\xi} = (M_{p,p+\xi} \otimes k^2) \oplus (M_{p+\xi',p+\xi''}\otimes k^2)$ as multiplication by $\l$ on the first summand and 
multiplication by $\l'$ on the second summand.
\end{proof}

\begin{lemma}
\label{lem.M.xi.2}
Let $E[2]=\{o,\xi,\xi',\xi''\}$. Let $x \in E/E[2]$ and write $M=M_{x,\xi}$. 
\begin{enumerate}
  \item 
  If $\gamma \in \G$, then $\gamma^* M \cong M $ as $Q'$-modules.
  \item
  If $\gamma \in \G$ and $a \in \Aut_{Q'}(M)$, then there is a unique element $\c  \triangleright a  \in \Aut_{Q'}(M)$ such that
 \begin{equation*}
  \begin{tikzpicture}[auto,baseline=(current  bounding  box.center)]
    \path[anchor=base] (0,0) node (1) {$M$} +(2,0) node (2) {$\gamma^*M$} +(0,-2) node (3) {$M$} +(2,-2) node (4) {$\gamma^*M$};
    \draw[->] (1) to node[pos=.5] {$\scriptstyle \varphi_\gamma$} (2);
    \draw[->] (3) to node[pos=.5,swap] {$\scriptstyle \varphi_\gamma$} (4);    
    \draw[->] (1) to node[pos=.5,swap] {$\scriptstyle \c\triangleright a$} (3);    
    \draw[->] (2) to node[pos=.5] {$\scriptstyle \gamma^*(a)$} (4);
  \end{tikzpicture}  
\end{equation*}
commutes for all isomorphisms $\varphi_\gamma:M  \to \gamma^*M$.
\item{}
The map $(\c,a) \mapsto \c  \triangleright a$ defines a left action of $\G$ on $ \Aut_{Q'}(M)$.
\item{}
If we identify $k^\times \times k^\times$ with $\Aut_{Q'}\big(M_{x,\xi}\big)$ via the isomorphism $\Phi$ in \Cref{lem.M.xi.1}, then the 
$\G$-action on $ \Aut_{Q'}(M)$ is
$$
\xi  \triangleright  (\l,\l')=(\l,\l')
\qquad \hbox{and} \qquad
\xi'  \triangleright  (\l,\l')=\xi''  \triangleright  (\l,\l')=(\l',\l)
$$
for all $(\l,\l') \in k^\times \times k^\times$.
\end{enumerate}
\end{lemma}
\begin{proof}
Let $p \in E$ be such that $x=p+E[2]$. Thus $M=\big(M_{p,p+\xi} \oplus M_{p+\xi',p+\xi''}\big) \otimes k^2$. 

(1) 
This follows from the remark in \Cref{ssect.twisting.2}.

(2)
Choose an isomorphism $\varphi_\gamma:M  \to \gamma^*M$. Define $ \c  \triangleright a:=\varphi_\gamma^{-1} \gamma^*(a) \varphi_\gamma$.
Certainly the diagram commutes. If $\psi_\gamma:M\to \gamma^*M$ is another isomorphism, then 
$\psi_\gamma$ is a multiple of  $\varphi_\gamma$ by an element in $\Aut_{Q'}(\c^*M)$. But $\Aut_{Q'}(\c^*M)$ is abelian so 
$\psi_\gamma^{-1} \gamma^*(a) \psi_\gamma=\varphi_\gamma^{-1} \gamma^*(a) \varphi_\gamma$.

 (3)
 This is standard. See, for example, \Cref{lem.action}.
 
 (4)
 By \Cref{prop.twist}, $\xi^*M_{p,p+\xi} \cong M_{p,p+\xi}$ and $\xi^*M_{p+\xi',p+\xi''} \cong M_{p+\xi',p+\xi''}$ so $\varphi_\xi$
 preserves the summands $M_{p,p+\xi}\otimes k^2$ and $M_{p+\xi',p+\xi''}\otimes k^2$. Therefore $\xi$ acts on $(k^\times)^2$ trivially. 
 On the other hand, $(\xi')^*M_{p,p+\xi} \cong (\xi'')^* M_{p,p+\xi} \cong M_{p+\xi',p+\xi''}$ so $\xi'$ and $\xi''$ act on $(k^\times)^2$ by switching the
 two components. 
 \end{proof}

A $\G$-equivariant structure on a $Q'$-module
$M$  is the same thing as a left $Q'$-module $M$ endowed with a left action $\G \times M \to M$, 
$(\c,m) \mapsto m^\c$, such that $(xm)^\c=\c(x)m^\c$ for all $x \in Q'$, $m \in M$, and $\c \in \G$. 
We adopt this point of view several times in the rest of this section.

Recall that the action of $\G$ as automorphisms of $Q'$ is defined in terms of the actions of  $\G$ as automorphisms of $Q$ and $M_2(k)$
(see \Cref{ssect.quat.basis}). 

\begin{lemma}
\label{lem.equiv.structures}
Let $N$ be a graded left $Q$-module that is generated by $N_0$. 
The function that sends a $\G$-equivariant structure $\{\varphi_\c:N \otimes k^2  \longrightarrow \gamma^*(N \otimes k^2) \; | \; \c \in \G\}$  on the $Q'$-module $N \otimes k^2$ to the  $\G$-equivariant structure $\{\varphi_\c\big\vert_{N_0 \otimes k^2}:  N_0 \otimes k^2  \longrightarrow \gamma^*(N _0 \otimes k^2)  \; | \; \c \in \G\}$ on the $M_2(k)$-module  $N_0 \otimes k^2\,$ is injective. 
\end{lemma}
\begin{proof}
Certainly, if the maps $\{\varphi_\c:N \otimes k^2  \longrightarrow \gamma^*(N \otimes k^2) \; | \; \c \in \G\}$  
 give $N\otimes k^2$ the structure of a $\G$-equivariant $Q'$-module, then their restrictions to the degree zero components give $N_0\otimes k^2$ the structure of a $\G$-equivariant $M_2(k)$-module.
 
Since $Q'$ is generated as an algebra by $Q_0'$ and $Q_1'$, 
the formula $(xm)^\c=\c(x)m^\c$ implies that the action of $\c$ on  $N_{n+1} \otimes k^2$ 
is completely determined by  the action of $\c$ on $N_{n} \otimes k^2$.
Thus, if two $\G$-equivariant structures on $N \otimes k^2$ agree on $N_0 \otimes k^2$, then they agree on $N$.
\end{proof}

\subsubsection{Warning} 
The result in \Cref{lem.equiv.structures} does {\it not} extend to a result saying that two equivariant structures on $N$ are isomorphic if and only
if their restrictions to $N_0 \otimes k^2$ are isomorphic. \Cref{prop.equiv.structures.isom} 
says that all $\G$-equivariant structures on $N_0 \otimes k^2$ are isomorphic to each other.

The group $\G$ acts as $k$-algebra automorphisms of $M_2(k)$, We fixed a basis for $k^2$ such that $\omega \in \G$ acts on $M_2(k)$ as conjugation by the quaternionic  basis element $q_\omega$ defined in \Cref{ssect.quat.basis}. We use that basis in the next result.

\begin{proposition}
\label{prop.equiv.structures.isom}
Fix $\zeta,\eta,\xi \in \G$ such that $q_\zeta,q_\eta,q_\xi$ is a cyclic permutation of $q_1,q_2,q_3$.
\begin{enumerate}
  \item 
  Let $\phi_\omega:M_2(k) \to M_2(k)$, $\omega \in \G$,  be the linear isomorphisms  that take the following values on the basis $1,q_\zeta,q_\eta,q_\xi$
  for $M_2(k)$:
\begin{table}[htdp]
\begin{center}
\begin{tabular}{|l||l|l|l|l|l|}
\hline
$\quad q$ & $\quad 1$ & $\quad  q_\zeta $ & $\quad q_\eta$ & $\quad  q_\xi$  
\\
\hline
\hline
$\phi_0(q)$ & $\quad 1$ & $\quad  q_\zeta $ & $\quad q_\eta$ & $\quad  q_\xi$ 
\\
\hline
$\phi_\zeta(q)$ & $\quad 1$ & $\quad  q_\zeta $ & $\, -q_\eta$ & $\,  -q_\xi$ 
\\
\hline
$\phi_\eta(q)$ &  $\quad 1$ & $\,  -q_\zeta $ & $\quad q_\eta$ & $\,  -q_\xi$ 
  \\
\hline
$\phi_\xi(q)$ & $\quad 1$ & $\,  -q_\zeta $ & $\, -q_\eta$ & $\quad  q_\xi$ 
  \\
\hline
\end{tabular}
\end{center}
\vskip .12in
\caption{Action of $\G$ on $M_2(k)$}
\label{table.20.pts}
\end{table}
\newline
The action of $\G$ on $M_2(k)$ given by the maps $\phi_\omega$, together with the action of $M_2(k)$ on $M_2(k)$ by left multiplication,
 gives $M_2(k)$ the structure of a $\G$-equivariant left $M_2(k)$-module.
  \item 
  Every $\G$-equivariant $M_2(k)$-module is isomorphic to a direct sum of copies of the $\G$-equivariant $M_2(k)$-module in (2).
  \item
  Let $V$ be a finite dimensional $\G$-equivariant $M_2(k)$-module. As a $\G$-module, $V$ is isomorphic to a direct sum of copies of
  the regular representation. If $\omega \in \{\zeta,\eta,\xi\}$, then the $(+1)$- and $(-1)$-eigenspaces for the action of $\omega$ on $V$ have dimension $\frac{1}{2}\dim_k(V)$.
\end{enumerate}
\end{proposition}
\begin{proof}
(1)
Whenever a group $\G$ acts as automorphisms of a ring $R$, $R$ viewed as left $R$-module via multiplication is a 
$\G$-equivariant $R$-module with respect to the action of $\G$ as automorphisms of $R$. The value of $\phi_\omega(q_{\omega'})$
in the table is $q_\omega q_{\omega'} q_\omega^{-1}$ so, by the previous sentence, this action of $\G$ makes $M_2(k)$ a $\G$-equivariant 
$M_2(k)$-module. 

(2)
By \Cref{le.descent}, there is an equivalence from the category of $\G$-equivariant $M_2(k)$-modules to the category of vector spaces, 
the functor implementing the equivalence being $M \rightsquigarrow M^\G$. Since $M_2(k)^\G\cong k$, the result follows.

Alternatively, a $\G$-equivariant left $M_2(k)$-module is the same thing as a left module over the skew group ring 
$M_2(k) \rtimes \G$ which has dimension16; the $\G$-equivariant $M_2(k)$-module in (2) is irreducible of dimension 4 so we conclude that
$M_2(k) \rtimes \G \cong M_4(k)$. The result follows. 

(3) follows from (2) because $M_2(k)$ is isomorphic as a $\G$-module to the regular representation. 
\end{proof}

\begin{theorem}
\label{pr.lines_wtQ}
Let $E[2]=\{o,\xi,\xi',\xi''\}$. Let $M$ be the $Q'$-module $(M_{p,p+\xi}\oplus M_{p+\xi',p+\xi'+\xi})\otimes k^2$.
\begin{enumerate}
  \item 
  There are exactly two $\G$-equivariant structures on $M$ up to isomorphism.
  \item 
   The group $H^1(\G,\Aut_{Q'}(M))$ acts simply transitively on this two-element set. 
  \item  
Up to isomorphism one equivariant structure is obtained from the other by interchanging the 
$(+1)$- and $(-1)$-eigenspaces for the action of $\xi$ on $M$ and simultaneously interchanging the $(+1)$- and $(-1)$-eigenspaces for the action of $\xi+\xi'$ on $M$, and leaving the $(+1)$- and $(-1)$-eigenspaces for the action of $\xi'$ unchanged. 
\end{enumerate}
\end{theorem}
\begin{proof}
If $x=p+E[2]$, then $M$ is the module $M_{x,\xi}$  in \Cref{lem.M.xi.1,lem.M.xi.2}.

{\bf Step 1: Existence of an equivariant structure.} 
Let $\varphi_\gamma:M \to \gamma^* M$, $\gamma \in \G$, be arbitrary $Q'$-module isomorphisms. An arbitrary choice of such isomorphisms
need not  give an equivariant structure on $M$; i.e., there is no reason the diagrams \Cref{eq.equiv_diagram} should commute. 
The failure of  \Cref{eq.equiv_diagram} to commute is measured by the elements
\begin{equation}\label{eq.obstr}
 a_{\gamma,\delta} \; := \;  \varphi_{\gamma\delta}^{-1}\circ t_{\gamma,\delta} \circ \gamma^*(\varphi_\delta)\circ \varphi_\gamma,\quad 
 \gamma,\delta \in \G
\end{equation}
in $\Aut_{Q'}(M)$ where $t_{\gamma,\delta}$ is as in \Cref{eq.equiv_diagram} and the right-hand side of \Cref{eq.obstr} is the clockwise composition of the automorphisms in \Cref{eq.equiv_diagram}. 

A tedious calculation (see \Cref{lem.2-cocycle}) shows that the function $(\gamma,\delta) \mapsto a_{\gamma,\delta}$ 
is a 2-cocycle for $\G$ valued in the $\G$-module $\Aut_{Q'}(M)\cong(k^\times)^2$ defined in \Cref{lem.M.xi.2}. 
Let $\xi' \in \G-\langle \xi\rangle$. Since $\Gamma=\langle \xi \rangle \times \langle \xi' \rangle$ 
it follows from the Hochschild-Serre spectral sequence 
\begin{equation}\label{eq.HS}
E_2^{a,b} = H^a(\langle \xi \rangle,H^b( \langle \xi' \rangle,(k^\times)^2))\Rightarrow H^{a+b}(\G,(k^\times)^2)
\end{equation}
and the cohomology of $\ZZ/2$ that $H^2(\G,(k^\times)^2)$ is trivial. Hence the obstruction cocycle $(a_{\delta,\gamma})$ 
is cohomologous to zero. Thus $(a_{\delta,\gamma})$ is the coboundary of some function $\G \to \Aut_{Q'}(M)$, 
$ \gamma\mapsto a_\gamma$;  the isomorphisms $\varphi_\gamma a_\gamma^{-1}$ now
form an equivariant structure on $M$.

{\bf Step 2: Classification of equivariant structures.} 
By Step 1, there is at least one $\G$-equivariant structure on $M$. Suppose the maps $\varphi_\gamma:M\to \gamma^*M$, $\gamma\in \G$,
 provide such an equivariant structure. 

Let $(\psi_{\gamma})_{\gamma\in\G}$ be another equivariant structure on $M$. Running through the compatibility conditions comprising equivariance, the maps $a_\gamma=(\varphi_\gamma)^{-1}\psi_\gamma$ can be seen to form a 1-cocycle of $\G$ valued in the $\G$-module $\Aut_{Q'}(M)\cong (k^\times)^2$. We similarly leave it to the reader to check that cocycles $(a_\gamma)$ and $(a'_\gamma)$ give rise to isomorphic equivariant structures 
\[
 \psi_\gamma = \varphi_\gamma a_\gamma \text{ and } \psi'_\gamma = \varphi_\gamma a'_\gamma 
\]  
if and only if they are cohomologous. In other words, the set of isomorphism classes of equivariant structures on $M$ is acted upon simply and transitively by $H^1(\G,(k^\times)^2)$. Using the Hochschild-Serre spectral sequence once more we get $H^1(\G,(k^\times)^2)\cong \bZ/2$ (see the proof of (3) below).  

This completes the proof of (1) and (2).

(3) 
The Hochschild-Serre spectral sequence yields an isomorphism
\begin{equation}\label{eq.spectral_vanish}
H^1(\G,\Aut(M)) \; \cong \; H^1\big(\langle \xi\rangle,H^0(\langle \xi'\rangle,(k^\times)^2)\big)\;  \oplus \; H^0\big(\langle \xi\rangle,H^1(\langle \xi'\rangle,(k^\times)^2)\big).
\end{equation}
Since $\xi'$ interchanges the two copies of $k^\times$, the $H^1$ term in the second summand vanishes so we are left with a natural isomorphism
\begin{equation*}
H^1(\G,\Aut(M)) \; \cong \; H^1(\langle \xi\rangle,k^\times) \; \cong \; \Hom_{\ZZ}(\langle \xi\rangle,k^\times),
\end{equation*}
where this time $k^\times$ is the diagonal subgroup of $\Aut_{Q'}(M)$. 

The function $f:\G \to \Aut_{Q'}(M)$ defined by $f(\xi)=f(\xi'+\xi)=(-1,-1)$  and $f(o)=f(\xi')=(1,1)$ is a 1-cocycle whose class $[f]$ in 
$H^1(\G,\Aut(M))$ is non-trivial. 
If the $Q'$-module isomorphisms $\{\phi_\c:M \to \gamma^*M \; | \; \c \in \G\}$ give $M$ a $\G$-equivariant structure, then the $\G$-equivariant structure on $M$ associated to the result of $[f]$ acting on the given equivariant structure is given by the isomorphisms
$\{\phi_\c\circ f(\c):M \to \gamma^*M \; | \; \c \in \G\}$. Recall that $\c^*M$ is $M$ as a graded vector space. The $(+1)$-eigenspace for the action of
$\xi$ on $M$ with equivariant structure $\{\phi_\c\}_{\c \in \G}$ is $\{m \in M \; | \; \phi_\xi(m)=m\}$ which is the $(-1)$-eigenspace for 
$\phi_\xi \circ f(\xi)$. Likewise, the $(-1)$-eigenspace for the action of $\phi_{\xi+\xi'} \circ f(\xi+\xi')$ is  the $(+1)$-eigenspaces for 
the action of $\phi_{\xi+\xi'}$. On the other hand, the eigenspaces for $\xi'$ are the same for both equivariant structures on $M_{x,\xi}$. 
\end{proof}

\subsubsection{}
There is a lack of symmetry in part (3) of \Cref{pr.lines_wtQ}: the eigenspaces for $\xi+\xi'$ are switched but those for $\xi'$ are not. The explanation is that the equivariant structure obtained by interchanging the eigenspaces for $\xi'$ but not $\xi+\xi'$ (but still exchanging the eigenspaces for $\xi$) is isomorphic to that obtained by switching the eigenspaces for $\xi+\xi'$ but not those for $\xi'$.

 \subsubsection{}
The proof of \Cref{pr.lines_wtQ} illustrates a familiar pattern in obstruction theory. The class of structures we are interested in, isomorphism 
classes of equivariant structures in this case, is a \define{pseudotorsor} over a cohomology group. Whether or not it is empty is controlled by an obstruction living in a cohomology group, $H^2$ for us, as in Step 1 of the proof, and when this obstruction vanishes the cohomology group 
of one degree lower, $H^1$ in our case, acts on the class of structures simply transitively.

\subsection{An explicit equivariant  structure on $M_{x,\xi}$}
\label{ssect.explicit.equivariant}

Let $\{\xi_1,\xi_2,\xi_3\}$ denote both the 2-torsion points on $E$ and the corresponding elements in $\G$, 
labelled so that the action of $\G$ as automorphisms of $M_2(k)$ is such that each $\xi_j$ acts as conjugation by the element $q_j$ in
(\ref{defn.a_i}). 

Let $p \in E$ and let $x=p+\langle \xi_1 \rangle \in E/\langle \xi_1 \rangle$.
Let $M=M_{x,\xi_1}=(M_{p,p+\xi_1} \oplus M_{p+\xi_2,p+\xi_3}) \otimes k^2$.
Fix a basis $e$ for the degree-zero component of  $M_{p,p+\xi_1}$ and a basis $e'$ for the degree-zero component of  $M_{p+\xi_2,p+\xi_3}$.

If $u={\,0 \choose \,1}$ and $v={\,1 \choose \,0}$, then
\begin{align*}
&   q_1u= -iu,  \qquad  q_2u= iv, \qquad q_3u=  -v, \\
&   q_1v=iv,   \qquad \quad q_2v=iu ,  \qquad q_3v=u.
\end{align*}

\begin{lemma}
\label{lem.Q-tilde-line-module}
Let $\b_0x_0+\b_1x_1+\b_2x_2+\b_3x_3$ be a linear form that vanishes at $p$ and $p+\xi_1$. Then
\begin{enumerate}
  \item 
the line through $p$ and $p+\xi_1$ is  $\b_0x_0+\b_1x_1=\b_2x_2+\b_3x_3=0$,
  \item 
the line through  $p+\xi_2$ and $p+\xi_3$ is    $\b_0x_0-\b_1x_1=\b_2x_2-\b_3x_3=0$, 
  \item{}
$\b_0y_0+i\b_1y_1$ and $i\b_2y_2+\b_3y_3$ annihilate $e \otimes u+ e' \otimes v$ and are linearly independent, and
  \item{}
$\b_0y_0-i\b_1y_1$ and $i\b_2y_2-\b_3y_3$ annihilate $e \otimes v+ e' \otimes u$ and are linearly independent.
\end{enumerate}
\end{lemma}
\begin{proof}
By \Cref{lem.pts.of.E}, at least three of the coordinate functions $x_0.x_1,x_2,x_3$ are non-zero at $p$. Thus $(\b_0,\b_1) \ne (0,0)$
and $(\b_2,\b_3) \ne (0,0)$. Therefore the equations in (2) and (3) really do define lines in $\PP(Q_1^*)$. 
It also follows that $\b_0y_0+i\b_1y_1$ and $i\b_2y_2+\b_3y_3$ are linearly independent.

(1)
Translation by $\xi_1$ leaves the set $\{p,p+\xi_1\}$ stable so 
$\xi_1(\b_0x_0+\b_1x_1+\b_2x_2+\b_3x_3)$ also vanishes at $p$ and $p+\xi_1$. Since $\xi_1(\b_0x_0+\b_1x_1+\b_2x_2+\b_3x_3)=\b_0x_0+\b_1x_1-\b_2x_2-\b_3x_3$, (1) follows.

(2)
Since translation by $\xi_2$ sends  $\{p,p+\xi_1\}$ to  $\{p+\xi_2,p+\xi_3\}$,
$\xi_2(\b_0x_0+\b_1x_1)$ and $\xi_2(\b_2x_2+\b_3x_3)$ vanish at $p+\xi_2$ and $p+\xi_3$.
Thus (2) is true. 

(3)
Since
\begin{align*}
y_0 \cdot(e \otimes u+e' \otimes v)  \; =\; (x_0 \otimes q_0)\cdot(e \otimes u+e' \otimes v) & \; =\; x_0 e \otimes u + x_0 e' \otimes v,
\\
y_1 \cdot(e \otimes u+e' \otimes v)  \; =\; (x_1 \otimes q_1)\cdot(e \otimes u+e' \otimes v) & \; =\; -i x_1 e \otimes u + ix_1 e' \otimes v,
\\
y_2 \cdot(e \otimes u+e' \otimes v)  \; =\; (x_2 \otimes q_2)\cdot(e \otimes u+e' \otimes v)  & \; =\; ix_2 e \otimes v + ix_2 e' \otimes u,  \; \; \hbox{and}
\\
y_3 \cdot(e \otimes u+e' \otimes v)  \; =\; (x_3 \otimes q_3)\cdot(e \otimes u+e' \otimes v) & \; =\; - x_3 e \otimes v +  x_3 e' \otimes u,
\end{align*} 
$(\b_0y_0+i\b_1y_1-i\b_2y_2-\b_3y_3)\cdot(e \otimes u+ e' \otimes v)$ equals
$$
 (\b_0x_0 + \b_1x_1)e \otimes u \; + \; (\b_2x_2+\b_3x_3) e \otimes v \;  + \; (\b_2x_2-\b_3x_3)e'\otimes u \; + \; (\b_0x_0- \b_1x_1)e' \otimes v.
$$
Since $e \in (M_{p,p+\xi_1})_0$ it follows from (1) that $ (\b_0x_0 + \b_1x_1)e =(\b_2x_2+\b_3x_3) e=0$.
Since $e' \in (M_{p+\xi_2,p+\xi_3})_0$ it follows from (2) that $ (\b_0x_0 - \b_1x_1)e' =(\b_2x_2-\b_3x_3) e'=0$.
Therefore (3) is true. The proof of (4) is similar.
\end{proof}

Let $\phi_0$ be the identity map on $M_0$ and let $\phi_1,\phi_2 \in \GL(M_0)$ be the linear automorphisms which act on the basis 
$\{e \otimes u,e \otimes v,e'\otimes u,e'\otimes v\}$ as in Table \ref{Gamma.action2}.
\begin{table}[htdp]
\begin{center}
\begin{tabular}{|l||l|l|l|l|l|}
\hline
 & $e \otimes u$ & $\phantom{-}e \otimes v$ & $\phantom{-}e'\otimes u$ & $e'\otimes v$
\\
\hline
\hline
$\phi_1$ & $e \otimes u$ & $-e \otimes v$ & $-e'\otimes u$ & $e'\otimes v$
\\
\hline
$\phi_2$ &  $e' \otimes v$ & $\phantom{-}e' \otimes u$ & $\phantom{-}e\otimes v$ & $e\otimes u$ 
\\
\hline
\end{tabular}
\end{center}
\vskip .12in
\caption{Equivariant structure on $M_0$}
\label{Gamma.action2}
\end{table}

\noindent
Let $\phi_3=\phi_1\phi_2$. 

The following observation is elementary. 

\begin{lemma}
\label{lem.Z2.equiv.structure} 
Let $a$ be an element in a 
ring $R$ such that $a^2=1$. There is a group homomorphism $\ZZ/2 \to \Aut(R)$ given by sending the non-identity element to the automorphism
$b \mapsto aba^{-1}$. Let $M$ be a left $R$-module and define the group homomorphism $\ZZ/2 \to \Aut_\ZZ(M)$ by sending the non-identity element to the automorphism $m \mapsto am$. This action of  $\ZZ/2$ makes $M$ a $\ZZ/2$-equivariant $R$-module.
\end{lemma}

\begin{theorem}
\label{thm.lines-and-line-modules}
Let each $\xi_i$ act on $M_0$ as the linear map $\phi_i$ in Table \ref{Gamma.action2}.
\begin{enumerate}
  \item 
  This action of $\G$ on $M_0$  extends to an action of $\G$ on $M$ that makes $M$ a  $\G$-equivariant $Q'$-module. 
  \item 
  The $\wtQ$-line module $M^\G$ is generated by $e \otimes u+ e' \otimes v$.
  \item{}
  If $\b_0x_0+\b_1x_1 =\b_2x_2+\b_3x_3=0$ is the line in $\PP(Q_1^*)$ that passes through $p$ and $p+\xi_1$, then
  the line in $\PP(\wtQ_1^*)$ corresponding to $M^\G$ is $\b_0y_0+i\b_1y_1=i\b_2y_2+\b_3y_3=0$. 
\end{enumerate}
\end{theorem}
\begin{proof}
(1)
We will use \Cref{lem.Z2.equiv.structure} to show that $M_0$ is a $\G$-equivariant $M_2(k)$-module.

First, consider the action of $\xi_1$ by $\phi_1$ on $e \otimes k^2$. With respect to the ordered basis $\{e \otimes u,e\otimes v\}$,
$\xi_1$ acts on $e \otimes k^2$ as multiplication by $1 \otimes {1 \;\phantom{-} 0 \choose 0 \;-1}$. The action of $\xi_1$ on $M_2(k)$ is 
$b \mapsto q_1bq_1^{-1}$. Since conjugation by $q_1$ is the same as conjugation by ${1 \;\phantom{-} 0 \choose 0 \;-1}$, \Cref{lem.Z2.equiv.structure} tells us that $e \otimes k^2$ is a $\langle \xi_1 \rangle$-equivariant $M_2(k)$-module.

Now consider the action of  $\xi_1$ by $\phi_1$ on $e' \otimes k^2$. With respect to the ordered basis $\{e \otimes u,e\otimes v\}$,
$\xi_1$ acts on $e' \otimes k^2$ as multiplication by $1 \otimes {-1\; \;0 \choose \phantom{-} 0\; \;1}$.  Since conjugation by $q_1$ is the same as conjugation by $ {-1 \;\;0 \choose \phantom{-} 0\; \;1}$, \Cref{lem.Z2.equiv.structure} tells us that $e' \otimes k^2$ is a $\langle \xi_1 \rangle$-equivariant $M_2(k)$-module.

Thus, $M_0$ is $\langle \xi_1 \rangle$-equivariant $M_2(k)$-module. A similar argument shows that $M_0$ is a $\langle \xi_j \rangle$-equivariant $M_2(k)$-module for the other $j$'s.  Since $\{\phi_0,\phi_1,\phi_2,\phi_3\}$
is a subgroup of $\GL(M_0)$ isomorphic to $\G$, these $\ZZ/2$-equivariant structures fit together to make 
$M_0=(e \otimes k^2) \oplus (e' \otimes k^2)$ a $\G$-equivariant $M_2(k)$-module.

To extend the equivariant structure to all of $M$, simply define automorphisms $\phi_i$ of $M$ by 
\begin{equation*}
  \phi_i(am) = \xi_i(a)\phi_i(m),\quad \forall a\in Q',\ m\in M_0. 
\end{equation*}
That this action is well-defined boils down to checking that whenever $a\in Q'$ annihilates $m\in M_0$, $\xi_i(a)$ annihilates $\phi_i(m)$. For this it suffices to assume that $m$ is an eigenvector of $\phi_i$ (since $M_0$ breaks up as a direct sum of $\G$-eigenspaces), and hence to prove that
\begin{equation*}
  am = 0\Rightarrow \xi_i(a)m=0,\quad \forall a\in Q',\ m\in M_0.
\end{equation*}
The conclusion follows from the fact that all twists $\xi_i^*M$ are isomorphic to $M$ as $Q'$-modules (because we already know there are equivariant structures on $M$). 

(2) 
By \Cref{pr.descent_lines_wtQ}, $M^\G$ is a line module for $\wtQ$. One sees from \Cref{Gamma.action} that $e \otimes u+ e' \otimes v$
is in $M^\G_0$ so it generates the $\wtQ$-line module $M^\G$. 

(3)
The correspondence between line modules for $\wtQ$ and lines in  $\PP(\wtQ_1^*)$ is given by sending a line module $\wtQ/\wtQ y+\wtQ y'$
to the line $y=y'=0$. Thus, (3) follows from \Cref{lem.Q-tilde-line-module}(3).
\end{proof}

\subsection{3 elliptic curves parametrizing some line modules}

Let $\GG(1,3)$ be the Grassmannian of lines in $\PP(\wtQ_1^*)$. There is a bijection
$$
\GG(1,3) \; \longleftrightarrow \; \{\hbox{isomorphism classes of cyclic graded $\wtQ$-modules with Hilbert series $1+2t$}\}
$$
given by the function sending a line $y=y'=0$ to the module $\wtQ/\wtQ y+\wtQ y' + \wtQ_{\ge 2}$ and its inverse which sends a  cyclic graded 
$\wtQ$-module $N$ with Hilbert series $1+2t$ to the vanishing locus of the subspace of $\wtQ_1$ that annihilates $N_0$.  

Let $L$ be a line module for $\wtQ$. The Hilbert series for $L/L_{\ge 2}$  is $1+2t$ so $L$ determines a point in $\GG(1,3)$. Since $L \cong \wtQ/\wtQ y + \wtQ y'$ for some linearly independent elements $y,y' \in \wtQ_1$, the isomorphism class of $L$ is determined by the isomorphism class of $L/L_{\ge 2}$.
Thus, there is a well-defined map 
$$
 \{\hbox{isomorphism classes of line modules for $\wtQ$}\}  \; \longrightarrow \; \GG(1,3).
 $$

\begin{proposition}
\label{prop.G13}
Let $g: \PP(Q_1^*) \to \PP(\wtQ_1^*)$ be the isomorphism induced by the linear  isomorphism $\wtQ_1 \to Q_1$,
$$
y_0 \mapsto x_0, \qquad y_1 \mapsto -ix_1, \qquad y_2 \mapsto -ix_2, \qquad  y_3 \mapsto x_3.
$$
The function $f:E/\langle \xi_1\rangle \to \GG(1,3)$ defined by
$$
f\big(p+\langle \xi_1\rangle\big) := g(\hbox{the line in $\PP(Q_1^*)$ that passes  through $p$ and $p+\xi_1$})
$$
is a closed immersion and $f\big(E/\langle \xi_1\rangle\big)$ 
parametrizes the isomorphism classes of $\G$-equivariant $Q'$-modules of the form $M_{x,\xi_1}$, $x \in E/E[2]$.  
If $x=p+E[2]$, then the lines  $f\big(p+\langle \xi_1\rangle\big)$ and $ f\big(p+\xi_2+\langle \xi_1\rangle\big)$
correspond to the two non-isomorphic equivariant  structures on $M_{x,\xi_1}$. 
\end{proposition} 
\begin{proof}
The map that sends a point $p \in E$ to the line through $p$ and $p+\xi_1$ is a morphism from $E$ to the Grassmanian of lines in 
$\PP(Q_1^*)$. Composing that map with $g$ gives a morphism $h:E \to \GG(1,3)$. Since $h(p)=h(p+\xi_1)$, $h$ factors as a composition 
\begin{equation}
\label{E.to.G}
E  \longrightarrow E/\langle \xi_1\rangle \longrightarrow \GG(1,3)
\end{equation}
where the first map is the quotient map and the second is $f$. By the universal property of the quotient map, $f$ is a morphism. In fact, $f$ is the composition $\c\b$ of the two maps from \Cref{pr.S2E-G,le.E-S2E} and hence is a closed immersion.  

The line in $\PP(Q_1^*)$ through $p$ and $p+\xi_1$ is of the form  $\b_0x_0+\b_1x_1=\b_2x_2+\b_3x_3=0$. Therefore $f\big(p+\langle \xi_1\rangle\big)$ is the line $g(\b_0x_0+\b_1x_1)=g(\b_2x_2+\b_3x_3)=0$, i.e., the line
$i\b_0y_0-\b_1y_1=\b_2y_2-i\b_3y_3=0$. Thus, $f(p+\langle \xi_1\rangle)$ is the line in $\PP(\wtQ_1^*)$
that corresponds to the $\wtQ$-line module, $M^\G$, that corresponds to the $\G$-equivariant structure on $M=M_{x,\xi_1}$ with the 
equivariant structure described in \Cref{thm.lines-and-line-modules}.
\end{proof}

There are versions of all the results in \Cref{ssect.explicit.equivariant} with $\xi_2$ and $\xi_3$ in place of $\xi_1$. 
In particular, by \Cref{prop.G13} there are morphisms $E/\langle \xi_1 \rangle \to \GG(1,3)$, $E/\langle \xi_2 \rangle \to \GG(1,3)$, 
and $E/\langle \xi_3 \rangle \to \GG(1,3)$. It is easy to see that these morphisms are injective but we have not yet shown that the images are 
smooth. It is clear that the images of these morphisms are disjoint from one another. 

\begin{theorem}
The set of $\Gamma$-equivariant $Q'$-modules in \Cref{pr.lines_wtQ} is parametrized by 
$$
\big( E/\langle \xi \rangle\big) \, \sqcup \, \big( E/\langle \xi' \rangle\big) \, \sqcup \, \big( E/\langle \xi'' \rangle\big)
$$
where $\{\xi,\xi',\xi''\}$ is the set of 2-torsion points on $E$. 
\end{theorem}

In fact, we can say more about these three components of the scheme of line modules. We will say that a closed subscheme of a projective space $\PP^N$ is {\define spatial} if its inclusion factors through some linear $\PP^3\subset \PP^N$ but not through a linear $\PP^2\subset\PP^N$.

\begin{proposition}\label{pr.spatial_deg4}
  For each 2-torsion point $\xi$ the elliptic curve $E/\langle\xi\rangle\subset \GG(1,3)\subset \PP^5$ is spatial of degree four. 
\end{proposition}
\begin{proof}
That $E/\langle\xi\rangle$ is contained in a $\PP^3\subset\PP^5$ follows from its construction in \Cref{prop.G13}. Indeed, suppose in order to fix notation that $\xi=\xi_1$ and denote $\overline{E}=E/\langle\xi\rangle$. If the Pl\"ucker coordinates of the line 
\begin{equation*}
  \sum_{j=0}^3 \lambda_jy_j=\sum_{j=0}^3\lambda'_jy_j=0
\end{equation*}
are the minors $M_{ij}$, $0\le i<j\le 3$ of the matrix 
\begin{equation*}
  M = \begin{pmatrix}\lambda_0&\lambda_1&\lambda_2&\lambda_3\\\lambda'_0&\lambda'_1&\lambda'_2&\lambda'_3\end{pmatrix}
\end{equation*}
supported on columns $i$ and $j$, then the two coordinates $M_{01}$ and $M_{23}$ vanish on $\overline{E}$ by part (3) of \Cref{thm.lines-and-line-modules}.

The fact that $\overline{E}$ is not contained in a $\bP^2$ will follow once we prove that the degree of the embedding into $\PP^5$ is four, as claimed in the statement. 

To check the degree assertion we will intersect $\overline{E}$ with a hyperplane section of $\GG(1,3)\subset \PP^5$, judiciously chosen so that it is not tangent to $\overline{E}$ and the number of intersection points is clearly four.

For every line $\ell$ in $\PP^3$ the collection of all lines in $\GG(1,3)$ intersecting $\ell$ is a hyperplane section $H_\ell$ of $\GG(1,3)\subset \PP^5$. Let $\ell=\overline{pq}$ be a secant line of $E$. The points in $\overline{E}\cap H_\ell$ are the classes modulo $\langle \xi\rangle$ of those $u\in E$ for which the secant line $\overline{u(u+\xi)}$ intersects $\ell$. 

If 
\begin{equation}\label{eq.gen_cond}
  q\ne p+\xi\quad \text {and}\quad 3p+q+\xi\ne 0,\ p+3q+\xi\ne 0 
\end{equation}
then there are exactly four such classes modulo $\langle\xi\rangle$, namely those of $p$, $q$, $u$ and $u+\xi'$, where $u+(u+\xi)+p+q=0$ and $E[2]-\{0\}\ni \xi'\ne \xi$. 

It remains to check that $p,q\in E$ can be chosen so that $H_\ell$ is not tangent to $\overline{E}$ at any of the four points where they intersect, in addition to satisfying \Cref{eq.gen_cond}. 

Identify, as usual, the tangent space to $\GG(1,3)$ at some line $m$ (simultaneously regarded as a 2-plane in the 4-dimensional vector space $V$) with the space of linear maps $m\to V/m$. Generally, we will conflate linear subspaces of $V$ and their projectivized versions.  

For any $u\in E$, the tangent line to $\overline{E}\subset \GG(1,3)$ at $\overline{u(u+\xi)}$ can be identified with the space of linear maps $\overline{u(u+\xi)}\to V/\overline{u(u+\xi)}$ that send the lines $u$ and $u+\xi$ in $V$ to the 2-planes $T_uE$ and $T_{u+\xi}E$ in $V$ respectively modulo $\overline{u(u+\xi)}$. 

On the other hand, reverting to the notation introduced above for $u\in E$ so that $2u+\xi+p+q=0$, the tangent space at $\overline{u(u+\xi)}\in\GG(1,3)$ to $H_\ell$ consists of those linear maps $\overline{u(u+\xi)}\to V/\overline{u(u+\xi)}$ that send the intersection $s=\overline{pq}\cap\overline{u(u+\xi)}$ to $\overline{pq}$ modulo $\overline{u(u+\xi)}$ (see e.g. \cite[Example 16.6]{H92}).  

Since the line $s\subset V$ is in the span of $u$ and $u+\xi$, we would be certain that the tangent space in the previous paragraph does not contain the tangent line described two paragraphs up if we knew that the tangents to $E$ at $u$ and $u+\xi$ are coplanar. This is indeed the case if $4u=0$, so simply take $u\in E[4]$ and afterwards select $p$ and $q$ so that \Cref{eq.gen_cond} holds. 
\end{proof}

\subsubsection{}
There is another  perspective on the $\G$-equivariant $Q'$-modules parametrized by $E/\langle \xi \rangle$.
The family of $Q'$-modules $M_{x,\xi}$ is parametrized by $x \in E/E[2]$. 
The quotient of the fundamental groups, $\pi_1(E/E[2])/\pi_1(E/\langle \xi \rangle)$, which is naturally isomorphic to 
$E[2]/\langle \xi \rangle$, acts freely and transitively on each fiber of the natural map $E/\langle \xi \rangle \to E/E[2]$.
If we identify the fiber over $x$ with the set of isomorphism classes of equivariant structures on $M_{x,\xi}$, then 
$H^1(\Gamma,\Aut(M_{x,\xi}))$ also acts on the fiber over $x$. 
As the paragraph explains, these actions of  $E/\langle \xi \rangle$ and $H^1(\Gamma,\Aut(M_{x,\xi}))$ on the fibers are 
compatible in a natural way.

The Weil pairing $\langle \cdot,\cdot \rangle:E[2] \times E[2] \to \mu_2= \{\pm 1\} \subseteq k^\times$ is a non-degenerate skew-symmetric bilinear form on $E[2]$ viewed as a 2-dimensional vector space over $\FF_2$. 
Since $\langle \xi,\xi \rangle =1$,
there is an induced non-degenerate bilinear map $\langle \xi \rangle \times E[2]/ \langle \xi \rangle \to \mu_2$ or, 
what is essentially the same thing,  a group isomorphism 
$$
E[2]/\langle \xi \rangle \; \longrightarrow \; \Hom_\ZZ( \langle \xi \rangle,\mu_2) \; = \; \Hom_\ZZ( \langle \xi \rangle,k^\times)
\; \cong \;  H^1(\Gamma,\Aut(M_{x,\xi}))
$$
where the right-most isomorphism was established in the proof of \Cref{pr.lines_wtQ}(3).

 \subsection{}
 Under quite general conditions, which $\wtQ$ satisfies,  
Shelton and Vancliff prove that every irreducible component of the scheme parametrizing the line modules has dimension $\ge 1$
  \cite[Cor.2.6]{ShV02} and that every point module is a quotient of a line module \cite[Prop.3.1]{ShV02}.
  We will investigate this relationship in a subsequent paper. We also show there that the line modules for $\wtQ$ described above 
  are {\it not} all the line modules.


  \appendix
  
  \section{Equivariant structures}

 \subsection{Groups acting on categories}
 
 An action of a group $\G$ on a category $\sC$ consists of data $\{\a^*,t_{\a,\b} \; | \; \a,\b \in \G\}$ where each $\a^*:\sC \to \sC$ 
 is an auto-equivalence and each $t_{\a,\b}:\a^* \b^* \to (\a\b)^*$ is a natural isomorphism such that the diagrams
 $$
 \xymatrix{
 \a^* \circ \b^* \circ \c^* \ar[rr]^{\a^* \cdot t_{\b,\c}} \ar[d]_{t_{\a,\b}\cdot \c^*} && \a^* \circ (\b\c)^* \ar[d]^{t_{\a,\b\c}}
 \\
 (\a\b)^*\circ \c^* \ar[rr]_{t_{\a\b,\c}} && (\a\b\c)^*
 }
 $$
 commute for all $\a,\b,\c \in \G$.

 \begin{lemma}
 \label{lem.action}
 Let $x \in {\sf{Ob}}(\sC)$ and $\phi=\{ \phi_\a:x \to \a^* x \; | \; \a \in \G\}$ a set of isomorphisms.
  If $\Aut(x)$ is abelian, then there is an action of $\G$ on $\Aut(x)$ given by the formula
 \begin{align*}
\G \times \Aut(x)  & \;  \to  \;  \Aut(x) 
\\
 (\a,f) & \;  \mapsto \; \a\cdot f := \phi_\a^{-1}  \a^*(f)  \phi_\a. 
\end{align*}
This action   does not depend on the choice of the $\phi_\a$'s.
\end{lemma}
\begin{proof} 
 Because $t_{\a,\b}:\a^*\circ \b^* \to (\a\b)^*$ is a natural transformation, the diagram
 $$
 \xymatrix{
 \a^*(\b^*x) \ar[rr]^{(t_{\a,\b})_x} \ar[d]_{\a^*\b^*(f)} && (\a\b)^* x   \ar[d]^{(\a\b)^*(f)}
 \\
 \a^*(\b^*x)  \ar[rr]_{(t_{\a,\b})_x} &&  (\a\b)^*x
 }
 $$
 commutes for all $f \in \Aut(x)$ and all $\a,\b \in \G$. In other words,
 \begin{equation}
 \label{eq.nat.trans}
( \a\b)^*(f) =  (t_{\a,\b})_x \circ \a^*\b^*(f) \circ (t_{\a,\b})_x^{-1}.
  \end{equation}
  
  Since $\Aut(\a^*\b^*x)$ is abelian, $(t_{\a,\b})_x^{-1}  \phi_{\a\b} \phi_\a^{-1}  \a^*(\phi_\b)^{-1}$  commutes with $\a^* \b^*(f)$.
This fact can be expressed as 
$$
 \phi_\a^{-1}  \a^*( \phi_\b^{-1})  \circ \a^* \b^*(f) \circ \a^*(\phi_\b)  \phi_\a =  \phi_{\a\b}^{-1}   (t_{\a,\b})_x \circ \a^*\b^*(f) \circ (t_{\a,\b})_x^{-1}  \phi_{\a\b}
 $$
 which we re-write as 
 \begin{equation}
 \label{eq.lem.action}
 \phi_\a^{-1}  \a^*\Big( \phi_\b^{-1}  \b^*(f) \phi_\b\Big)  \phi_\a =  \phi_{\a\b}^{-1}   (t_{\a,\b})_x \circ \a^*\b^*(f) \circ (t_{\a,\b})_x^{-1}  \phi_{\a\b}.
\end{equation}
 The left-hand side of (\ref{eq.lem.action}) is $ \phi_\a^{-1}  \a^*(\b\cdot f)  \phi_\a = \a\cdot(\b\cdot f)$  and, by (\ref{eq.nat.trans}),  
 the right-hand side of (\ref{eq.lem.action})  is equal to
  $$
  \phi_{\a\b}^{-1}    (\a\b)^*(f)  \phi_{\a\b}
$$
which equals $(\a\b)\cdot f$. Thus $ \a\cdot(\b\cdot f) = (\a\b)\cdot f$.

To see that the action does not depend on the choice of the $\phi_\a$'s suppose that  $\{ \phi'_\a:x \to \a^* x \; | \; \a \in \G\}$ is another
collection of isomorphisms. There are automorphisms $\psi_\a \in \Aut(\a^* x)$ such that $\phi'_\a=\psi_\a\phi_\a $. The action of $\G$ on $\Aut(x)$
associated to the $\phi'_\a$, $\a \in \G$, is
$$
(\a,f) \; \mapsto \; ( \phi'_\a)^{-1}  \a^*(f) \phi'_\a \; = \;  \phi_\a^{-1} \psi_\a^{-1}  \a^*(f)  \psi_\a\phi_\a;
 $$
but $ \psi_\a^{-1}   \a^*(f)  \psi_\a =\a^*(f)$ because $\Aut(\a^*x)$ is abelian, so the right-hand side of the displayed equation is 
equal to $\a \cdot f$.
  \end{proof}
  
   \subsection{Equivariant objects}
 Suppose $\G$ acts on $\sC$. A {\sf $\G$-equivariant structure} on an object $x \in \sC$ is a set of isomorphisms 
 $\{\phi_\a:x \to \a^* x
 \; | \; \a \in \G\}$ such that the diagrams
  \begin{equation}
  \label{equiv.diag}
 \xymatrix{
 x\ar[rr]^{\phi_\a} \ar[d]_{\phi_{\a\b}} && \a^* x   \ar[d]^{\a^*(\phi_\b)}
 \\
 (\a\b)^*x && \ar[ll]^{(t_{\a,\b})_x}  \a^*(\b^*x)
 }
 \end{equation}
 commute for all $\a,\b,\c \in \G$. 
 
An arbitrary  set of  isomorphisms $\phi_\a:x \to \a^* x$, $\a \in \G$, will not usually give an equivariant structure on $x$.
Their failure to do so, i.e.,  the failure of (\ref{equiv.diag}) to commute,  is measured by the automorphisms 
 \begin{equation}
 \label{defn.a-alpha-beta}
 a_{\a,\b} \; := \; \phi^{-1}_{\a\b} \circ (t_{\a,\b})_x \circ \a^*(\phi_\b) \circ \phi_\a
\end{equation}
 of $x$. 
 
 \begin{lemma}
 \label{lem.2-cocycle}
 Let $x \in {\sf Ob}(\sC)$ and and let $\{\phi_\a:x \to \a^* x \; | \; \a \in \G\}$ be a set of isomorphisms. 
 If $\Aut(x)$ is abelian, then the function 
 $$
 a: \G \times \G \to \Aut(x),  \qquad (\a,\b)  \mapsto a_{\a,\b},
 $$
 is a 2-cocycle.
 \end{lemma}
 \begin{proof}
 We must show that 
 $
 a_{\a\b,\c}  \circ a_{\a,\b}=a_{\a,\b\c} \circ (\a\cdot a_{\b,\c})
 $
 for all $\a,\b,\c \in \G$.  
  
 First,  $a_{\a\b,\c}  \circ a_{\a,\b}$ equals
 \begin{align*}
 & \phi^{-1}_{\a\b\c} \circ (t_{\a\b,\c})_x \circ (\a\b)^*(\phi_\c) \circ \phi_{\a\b} 
 \circ \phi^{-1}_{\a\b} \circ (t_{\a,\b})_x \circ \a^*(\phi_\b) \circ \phi_\a
 \\
  \; = \;  & \phi^{-1}_{\a\b\c} \circ (t_{\a\b,\c})_x \circ (\a\b)^*(\phi_\c)  \circ (t_{\a,\b})_x \circ \a^*(\phi_\b) \circ \phi_\a
 \\
  \; = \;  & \phi^{-1}_{\a\b\c} \circ (t_{\a\b,\c})_x \circ  (t_{\a,\b})_{\c^*x} \circ  \a^*\b^*(\phi_\c)  \circ \a^*(\phi_\b) \circ \phi_\a
 \end{align*}
 where the last equality follows from the commutative diagram
 $$
 \xymatrix{
 \a^*\b^*x \ar[rr]^{(t_{\a,\b})_x} \ar[d]_{\a^*\b^*(\phi_\c)} && (\a\b)^* x   \ar[d]^{(\a\b)^*(\phi_\c)}
 \\
 \a^*\b^*(\c^*x)  \ar[rr]_{(t_{\a,\b})_{\c^*x}} &&  (\a\b)^*(\c^*x)
 }
 $$
 which exists by virtue of the fact that $t_{\a,\b}$ is a natural transformation (applied to the isomorphism $\phi_\c:x \to \c^* x$).

 On the other hand, $a_{\a,\b\c} \circ (\a\cdot a_{\b,\c})$ equals
  \begin{align*}
 & \phi^{-1}_{\a\b\c} \circ (t_{\a,\b\c})_x \circ \a^*(\phi_{\b\c}) \circ \phi_{\a} 
 \circ \phi^{-1}_{\a} \circ \a^*\big(\phi^{-1}_{\b\c} \circ (t_{\b,\c})_x \circ \b^*(\phi_\c) \circ \phi_\b\big) \circ \phi_\a
 \\
  \; = \;  & \phi^{-1}_{\a\b\c} \circ (t_{\a,\b\c})_x \circ  \a^*((t_{\b,\c})_x) \circ \a^*\b^*(\phi_\c) \circ \a^*(\phi_\b) \circ \phi_\a
 \\
  \; = \;  &  \phi^{-1}_{\a\b\c} \circ (t_{\a,\b\c})_x \circ  (\a^*\cdot t_{\b,\c})_x \circ \a^*\b^*(\phi_\c) \circ \a^*(\phi_\b) \circ \phi_\a
 \\
  \; = \;  &  \phi^{-1}_{\a\b\c} \circ (t_{\a\b,\c})_x \circ  (t_{\a,\b})_{\c^*x} \circ  \a^*\b^*(\phi_\c) \circ \a^*(\phi_\b) \circ \phi_\a
 \end{align*}
Thus,  $a_{\a,\b\c} \circ (\a\cdot a_{\b,\c})=a_{\a\b,\c}  \circ a_{\a,\b}$. 
 \end{proof}

\begin{proposition} 
Let $x \in {\sf Ob}(\sC)$ and suppose  $\Aut(x)$ is abelian. If the 2-cocycle $(\a,\b) \mapsto a_{\a,\b}$ defined in (\ref{defn.a-alpha-beta}) 
is the coboundary of the function  $f:\G \to \Aut(x)$, $\a \mapsto a_\a$, then the isomorphisms $\{\phi_\a a_\a^{-1}:x \to \a^*x \; | \; \a \in \G\}$
form an equivariant structure on $x$.
\end{proposition}
\begin{proof}
The hypothesis says that 
$$
\phi^{-1}_{\a\b} \circ (t_{\a,\b})_x \circ \a^*(\phi_\b) \circ \phi_\a \; = \;  (df)(\a,\b)  \; = \;   (\a\cdot a_\b) \circ a^{-1}_{\a\b} \circ a_\a
$$
for all $\a,\b \in \G$. Since $\Aut(x)$ is abelian, we can rewrite this as
\begin{align*}
\phi^{-1}_{\a\b} \circ (t_{\a,\b})_x \circ \a^*(\phi_\b) \circ \phi_\a   &\; = \; a^{-1}_{\a\b} \circ a_\a \circ (\a\cdot a_\b) 
\\
&\; = \; a^{-1}_{\a\b} \circ a_\a \circ \phi_\a^{-1} \a^*(a_\b) \phi_\a
\end{align*}
whence $(t_{\a,\b})_x \circ \a^*(\phi_\b)   = \phi_{\a\b}  a^{-1}_{\a\b} \circ a_\a\phi_\a^{-1}  \circ  \a^*(a_\b) $. In other words, the diagram
$$
\xymatrix{
 x\ar[rr]^{\phi_\a a_\a^{-1}} \ar[d]_{\phi_{\a\b} a_{\a\b}^{-1}} && \a^* x   \ar[d]^{\a^*(\phi_\b a_\b^{-1})}
 \\
 (\a\b)^*x && \ar[ll]^{(t_{\a,\b})_x}  \a^*(\b^*x)
 } 
 $$
 commutes; i.e.,  the maps $\{\phi_\a a_\a^{-1} :x \to \a^*x \; | \; \a \in \G\}$ form an equivariant structure on $x$.
 \end{proof}

 \subsection{Classification of equivariant structures}
 
 In order to classify equivariant structures we must first say what it means for two equivariant structures  to be the ``same''. 
 
 Suppose that $\G$ acts on $\sC$. The objects in the category $\sC^\G$ of {\sf $\G$-equivariant objects in $\sC$} are pairs $(x,\phi)$ 
 consisting of an object $x$ in $\sC$ and a set of isomorphisms $\phi=\{\phi_\a:x \to \a^*x \; | \; \a \in \G\}$ that give $x$ the structure of a $\G$-equivariant object. A morphism $f:(x,\phi) \to (y,\psi)$ in $\sC^\G$ is a morphism $f:x \to y$ in $\sC$ such that the diagram
$$
\xymatrix{
 x\ar[rr]^{\phi_\a} \ar[d]_{f}  && \a^* x   \ar[d]^{\a^*(f)}
 \\
y \ar[rr]_{\psi_\a} && \a^*y
 } 
 $$
 commutes for all $\a \in \G$.
 
 We will classify equivariant structures on an $x \in {\sf Ob}(\sC)$ up to isomorphism in the special case when $\Aut(x)$ is abelian.

 \begin{lemma}
Let $x \in {\sf Ob}(\sC)$. Suppose that $\{\phi_\a:x \to \a^* x \; | \;  \a \in \G\}$ and
 $\{\psi_\a:x \to \a^* x \;  | \; \a \in \G\}$ are equivariant structures on $x$.  If $\Aut(x)$ is abelian, then 
 the function $f:\G \to \Aut(x)$, $f(\a)  := \psi^{-1}_\a \phi_\a$, is a 1-cocycle.
 \end{lemma}
 \begin{proof}
By definition,
 \begin{equation}
 \label{compare.equiv}
 (df)(\a,\b)= ( \a \cdot \psi^{-1}_\b \phi_\b) \circ  \big(\psi^{-1}_{\a\b} \phi_{\a\b}\big)^{-1} \circ  \psi^{-1}_\a \phi_\a.
\end{equation}
 Because the $\phi$'s and $\psi$'s define equivariant structures, 
\begin{align*}
 \psi^{-1}_{\a\b} \phi_{\a\b} & \; = \; \Big(t_{\a,\b} \a^*(\psi_\b) \psi_\a \Big)^{-1} \circ \Big(t_{\a,\b} \a^*(\phi_\b) \phi_\a\Big)
 \\
 & \; = \;  \psi_\a^{-1}  \a^*(\psi_\b^{-1}\phi_\b) \phi_\a
 \end{align*}
 Therefore 
\begin{align*}
(df)(\a,\b) & \; = \;  \phi_\a^{-1}  \a^*(\psi_\b^{-1}\phi_\b)\phi_\a  \circ  \big(  \psi_\a^{-1}  \a^*(\psi_\b^{-1}\phi_\b) \phi_\a \big)^{-1} \circ  \psi^{-1}_\a \phi_\a
\\
& \; = \;   \id_x.
 \end{align*}
 Thus, $f$ is a 1-cocycle as claimed.
\end{proof}

 Let $x \in {\sf{Ob}}(x)$. We write $\Phi(x)$ for the set of equivariant structures on $x$ and $\Phi(x)_{\sf Isom}$ for the set of isomorphism
 classes of equivariant structures on $x$. If $\phi=\{\phi_\a:x \to \a^*x \; | \; \a \in \G\} \in \Phi(x)$ we write $[\phi]$ for the isomorphism class of
 $\phi$; i.e., $\phi \mapsto [\phi]$ denotes the obvious function $\Phi(x) \to \Phi(x)_{\sf Isom}$.

\begin{proposition}
Let $x \in {\sf Ob}(\sC)$ and suppose $\Aut(x)$ is abelian. If $\phi=\{\phi_\a:x \to \a^*x \; | \; \a \in \G\}$ is an equivariant structure on $x$ and 
$f:\G \to \Aut(x)$, $\a \mapsto f_\a$, a 1-cocycle, then 
$$
(f \cdot \phi) \; := \; \{\phi_\a f_\a:x \to \a^* x \; | \;  \a \in \G\}
$$ 
is an equivariant structure on $x$ that depends only on the class of $f$ in $H^1(\G,\Aut(x))$.  This gives an action of 
$H^1(\G,\Aut(x))$ on $\Phi(x)_{\sf Isom}$. Furthermore, if $\Phi(x) \ne \varnothing$, then  
$H^1(\G,\Aut(x))$ acts simply transitively on $\Phi(x)_{\sf Isom}$.
 \end{proposition}
 \begin{proof}
 Let $[f] \in H^1(\G,\Aut(x))$ where $f$ is a 1-cocycle. Let $\phi=\{\phi_\a\} \in \Phi(x)$. 
 Because $f$ is a 1-cocycle, $(\a\cdot f_\b)f_{\a\b}^{-1}f_\a = \id_x$. Because $\Aut(x)$ is abelian this equality can be rewritten as
 $$
 f_{\a\b} \; = \;  (\a\cdot f_\b) f_\a   \; = \;  \phi_\a^{-1} \a^*(f_\b)\phi_\a f_\a .
 $$
Since the $\phi_\a$'s form an equivariant structure on $x$, 
 $$
  \phi_{\a\b}  = (t_{\a,\b})_x \a^*(\phi_\b ) \phi_\a 
 $$
 for all $\a,\b \in \G$.  Therefore 
 $$
  \phi_{\a\b} f_{\a\b} =\Big( (t_{\a,\b})_x \a^*(\phi_\b ) \phi_\a \Big) \circ \Big( \phi_\a^{-1} \a^*(f_\b)\phi_\a f_\a \Big) =  (t_{\a,\b})_x \a^*(\phi_\b  f_\b)\phi_\a f_\a.
  $$
In other words, the diagram
  $$
\xymatrix{
 x\ar[rr]^{\phi_\a f_\a} \ar[d]_{\phi_{\a\b} f_{\a\b}} && \a^* x   \ar[d]^{\a^*(\phi_\b f_\b)}
 \\
 (\a\b)^*x && \ar[ll]^{(t_{\a,\b})_x}  \a^*(\b^*x)
 } 
 $$
 commutes; i.e.,  the maps $\{\phi_\a f_\a :x \to \a^*x \; | \; \a \in \G\}$ form an equivariant structure on $x$. 
 
 We now show that the isomorphism class of $(x,f\cdot \phi)$ depends only on the cohomology class of $f$. 
 Let  $f,f':\G \to \Aut(x)$ be 1-cocycles. They are
cohomologous if and only if $f'f^{-1}=dg$ for some $g \in C^0(\G,\Aut(x))=\Aut(x)$, i.e., if  and only if there is $g \in \Aut(x)$ such that 
$$
f'_\a f^{-1}_\a=(dg)(\a) = (\a\cdot g) g^{-1}
 $$
 for all $\a \in \G$. On the other hand, $(x,f\cdot \phi) \cong (x,f'\cdot \phi)$ if and only if there is an isomorphism $g:x \to x$ such that 
 the diagram
$$
\xymatrix{
 x\ar[rr]^{\phi_\a f_\a} \ar[d]_{g}  && \a^* x   \ar[d]^{\a^*(g)}
 \\
x \ar[rr]_{\phi_\a f'_\a} && \a^*x
 } 
 $$
 commutes for all $\a \in \G$; i.e., if and only if $\a^*(g)\phi_\a f_\a = \phi_\a f'_\a g$ or, equivalently, $\phi_\a^{-1} \a^*(g)\phi_\a f_\a = f'_\a g$
 for all $\a \in \G$. Since $\Aut(x)$ is abelian, this is equivalent to the condition that 
 $\phi_\a^{-1} \a^*(g)\phi_\a g^{-1}= f'_\a  f_\a ^{-1}$ for all $\a \in \G$,  i.e., $(\a\cdot g)  g^{-1}= f'_\a  f_\a ^{-1}$. This completes the proof that
 $(x,f\cdot \phi) \cong (x,f'\cdot \phi)$ if and only if $[f]=[f']$. Thus, once we have show that $([f],[\phi]) \mapsto [f\cdot \phi]$, really is an action,
 as we do in the next paragraph, we will have shown that $H^1(\G,\Aut(x))$ acts on $\Phi(x)_{\sf Isom}$ and all isotropy groups are trivial.

We now check that $([f],\phi) \mapsto (f\cdot \phi)$ is an action of $H^1(\G,\Aut(x))$ on $\Phi(x)$. Let $f,f':\G \to \Aut(x)$ be 1-cocycles. Then $f\cdot (f' \cdot \phi)= \{\phi_\a f'_\a f_\a \; | \; \a \in \G\}$. Since $f_\a$ and 
$f_\a'$ are elements in the abelian group $\Aut(x)$, $f'_\a f_\a =f_\a f'_\a$, from which it follows that $f\cdot (f' \cdot \phi) = (ff') \cdot \phi$.

It remains to show that $H^1(\G,\Aut(x))$ acts transitively on $\Phi(x)_{\sf Isom}$ is transitive. Let $\phi,\phi' \in \Phi(x)$. We will show there is 
a 1-cocycle $f$ such that $\phi' \cong f\cdot \phi$.  
By Lemma \ref{lem.1.cocycle} below, the function $f:\G \to \Aut(x)$ defined by $f(\a):=\phi_\a^{-1}\phi'_\a$ is a 1-cocycle. 
But $(f \cdot \phi)_\a = \phi_\a f_\a = \phi_\a'$ so $\phi'=f\cdot \phi$. 
 \end{proof}
 
\begin{lemma}
\label{lem.1.cocycle}
Let $x \in {\sf Ob}(\sC)$ and suppose that $\Aut(x)$ is abelian
If $\phi,\psi \in \Phi(x)$, then $\phi^{-1}\psi:=\{\phi_\a^{-1}\psi_\a \; | \; \a \in \G\}$ is a 1-cocycle for $\G$ with values in $\Aut(x)$.
\end{lemma}
\begin{proof}
We must show that $d(\phi^{-1}\psi)(\a,\b)$ is the identity for all $\a,\b \in \G$. This is the case because
\begin{align*}
d(\phi^{-1}\psi)(\a,\b) & \; = \; \a\cdot(\phi^{-1}_\b\psi_\b) \circ (\phi^{-1}_{\a\b}\psi_{\a\b})^{-1} \circ \phi^{-1}_\a\psi_\a
\\
& \; = \;  \a\cdot(\phi^{-1}_\b\psi_\b) \circ \phi^{-1}_\a\psi_\a \circ (\phi^{-1}_{\a\b}\psi_{\a\b})^{-1}
\\
& \; = \;  \phi_\a^{-1} \a^*(\phi^{-1}_\b\psi_\b)\phi_\a \circ \phi^{-1}_\a\psi_\a \circ \psi_{\a\b}^{-1} \phi_{\a\b}
\\
& \; = \;  \phi_\a^{-1} \a^*(\phi^{-1}_\b\psi_\b) \psi_\a \circ  \psi_\a^{-1} \a^*(\psi_\b)^{-1} (t_{\a,\b})_x^{-1} \circ (t_{\a,\b})_x\a^*(\phi_\b)   \phi_\a
\\
& \; = \;  \phi_\a^{-1} \a^*(\phi^{-1}_\b\psi_\b)  \a^*(\psi_\b)^{-1}  \a^*(\phi_\b)   \phi_\a
\end{align*}
which is certainly equal to $\id_x$. 
\end{proof}

\subsection{Equivariant modules}

Let $\G$ act as $k$-algebra automorphisms of a $k$-algebra $R$. If $\a \in \G$ and $M$ is a left $R$-module we define $\a^*M$ to be $M$ as a
$k$-vector space with a new action of $R$, namely $x\cdot_\a m:= \a^{-1}(x)m$. If $f:M \to N$ is an $R$-module homomorphism we 
define $\a^*(f):\a^*M \to \a^* N$ to be the function $f$, now viewed as a homomorphism from $\a^*M$ to $\a^*N$.
In this way, $\a^*$ becomes an auto-equivalence of the category of left $R$-modules, $\Mod(R)$. 
Since $\a^*\b^* = (\a\b)^*$ this gives an action of $\G$ on $\Mod(R)$. 

Suppose $M$ is a $\G$-equivariant left $R$-module via the isomorphisms $\phi_\a:M \to \a^*M$, $\a \in \G$. Since $\a^*M=M$, each $\phi_\a$ is
a $k$-linear map $\phi_\a:M \to M$ and it has the property that $\phi_\a(xm)=x\cdot_\a \phi_\a(m)=\a^{-1}(x)\phi_\a(m)$ or, equivalently, $\phi_\a^{-1}(xm)=\a(x)\phi_\a^{-1}(m)$, for all $x \in R$ and $m \in M$. If we write $m^\a:=\phi_\a^{-1}(m)$, 
then we obtain a left action of $\G$ on $M$ with the property that $(xm)^\a=\a(x)m^\a$ for all $x \in R$, $\a \in \G$, and $m \in M$.  

Conversely, if $M$ is a left $R$-module with a left action of $\G$ on $M$ such that $(xm)^\a=\a(x)m^\a$ for all $x \in R$, $\a \in \G$, and $m \in M$,
then the maps $\phi_\a:M \to \a^*M$ defined by $\phi_\a(m)=m^{\a^{-1}}$ gives $M$ the structure of a $\G$-equivariant $R$-module. 

Thus, a $\G$-equivariant $R$-module is an $R$-module, $M$ say, together with an action of $\G$ via a group homomorphism $\G \to \Aut_\ZZ(M)$, $\a \mapsto (m \mapsto m^\a)$, such that $(xm)^\a=\a(x)m^\a$ for all $\a \in \G$ and $m \in M$.

\bibliography{biblio}{}
\bibliographystyle{plain}

\end{document}